\documentclass[a4paper,reqno,12pt]{amsart}

\usepackage{graphicx} 
\usepackage{comment}
\usepackage{amsmath}
\usepackage{amssymb}
\usepackage{amsthm}
\usepackage[noadjust]{cite}
\usepackage{mathrsfs}
\usepackage[all,cmtip]{xy}
\usepackage{cite}
\usepackage{bm}
\usepackage{cases}
\usepackage{enumitem}
\usepackage{indentfirst}
\usepackage{geometry}
\usepackage{xy}
\usepackage{color}
\allowdisplaybreaks[4]
\usepackage[colorlinks=true, citecolor=red, linkcolor=blue, filecolor=magenta, urlcolor=green,backref=page]{hyperref}

\newtheorem{thm}{Theorem}[section]
\newtheorem{prop}[thm]{Proposition}

\newtheorem{lem}[thm]{Lemma}

\newtheorem{cor}[thm]{Corollary}

\theoremstyle{definition}
\newtheorem{definition}[thm]{Definition}
\newtheorem{ques}[thm]{Question}

\newtheorem{remark}[thm]{Remark}
\numberwithin{equation}{section}

\newcommand{\bN}{\mathbb{N}}
\newcommand{\bP}{\mathbb{P}}
\newcommand{\bQ}{\mathbb{Q}}
\newcommand{\bR}{\mathbb{R}}

\newcommand\cA{\mathcal{A}}
\newcommand\cB{\mathcal{B}}

\newcommand\cD{\mathcal{D}}
\newcommand\cE{\mathcal{E}}
\newcommand\cF{{\mathcal{F}}}
\newcommand\cG{{\mathcal{G}}}
\newcommand\cH{{\mathcal{H}}}
\newcommand\cI{{\mathcal{I}}}
\newcommand\cJ{{\mathcal{J}}}

\newcommand\cL{{\mathcal{L}}}
\newcommand\cM{{\mathcal{M}}}
\newcommand\cO{{\mathcal{O}}}
\newcommand\cP{\mathcal{P}}

\newcommand\cR{{\mathcal{R}}}
\newcommand\cS{{\mathcal{S}}}
\newcommand\cT{{\mathcal{T}}}
\newcommand\cU{\mathcal{U}}
\newcommand\cW{{\mathcal{W}}}
\newcommand\cV{{\mathcal{V}}}
\newcommand\cX{{\mathcal{X}}}
\newcommand\cY{{\mathcal{Y}}}
\newcommand\cZ{{\mathcal{Z}}}

\newcommand\M{\mathbf{M}}
\newcommand\N{\mathbf{N}}
\newcommand\R{\mathbf{R}}

\newcommand\h{\mathrm{h}}
\newcommand\vv{\mathrm{v}}

\DeclareMathOperator{\Aut}{Aut}
\newcommand{\mult}{\operatorname{mult}}

\newcommand{\coeff}{\operatorname{coeff}}
\newcommand{\Supp}{\operatorname{Supp}}
\newcommand{\lct}{\operatorname{lct}}

\newcommand{\vol}{\operatorname{vol}}
\newcommand{\Ivol}{\operatorname{Ivol}}
\newcommand{\Div}{\operatorname{Div}}

\newcommand{\red}{\operatorname{red}}
\newcommand{\Pic}{\operatorname{Pic}}

\title[Polarized Log Calabi--Yau Fibrations with Bounded Bases]{Boundedness of Polarized Log Calabi--Yau Fibrations with Bounded Bases}
\author{Xiaowei Jiang, Junpeng Jiao,  Minzhe Zhu }

	\date{July 13, 2026}

\address{Shanghai Center for Mathematical Sciences, Fudan University, Shanghai 200438, China}
	\email{xw\_jiang@fudan.edu.cn}
    \address{Yau Mathematical Sciences Center, Tsinghua University, Haidian District, Beijing, 100084, China}
	\email{jiao\_jp@mail.tsinghua.edu.cn}
    \address{School of Mathematics, Korea Institute for Advanced Study, Seoul,
02455, Korea}
\email{zhumz@kias.re.kr}

\begin{document}
\begin{abstract}
We investigate the boundedness problem for log Calabi--Yau fibrations whose bases and general fibers are bounded. We prove that the total spaces of log Calabi--Yau fibrations are bounded in codimension one after fixing some natural invariants, which confirms a conjecture of Birkar and Hacon. We also prove that the total spaces are bounded if, in addition, the irregularity of the general fibers vanishes. Then we apply our results to the boundedness problem for stable minimal models and fibered Calabi--Yau varieties.
\end{abstract}
\maketitle
\tableofcontents
\section{Introduction}

Throughout this paper,  we  work over an algebraically closed field $k$ of characteristic zero. By \textit{integral divisor}, we mean a $\mathbb{Q}$-Cartier Weil divisor with integer coefficients that is not necessarily Cartier and need not be effective.
\vspace*{10pt}

According to the minimal model program conjecture and the abundance conjecture, every projective variety $Y$ is birational to a projective variety $X$ with mild singularities such that either $X$ is canonically polarized, or $X$ admits a Mori--Fano fibration $X \to Z$, or $X$ admits a Calabi--Yau fibration $X \to Z$.
For this reason, canonically polarized varieties, Fano varieties, Calabi--Yau varieties, and their fibrations play a central role in birational geometry. From the perspective of constructing a moduli space for a given class of varieties, the first step is to determine whether they form a bounded family. For the definition of boundedness for varieties, see \S\ref{sec:bdd family}.

The boundedness of canonically polarized varieties is established in \cite{haconACCLogCanonical2014,haconBoundednessModuliVarieties2018}, and the boundedness of Fano varieties with mild singularities, known as the famous BAB conjecture, is proved by Birkar \cite{birkarAntipluricanonicalSystemsFano2019, birkarSingularitiesLinearSystems2021}. However, for Calabi--Yau varieties, due to the lack of a natural polarization, the question of boundedness remains widely open even in dimension three for strict Calabi--Yau manifolds.
Nevertheless, for polarized Calabi--Yau varieties, Birkar shows that boundedness holds under certain conditions \cite{birkarGeometryPolarisedVarieties2023}: either one allows a non-effective polarization while requiring the underlying variety to be klt, or, if the underlying variety is slc, the polarization must be an effective divisor that does not contain the non-klt center of the variety.

Based on the predictions of the minimal model program and the abundance conjecture, it is important to extend boundedness results to Fano fibrations and Calabi--Yau fibrations. Such fibrations also frequently appear in inductive arguments.
In \cite{jiangBirationalBoundednessFano2018}, Jiang considered the birational boundedness of Fano fibrations under several conjectural assumptions. 
Later, Birkar used some of these arguments to obtain the birational boundedness of Fano fibrations and carried out further work to establish boundedness \cite{birkarBoundednessFanoType2024}.
However, the boundedness of Calabi--Yau fibrations is not fully understood, although some literature addresses this direction \cite{filipazziInvariancePlurigeneraBoundedness2020,birkarBoundednessVolumeGeneralised2021,birkarModuliAlgebraicVarieties2022,jiangBoundednessModuliTraditional2023,birkarBoundednessEllipticCalabiYau2024,filipazziBoundednessNfoldsKappa2024,filipazziBoundednessEllipticCalabi2025,hashizumeBoundednessModuliSpaces2023,jiaoBoundednessPolarizedCalabiYau2022,zhuBoundednessStableMinimal2025}.

A \textit{Calabi--Yau fibration} $f\colon X\to Z$ consists of a projective variety $X$  and a contraction
$f\colon X \to Z$ such that
$K_X \sim_{\bQ,Z} 0$.
In this paper, we investigate the following guiding question: under what conditions does the total space of a Calabi--Yau fibration
$f\colon X\to Z$ belong to a bounded family?
We answer this question with the following result, which is a special case of
Theorem~\ref{thm: finite coeff bdd in codim 1} and Theorem~\ref{mainthm2}.

\begin{thm}\label{thm:toy}
Let $d\in \bN$ and $v,r,\epsilon\in\bR^{>0}$. 
Consider projective normal varieties $X$ such that
\begin{enumerate}
    \item $X$ is $\epsilon$-lc of dimension $d$,
    \item $f:X\to Z$ is a fibration such  that $K_X\sim_{\bQ,Z}0$,
    \item $A$ is an integral divisor on $X$ that is ample over $Z$, and $\vol(A|_F)\leq v$, where $F$ is the general fiber of $f:X\to Z$,  
    \item $H\geq 0$ is a very ample divisor on $Z$ such that $H^{\dim Z}\leq r$, and
    \item $f^*H-K_X$ is pseudo-effective.
\end{enumerate}
Then the set of such $X$ is bounded in codimension one.
If further,
\begin{enumerate}
    \setcounter{enumi}{5}
    \item $\Supp R^1f_*\cO_X\subsetneq Z$,
\end{enumerate}
then the set of such $X$ forms a bounded family.
\end{thm}

Note that condition~(4) implies that the base $Z$ belongs to a bounded family, and condition~(3) implies that the general fiber $F$ belongs to a bounded family; see \cite[Corollary~1.6]{birkarGeometryPolarisedVarieties2023}.
However, even if both the base and the general fiber belong to bounded families, the total space need not be bounded in general.

For example, by \cite[Example~3.1]{filipazziBoundednessNfoldsKappa2024},  there exists a set of smooth elliptic surfaces
$f_n\colon S_n \to C$ over a curve $C$ with $g(C)\geq 2$,
such that each $f_n$ is smooth and isotrivial. 
It then follows from the canonical bundle formula that $K_{S_n} \sim_{\bQ} f_n^* K_C$. Hence condition~(5) is satisfied in this setting since we may choose a very ample divisor $H$ on $C$ such that $H-K_C$ is pseudo-effective.
However, the collection $\{S_n\}$ is not bounded, as $S_n$ does not admit a multisection of degree smaller than $n$.
This example illustrates that condition~(3) is not merely used to bound the general fibers.

On the other hand, by \cite[Corollaries~2.4 and~2.5]{mirandaModuliWeierstrassFibrations1981}, there exists a set of minimal elliptic surfaces
$g_n\colon X_n \to \bP^1$ with sections $C$ satisfying $C^2=-n$ for $n\in \bN$,
such that $K_{X_n} \sim_{\bQ} (n-2)F$, where $F$ denotes a fiber of $g_n$. In this case, condition~(3) is satisfied, yet the total spaces $X_n$ are still not bounded,
since the ``degree'' of $K_{X_n}$ cannot be controlled by a very ample divisor $H$ on $Z$
with bounded $H^{\dim Z}$.
This demonstrates the necessity of condition~(5).

Condition~(6) means that the irregularity of the general fiber $F$ vanishes,
and it is natural to ask whether this assumption can be removed in order to obtain full boundedness.

\vspace{0.5cm}
{\textbf{\sffamily{Polarized log Calabi--Yau fibration.}}} 
We now turn to the preparation for more general statements of our results. Inspired by the study of Fano type fibrations (see Definition \ref{def:FT fib} and Theorem \ref{FTF}) in \cite{jiangBirationalBoundednessFano2018,birkarBoundednessFanoType2024}, we introduce a special structure for log Calabi--Yau fibrations equipped with polarizations on both the base and the general fiber.

\begin{definition}
    A \textit{polarized log Calabi--Yau fibration} (resp. \textit{weak polarized log Calabi--Yau fibration}) $f:((X,B),A)\to (Z,H)$ consists of 
\begin{enumerate}
   
    \item a projective  pair $(X,B)$,
    \item a fibration $f:X\to Z$ such that $K_X+B\sim_{\bR}f^*N$ for some $\bR$-divisor $N$ on $Z$, 
    \item an integral divisor $A$ on $X$ that is ample over $Z$, and
    \item a very ample divisor $H\geq 0 $ on $Z$ such that $H-N$ is ample (resp. pseudo-effective).
 
\end{enumerate}

\end{definition}
Note that $H$ is a polarization on the base $Z$, and $A|_F$ is a polarization on the general fiber $F$ of $f\colon X \to Z$.
 If $f$ is constant, that is, if $Z$ is a point, then $K_X+B\sim_{\mathbb{Q}} 0$ and $A$ is an ample integral divisor on $X$. In this case, we call $((X,B),A)$ a \textit{polarized log Calabi--Yau pair}.

   We now fix some invariants of a (weak) polarized log Calabi--Yau  fibration.
\begin{definition}
Let $d\in \bN$, $v,r,\epsilon\in\mathbb{R}^{>0}$, and $\Phi\subset[0,1]\cap \bR$ be a DCC set.
\begin{enumerate}
    \item A \textit{(weak) $(d,r,\epsilon)$-polarized log Calabi--Yau  fibration} is a (weak) polarized log Calabi--Yau  fibration $f:((X,B),A) \to (Z,H)$ satisfying
    \begin{itemize}
        \item $(X,B)$ is a projective $\epsilon$-lc pair of dimension $d$, and
        \item $H^{\dim Z} \leq r$.
    \end{itemize}

    \item If, additionally, 
    \begin{itemize}
        \item $\vol(A|_F) \leq v$, where $F$ is a general fiber of $f:X \to Z$,
    \end{itemize}
    then we call $f:((X,B),A) \to (Z,H)$ a \textit{(weak) $(d,v,r,\epsilon)$-polarized log Calabi--Yau  fibration}.

    \item Furthermore, if 
    \begin{itemize}
        \item the coefficients of $B$ belong to $\Phi$,
    \end{itemize}
    then we refer to $f:((X,B),A) \to (Z,H)$ as a \textit{(weak) $(d,\Phi,v,r,\epsilon)$-polarized log Calabi--Yau  fibration}.
\end{enumerate}

\end{definition}

\vspace{0.5cm}
{\textbf{\sffamily{Boundedness of polarized log Calabi--Yau fibration.}}} 
Birkar's work on Fano type fibrations \cite{birkarBoundednessFanoType2024}
can be viewed as a relative version of the BAB theorem
\cite{birkarSingularitiesLinearSystems2021}.
In a similar spirit, our results may be regarded as a relative analogue of
Birkar's boundedness theorem for polarized log Calabi--Yau pairs
\cite{birkarGeometryPolarisedVarieties2023}.

One of our main boundedness statements in codimension one for weak polarized
log Calabi--Yau fibrations concerns the case where the coefficients of $B$
belong to a finite set $\Phi$.
In particular, this includes the special case $B=0$, that is, $\Phi=\{0\}$,
which was first conjectured by Birkar and Hacon.

\begin{thm}\label{thm: finite coeff bdd in codim 1}
 Let $d\in \bN$, $v,r,\epsilon\in\mathbb{Q}^{>0}$, and $\Phi\subset[0,1]\cap \bQ$ be a finite set. 
Consider the set of all weak $(d,\Phi,v,r,\epsilon)$-polarized log Calabi--Yau  fibrations $f:((X,B),A)\to (Z,H)$. Then the set of such $(X,B+f^*H)$ is log bounded in codimension one.   
\end{thm}

For a more general version of this result, see Theorem \ref{thm:bdd of weak pcy fibration}. By combining Theorem \ref{thm: finite coeff bdd in codim 1} with the technique from \cite{birkarSingularitiesFanoFibrations2023}, we  establish the boundedness in codimension one for polarized log Calabi--Yau  fibrations with arbitrary real coefficients for $B$. However, we need to assume the ampleness of $H - N$ due to a technical reason.

\begin{thm}\label{mainthm1}
    Let $d \in\bN$ and $v,r,\epsilon,\delta\in \bR^{>0}$. Consider the set of all $(d, v,r,\epsilon)$-polarized log Calabi--Yau  fibrations $((X,B),A)\to (Z,H)$ and $\bR$-divisors $0\leq \Delta\leq B$ where the non-zero coefficients of $\Delta$ are greater than $\delta$. 
    Then the set of such $(X,\Delta+f^*H)$ is log bounded in codimension one.
\end{thm}
\begin{ques}\label{question}
With the same notation as Theorem \ref{thm: finite coeff bdd in codim 1}, is the set of such pairs $(X, B +  f^* H)$ log bounded?
\end{ques}
When $\dim Z=1$,  boundedness is studied in \cite{hashizumeBoundednessModuliSpaces2023}.  For  $\bQ$-factorial terminal minimal threefolds of Kodaira dimension two, \cite{filipazziBoundednessEllipticCalabi2025} derives boundedness from boundedness in codimension one by studying the Kawamata--Morrison cone conjecture and the liftability of flops.
In this paper, we propose an alternative approach.
Under the additional condition that $\Supp R^1f_*\cO_X\subsetneq Z$, we  obtain the actual boundedness of $(d,\Phi,v,r,\epsilon)$-polarized log Calabi--Yau fibrations.  
\begin{thm}\label{mainthm2}
Let $d \in\bN$,  $v,r,\epsilon\in \bQ^{> 0}$, and $\Phi\subset[0,1]\cap \bQ$ be a finite set. Consider the set of all weak $(d,\Phi,v,r,\epsilon)$-polarized log Calabi--Yau  fibrations $f:((X,B),A)\to (Z,H)$ such that  $\Supp R^1f_*\cO_X\subsetneq Z$. Then the set of such $(X,B+f^*H)$ is log bounded.
\end{thm}

\vspace{0.5cm}
{\textbf{\sffamily{Boundedness of stable minimal models and fibered Calabi--Yau varieties.}}} 
We now apply these general boundedness results to some special cases of  polarized log Calabi--Yau fibrations. First, we consider the case where $K_X + B$ is semi-ample.
It turns out that under some natural conditions, we can choose $H = lN$ for some bounded positive integer $l >1$. Then $H - N$ is automatically ample.
In this case, $((X, B), A)$ is a so-called \textit{stable minimal model} \cite{birkarModuliAlgebraicVarieties2022, jiangBoundednessModuliTraditional2023,zhuBoundednessStableMinimal2025}.
Jiao \cite{jiaoBoundednessPolarizedCalabiYau2022} shows that $(X,B)$ is crepant birationally bounded. When $\dim F=1$, the log boundedness in codimension one of $(X,B)$ follows from \cite{filipazziBoundednessNfoldsKappa2024}.

\begin{cor}\label{boundedness of klt stable minimal models}
Let $d\in \bN$, $u,v\in \bQ^{>0}$, and $\Phi\subset[0,1]\cap \bQ$ be a DCC set. Consider the set of $((X,B),A)$ such that
\begin{enumerate}
    \item $(X,B)$ is a projective klt pair of dimension $d$,
    \item the coefficients of $B$ are in $\Phi$,
    \item $K_X+B$ is semi-ample defining a contraction $f: (X,B)\to Z$,
    \item $\Ivol(K_X+B)=u$,
    \item  $A$ is  an integral divisor on  $X$ that is ample over $Z$, and $\vol(A|_F)\leq v$, where $F$ is the general fiber of $f: X\to Z$.
\end{enumerate}
 Then the set of such $(X,B)$ is log bounded in codimension one. 
 Moreover, if 
 $\Supp R^1f_*\cO_X\subsetneq Z$, then the set of such $(X,B)$ forms a log bounded family.
\end{cor}

Next we consider another important case of polarized log Calabi--Yau fibration, where the total space is a Calabi--Yau variety. Such a fibration is called a \textit{fibered Calabi--Yau variety}. In this case, we have $N \sim_{\bQ} 0$, so $H - N$ is automatically ample. Furthermore, we assume that the base $Z$ is rationally connected. Note that if $X$ is a strict Calabi--Yau manifold, then by \cite[Corollary 5.1]{birkarBoundednessEllipticCalabiYau2024}, this condition is automatically satisfied.
\begin{cor}\label{cor: bdd of fibered cy}
   Let $d\in\bN$ and $\epsilon,v\in\bR^{>0}$. Assume that
\begin{enumerate}
\item $(X,B)$ is a projective $\epsilon$-lc pair of dimension $d$,
\item $K_X+B\sim_{\bR}0$,
\item $f: X \to Z$ is a contraction to a rationally connected variety $Z$, and
\item $A$ is an integral divisor on $X$ such that $0<\vol(A|_F) \leq v$, where $F$ is the general fiber of $f:X\to Z$.
\end{enumerate}
Then the set of such $X$ is bounded in codimension one.
\end{cor}
By \cite{birkarSingularitiesFanoFibrations2023}, we do not need to assume the boundedness of the torsion index of $K_X+B$.
If $\dim F = 1$, this case is studied in \cite[Theorem 1.4]{birkarBoundednessEllipticCalabiYau2024}. 
The case where $X$ has terminal singularities is treated in \cite[Theorem 8.2]{jiaoBoundednessPolarizedCalabiYau2022}.

\vspace{0.5cm}
{\textbf{\sffamily{Sketch of proof.}}} 
 We sketch the proofs of our main theorems, starting with Theorem \ref{thm: finite coeff bdd in codim 1}.
Given a $(d,\Phi,v,r,\epsilon)$-polarized log Calabi--Yau fibration $f:((X,B),A) \to (Z,H)$, note that the base $Z$ is bounded by assumption and the general fiber $((F,B_F),A_F)$ is bounded by \cite[Corollary 1.6]{birkarGeometryPolarisedVarieties2023}. We study the induced rational map from $Z$ to a ``moduli space'' of the general fibers. For this purpose, we use the strongly embedded fine moduli space $\cS$ of polarized log Calabi--Yau pairs constructed in \cite{birkarModuliAlgebraicVarieties2022,birkarGeometryPolarisedVarieties2023}. Since $\cS$ also parametrizes the polarizations, the universal family $(\cX,\cB)\to \cS$ is not necessarily of maximal variation, and the rational map $Z\dashrightarrow \cS$ is not necessarily bounded. Applying \cite{ambroModuliBdivisorLctrivial2005}, we obtain a new family $(\cX^!,\cB^!)\to \cS^!$ of maximal variation with $\mathbf{b}$-nef and big moduli part $\bm{\cM}^!$. Therefore, since the moduli part $\M_Z$ of $f:(X,B)\to Z$ is controlled by $H$, a volume argument shows that, up to a generically finite cover, the map $Z \dashrightarrow \cS^!$ is bounded; see Theorem~\ref{moduli map is bounded}.

The traditional strategy for proving boundedness of polarized fibrations is to modify the vertical part of the polarization $A$ so as to obtain a global ample divisor on $X$ with bounded volume; see \cite{birkarModuliAlgebraicVarieties2022,jiangBoundednessModuliTraditional2023,birkarBoundednessFanoType2024,zhuBoundednessStableMinimal2025,hashizumeBoundednessModuliSpaces2023}. However, this approach fails in our setting.
Instead, we construct a new polarization $L$ arising from the family $(\cX^!,\cB^!) \to \cS^!$ such that $L \equiv mA$ over the generic point of $Z$ for some fixed $m\in \bN$. More precisely,
we define a polarization $\cL$ on $(\cX,\cB) \to \cS$ by pulling back $\cA^!$ to a Galois cover of $\cX$, taking the Galois sum, and then descending it to $\cX$. For every $s \in \cS$, we have $\cL_s \equiv m\cA_s$; see Theorem~\ref{diagram of the universal Calabi--Yau fibration}(6). Finally, we define $L$ as the closure of the pullback of $\cL$ via the moduli map $Z \dashrightarrow \cS$.
Since $L$ arises from the fixed family $((\cX^!,\cB^!), \cA^!) \to \cS^!$ and the map $Z \dashrightarrow \cS^!$ is bounded up to a generically finite cover, by the invariance of plurigenera and an argument about the descent of volume from a generically finite cover, we can show that, after modifying the vertical part of $((X,B), L) \to Z$, the volume of $L$ can be controlled on a suitable birational model; 
see Lemma~\ref{volume is bounded up to a finite cover} and Theorem~\ref{volume is bounded}.

To proceed, we apply weak semistable reduction \cite{abramovichWeakSemistableReduction2000} and the minimal model program for lc pairs \cite{haconExistenceLogCanonical2013} in Theorem \ref{special model of a polarized log Calabi--Yau fibration} to obtain a new birational model $((X',\Delta'),L')\to Z$ of $((X,B),L)\to Z$ satisfying that
\begin{enumerate}
    \item $(X',\Delta'+\alpha L')$ is lc for some fixed positive real number $\alpha$, 
    \item $K_{X'}+\Delta'+\alpha L'$ is big, and
    \item $\vol(K_{X'} + \Delta' + \alpha L')$ is bounded from above.
\end{enumerate}
Then, we apply \cite{haconBirationalAutomorphismsVarieties2013,haconACCLogCanonical2014} to obtain the log birational boundedness of $(X',\Delta')$. 
Moreover, we can deduce that $\Supp(\Delta')$ contains both the strict transform of $\Supp(B)$ on $X'$ and the exceptional divisors over $X$. We then apply the MMP in family \cite{haconBoundednessModuliVarieties2018} to bound $(X,B)$ in codimension one.

Regarding the proof of Theorem \ref{mainthm1}, we consider two cases. If the horizontal part $B^\h$ of $B$ is nonzero, then $K_X$ is not pseudo-effective over $Z$. Let $t$ be the pseudo-effective threshold of $A$ with respect to $K_X$ over $Z$. By the argument as in \cite[Theorem 10.1]{birkarSingularitiesFanoFibrations2023}, we can run an MMP on $K_X+tA$ over $Z$ to decompose the new fibration into a Fano type fibration and a lower-dimensional polarized Calabi--Yau fibration. We then proceed by applying the boundedness of Fano type fibrations \cite{birkarBoundednessFanoType2024} and induction. If $B^\h=0$, then $B$ is vertical over $Z$. After reducing to the case where $Z$ is $\bQ$-factorial and modifying $B$ to be a very exceptional divisor over $Z$, we run an MMP on $K_X+B$ over $Z$ by \cite{birkarExistenceLogCanonical2012} to contract all components of $\Supp(B)$. This case then follows directly from Theorem \ref{thm: finite coeff bdd in codim 1} with $\Phi=\{0\}$.

Now we turn to the sketch of the proof of Theorem \ref{mainthm2}. The polarization $L$ constructed in Theorem \ref{thm: finite coeff bdd in codim 1} satisfies $L \equiv mA$ over the generic point of $Z$ for some fixed $m\in \bN$, and one might hope to modify its vertical part to extend this equivalence over all of $Z$. However, \cite[Example 5.1]{xieextensionofnumericallytrivial} shows that this is not possible in general.
Nonetheless, under the assumption that $\Supp R^1f_*\cO_X\subsetneq Z$, we may assume that $L \sim_\bQ mA$ over the generic point of $Z$. Then we construct a new polarization $J$ such that $J \sim_\bQ mA$ over $Z$ and the components of $\Supp(J-L)$ can be controlled uniformly, see Lemma \ref{approximate unbounded polarization via bounded polarization}. The remaining difficulty is that the components of $\Supp(J-L)$ are vertical over $Z$, hence non-big over $Z$, and the coefficients appearing in $J-L$ are uncontrolled. To address this issue, we study the finiteness of log canonical models where the boundary divisors vary in a polytope whose boundary contains non-big divisors, see Lemma \ref{finiteness of canonical models}. Finally, by the relative ampleness of $A$ over $Z$ and a standard argument of running an MMP in a family \cite{haconBoundednessModuliVarieties2018}, we establish the pure boundedness of $(X,B)$ from boundedness in codimension one.

\vspace{0.5cm}
{\textbf{\sffamily{Structure of the paper.}}} 
 This paper is organized as follows. In \S\ref{sec:Preliminaries} we recall some definitions and preliminary results. In \S\ref{sec:Polarized log Calabi--Yau fibrations: finite coefficients}, we prove Theorem \ref{thm: finite coeff bdd in codim 1}, which establishes boundedness for fibrations with finite coefficient sets. In \S\ref{sec:Polarized log Calabi--Yau fibrations: arbitrary coefficients}, we extend the argument to arbitrary coefficients and prove Theorem \ref{mainthm1}. In \S\ref{sec:Fibrations whose general fibers have vanishing irregularity}, we focus on fibrations whose general fibers have vanishing irregularity and prove Theorem \ref{mainthm2}. Finally, in \S\ref{sec:Stable minimal models and fibered Calabi--Yau varieties}, we deduce Corollaries \ref{boundedness of klt stable minimal models} and \ref{cor: bdd of fibered cy} as consequences of our main results.

\vspace{0.5cm}
{\textbf{\sffamily{Acknowledgement.}}} 
The first author carried out this work during his visit to the University of Utah and thanks Christopher D. Hacon for his hospitality. Both the first and second authors thank  Caucher Birkar for his constant support and helpful comments, as well as Christopher D. Hacon for helpful discussions. The third author expresses his gratitude to his Ph.D. advisor Meng Chen for great support and encouragement. He would also like to thank Chen Jiang, Fanjun Meng, and Lingyao Xie for very effective discussions about Section 5.  We would like to thank the anonymous referee for the numerous suggestions and corrections that improved this article considerably.

The second author is supported by National Key Research and Development Program of China (2025YFA1018100). The third author is supported by a KIAS Individual Grant (MG106901) at Korea Institute for Advanced Study.

\section{Preliminaries}\label{sec:Preliminaries}

\subsection{Notations and conventions}
	We collect some notations and conventions used in this paper.
	\begin{enumerate}

\item A projective morphism $f:X\to Z$ between normal quasi-projective varieties is called a \textit{contraction} if $f_*\cO_X=\cO_Z$. In particular, $f$ is surjective with connected fibers.

\item For a fibration  $f:X\to Z$, we use $X_{\eta}$ to denote the generic fiber of $f$ and $X_g$ to denote the general fiber of $f$. For an $\bR$-divisor $B$ on $X$, we write $B_{\eta}:=B|_{X_{\eta}}$ and $B_g:=B|_{X_g}$.
\item  Let $f : X \to Z$ be a morphism between  normal quasi-projective varieties, and let $M$ and $L$ be $\bR$-Cartier $\bR$-divisors on $X$. We say $M\sim_Z L$ (resp. $M \sim_{\bQ,Z}  L$, $M \sim_{\bR,Z}  L$) if there is a Cartier (resp. $\bQ$-Cartier, $\bR$-Cartier )  divisor $N$ on $Z$ such that $M - L \sim f^*N$ (resp.  $M - L \sim_{\bQ} f^*N$,  $M - L \sim_{\bR} f^*N$). 

\item  Let $X$ be a normal quasi-projective variety, and let $M$ be an $\bR$-divisor on $X$. Write $M=\sum m_i M_i$, where $M_i$ are the distinct irreducible components. We define $M_{\ge a}:=\sum_{m_i\ge a} m_i M_i$, $M_{\le a}:=\sum_{m_i\le a} m_i M_i$, $M_{> a}:=\sum_{m_i> a} m_i M_i$, and $M_{< a}:=\sum_{m_i< a} m_i M_i$.

\item Let $f : X \to Z$ be a morphism between normal quasi-projective varieties, and let $D$ be a $\bR$-divisor on $X$. We say $D$ is \textit{horizontal} over $Z$ if the induced map $\Supp D \to Z$ is dominant, otherwise we say $D$ is \textit{vertical} over $Z$. Given an $\bR$-divisor $D$ on $X$, there is a unique decomposition $D=D^\h+D^\vv$ such that 
	\begin{itemize}
		\item$\Supp D^\h$, $\Supp D^\vv$ have no common components,
		\item every component of $\Supp D^\h$ is horizontal over $Z$, and
		\item $D^\vv$ is vertical over $Z$.
	\end{itemize}
	We call $D^\h$ the \textit{horizontal part} of $D$ and $D^\vv$ the \textit{vertical part} of $D$ with respect to $f:X\to Z$.

\item We say that a set  $\Phi \subset\bR$ satisfies the \textit{descending chain condition}  (DCC, for short) if $\Phi$  does not contain any strictly decreasing infinite sequence. Similarly, we say that a set  $\Phi \subset\bR$  satisfies the \textit{ascending chain condition}  (ACC, for short) if $\Phi$ does not contain any strictly increasing infinite sequence.

\item Let $X$ be a normal  projective variety of dimension $d$, and let $D$ be a $\bQ$-divisor on $X$ such that the Iitaka dimension $\kappa(D)$ is non-negative. The \textit{Iitaka volume} of $D$, denoted by $\Ivol(D)$, is defined as  
\[
\Ivol(D)=\limsup_{m\to \infty}\frac{\kappa(D)!h^0(X,\cO_X(\lfloor mD\rfloor))}{m^{\kappa(D)}}.
\]
When $D$ is big, this is also called the \textit{volume} of $D$, denoted by $\vol(D)$. If $D$ is semi-ample and defines a contraction $f : X \to Z$ such that $D \sim_{\bQ} f^*H$ for some ample $\bQ$-divisor $H$ on $Z$, then 
$\Ivol(D) =\vol(H)= H^{\dim Z}$.
\end{enumerate}

\subsection{Pairs and singularities}

A \textit{sub-pair} $(X,B)$ consists of a normal quasi-projective variety $X$ and an $\bR$-divisor $B$ on $X$ such that $K_X+B$ is $\bR$-Cartier. If, in addition, $B \ge 0$, then $(X,B)$ is called a \textit{pair}.

Let $D$ be a prime divisor over $X$. Let $\pi \colon X' \to X$ be a log resolution of $(X,B)$ such that $D$ is a prime divisor on $X'$. We can write
\[
K_{X'} + B' = \pi^*(K_X + B).
\]
The \textit{log discrepancy} of $D$ with respect to $(X,B)$ is defined by
\[
a(D,X,B) := 1 - \mult_D B'.
\]

We say that $(X,B)$ is \textit{sub-klt} (resp. \textit{sub-lc}, \textit{sub-$\epsilon$-lc}) if $a(D,X,B) > 0$ (resp. $a(D,X,B) \ge 0$, $a(D,X,B) \ge \epsilon$) for every prime divisor $D$ over $X$. If $(X,B)$ is a pair, we drop the prefix ``sub'' and say that $(X,B)$ is \textit{klt} (resp. \textit{lc}, \textit{$\epsilon$-lc}).

A \textit{non-klt place} of $(X,B)$ is a prime divisor $D$ over $X$, that is, a prime divisor on some birational model of $X$, such that
$a(D,X,B) \le 0$.
A \textit{non-klt center} of $(X,B)$ is the image of a non-klt place on $X$.

A \textit{log place} of $(X,B)$ is a prime divisor $D$ over $X$ such that
$a(D,X,B) \in [0,1)$.
A \textit{log center} of $(X,B)$ is the image of a log place on $X$.

\subsection{Minimal models}

 Suppose that $f : X \to Z$ and $f ^m : X ^m\to Z$ are projective morphisms,  $\phi:X\dashrightarrow X^m$ is a birational map over $Z$ which does not extract any divisor, and $(X,B)$ and $(X^m,B^m)$ are lc pairs, where $B^m=\phi_*B$. If $a(E,X,B) > a(E,X^m,B^m )$ (resp. $a(E,X,B) \geq a(E,X^m,B^m )$) for all prime $\phi$-exceptional divisors $E \subset X$, $X^ m$ is $\bQ$-factorial, and $K_{X^m}+B^m$ is nef over $Z$, then we say that $\phi:X\dashrightarrow X^m$  is a \textit{minimal model} (resp. \textit{weak log canonical model})  of $(X,B)$ over $Z$. 
 
  A minimal model (resp. weak log canonical model) $\phi:X\dashrightarrow X^m$  of $(X,  B)$ over $Z$ is called a \textit{good minimal model} (resp. \textit{semi-ample model}) if $K_{X^m}+B^m$ is semi-ample over $Z$. In this case,
\[  R(X/Z, K_{X^m}+B^m) := \bigoplus_{l\geq 0} f^m_* \cO_{X^m} (l(K_{X^m}+B^m))\]
is a finitely generated $\cO_Z$-algebra, and let
\[X^c = \operatorname{Proj}  R(X/Z, K_{X^m}+B^m).\] 
  If $K_{X^m}+B^m$ is semi-ample and big over $Z$, then $X^c$ is called the \textit{log canonical model} of $(X,B)$ over $Z$.

\begin{definition}{(\cite[Definition 1.3]{birkarAugmentedBaseLocus2017})}
    Let $f:X\to Z$ be a contraction between two projective varieties, and let $L$ be an $\bR$-Cartier $\bR$-divisor on $X$. The \textit{relative exceptional locus} of $L$ (also called the \textit{relative null locus} when $L$ is nef over $Z$) is defined as \[\mathbb{E}(L/Z)=\bigcup_{L|_V \text{ is not big over $f(V)$}} V,\] where the union runs over the integral subvarieties $V\subseteq X$ with positive dimension.
\end{definition}

\begin{lem}\label{minimal models and canonical models commutes with base change}
Assume that \begin{itemize}
    \item $(X,B)$ is a lc pair and $f:X\to Z$ is a contraction,
    \item $\mu:Z'\to Z$ is a finite cover,
    \item $X'$ is the normalization of $X\times_Z Z'$ and denote the natural finite cover $X'\to X$ by $\pi$, and the contraction $X'\to Z'$ by $f'$,
    \item $(X',B')$ is a lc pair such that $K_{X'}+B'=\pi^*(K_X+B)$, and
    \item $\eta:X''\dashrightarrow X'/Z'$ is an isomorphism in codimension one and $B''$ is the strict transform of $B'$ on $X''$.
\end{itemize}
\[\xymatrix{(X'',B'')\ar@{-->}[r]^{\eta} \ar[dr]&(X',B')\ar[r]^{\pi}\ar[d]^{f'}&(X,B)\ar[d]_{f}\\ &Z'\ar[r]^{\mu}& Z}\]

Then we have the following statements:
\begin{enumerate}
    \item[(1)] If $(X,B)\dashrightarrow (X^m,B^m)$ is a good minimal model of $K_X+B$ over $Z$ and $(X'',B'')\dashrightarrow (X''^m,B''^m)$ is a good minimal model of $K_{X''}+B''$ over $Z'$, then $(X''^m,B''^m)$ is isomorphic in codimension one to the normalization of $(X^m,B^m)\times_Z Z'$,
    \item[(2)] If furthermore $K_X+B$ is big over $Z$, assume that $(X,B)\dashrightarrow (X^c,B^c)$ is the log canonical model of $K_X+B$ over $Z$, and $(X'',B'')\dashrightarrow (X''^c,B''^c)$ is the log canonical model of $K_{X''}+B''$ over $Z'$. Then $(X''^c,B''^c)$ is isomorphic to the normalization of $(X^c,B^c)\times_Z Z'$.
\end{enumerate}

\begin{proof}
(1). By \cite[Lemma 2.4(1)]{haconExistenceLogCanonical2013}, 
the set of exceptional divisors of $X\dashrightarrow X^m$ coincides with the support of $N_\sigma(K_X+B/Z)$, and the set of exceptional divisors of $X''\dashrightarrow X''^m$ coincides with the support of $N_\sigma(K_{X''}+B''/Z')$. Thus it suffices to prove \[\Supp(N_\sigma(K_{X''}+B''/Z'))=\eta^{-1}\pi^{-1}\Supp(N_\sigma(K_X+B/Z)).\]
    
Since $X'\rightarrow X$ is a finite cover, by \cite[\S 3, Theorem 5.16]{nakayamaZariski-decompositionandAbundance}, we have \[\pi^{-1}\Supp(N_\sigma(K_X+B/Z))=\Supp(N_\sigma(K_{X'}+B'/Z')).\] Since $(X',B')$ is isomorphic in codimension one to $(X'',B'')$, there is a one-to-one correspondence between $|m(K_{X'}+B')/Z'|$ and $|m(K_{X''}+B'')/Z'|$ for every $m\in \mathbb{N}$, hence \[\eta^{-1}\Supp(N_\sigma(K_{X'}+B'/Z'))=\Supp(N_\sigma(K_{X''}+B''/Z')),\] and we finish the proof.

    (2). Since $X'\rightarrow X$ is a finite cover, we have \[\pi^{-1}\mathbb{E}(K_X+B/Z)=\mathbb{E}(K_{X'}+B'/Z')\]
    by \cite[Theorem 1.1]{gomezStableAugmentedBase2022}.
    Since $(X'',B'')$ is isomorphic in codimension one to $(X',B')$, the divisorial part of $\mathbb{E}(K_{X''}+B''/Z')$ coincides with the strict transform of the divisorial part of $\mathbb{E}(K_{X'}+B'/Z')$.
 Let $(\widetilde{X}^c,\widetilde{B}^c)$ be the normalization of $(X^c,B^c)\times_Z Z'$. Since $X^m\rightarrow X^c$ contracts  $\mathbb{E}(K_{X^m}+B^m/Z)$ and $X''^m\rightarrow X''^c$ contracts  $\mathbb{E}(K_{X''^m}+B''^m/Z')$, we conclude that $(X''^c,B''^c)$ is isomorphic in codimension one to $(\widetilde{X}^c,\widetilde{B}^c)$.

Note that $K_{\widetilde{X}^c}+\widetilde{B}^c$ is ample because $K_{X^c}+B^c$ is ample and $\widetilde{X}^c\rightarrow X^c$ is a finite cover. Since $K_{\widetilde{X}^c}+\widetilde{B}^c$ and $K_{X''^c}+B''^c$ are both ample, and since $(X''^c,B''^c)$ and $(\widetilde{X}^c,\widetilde{B}^c)$ are isomorphic in codimension one, we conclude that $(X''^c,B''^c)$ is isomorphic to $(\widetilde{X}^c,\widetilde{B}^c)$.
    \end{proof}
\end{lem}

\subsection{b-divisors}
% \begin{definition}[\textbf{b}-divisors]
    Let $X$ be a normal quasi-projective variety. A \textit{$\textbf{b}$-divisor} $\M$ over $X$ is a collection of $\bR$-divisors $\M_Y$ on $Y$ for each birational morphism $Y \to X$  that are compatible with respect to pushdown, that is, if $Y' \to X$ is another birational morphism and $\psi: Y' \dashrightarrow  Y$ is a morphism, then $ \psi_*\M_{Y'} = \M_Y$.

    	A \textbf{b}-divisor $\M$ is \textit{\textbf{b}-$\bR$-Cartier} if there is a birational morphism $Y\to X$ such that 
 $\M_{Y}$ is $\bR$-Cartier and $\M_{Y'}$ is the pullback of $\M_{Y}$ on $Y'$ for any birational morphism $Y'\to Y$.
In this case, we say that $\M$ descends to $Y$, and it is represented by $\M_Y$,  
 we write $\M=\overline{\M_{Y}}$.

A \textbf{b}-$\bR$-Cartier $\bR$-divisor $\M$ represented by $\M_Y$ for some birational model $Y\to X$  is \textit{\textbf{b}-nef} if $\M_Y$ is nef. Similarly, $\M$ is \textit{\textbf{b}-nef and big}  if $\M_Y$ is nef and big.
% \end{definition}

\begin{definition}[Discrepancy \textbf{b}-divisors]
The discrepancy \textbf{b}-divisor $\textbf{A}=\textbf{A}(X, B)$ of a sub-pair $(X, B)$ is the \textbf{b}-$\bR$-divisor of $X$ with the trace $\textbf{A}_ Y=\sum a_i A_i$ defined by the formula
\[K_Y = f ^*(K_ X + B) +\textbf{A}_Y,\]
where $f : Y \to X$ is a proper birational morphism of normal quasi-projective varieties. Similarly, we define  $\textbf{A}^*=\textbf{A}^*(X, B)$ by
\[\textbf{A}_Y^* =\sum_{a_i>-1} a_i A_i.\]
Note that $\textbf{A}(X, B)=\textbf{A}^*(X, B)$ when $(X, B)$ is sub-klt.
By the definition, we have $\cO_X ( \lceil\textbf{A}^*(X,B)\rceil )\simeq\cO_X$ if $(X, B)$ is lc. We also have $\cO_X ( \lceil\textbf{A}(X,B)\rceil )\simeq\cO_X$ when $(X, B)$ is klt.
    
\end{definition}

\subsection{Generalized pairs and singularities}

A \textit{generalized sub-pair} $(X,B,\M)/Z$ consists of:
\begin{enumerate}
    \item a normal quasi-projective variety $X$ equipped with a projective morphism $X\to Z$,
    \item an $\bR$-divisor $B$ on $X$, and
    \item a \textbf{b}-$\bR$-Cartier \textbf{b}-divisor $\M$ over $X$, represented by a projective birational morphism $f: X'\to X$ and an $\bR$-Cartier $\bR$-divisor $\M_{X'}$ on $X'$  
    such that $\M_{X'}$ is nef over $Z$ and $K_X+B+\M_X$ is $\bR$-Cartier.
\end{enumerate}

When $Z$ is a point, we omit it and say that the pair is quasi-projective, in which case we also say that $(X,B,\M)$ is a generalized sub-pair with nef part $\M$. If, in addition, $B \geq 0$, then $(X,B,\M)$ is a \textit{generalized pair}.

Let $D$ be a prime divisor over $X$. Replace $X'$ with a log resolution of $(X,B)$ such that $D$ is a prime divisor on $X'$. We can write  
\[
K_{X'}+B'+\M_{X'}=\pi^*(K_X+B+\M_X).
\]  
Then the \textit{generalized log discrepancy} of $D$ is defined as  
\[
a(D,X,B,\M)=1-\mult_D B'.
\]

We say that $(X,B,\M)$ is \textit{sub-klt} (resp. \textit{sub-lc}, \textit{sub-$\epsilon$-lc}) if $a(D,X,B,\M)>0$ (resp. $a(D,X,B,\M)\geq0$, $a(D,X,B,\M)\geq \epsilon$) for every prime divisor $D$ over $X$.  
If $(X,B,\M)$ is a generalized pair, we remove the prefix ``sub'' and say the generalized pair is \textit{klt} (resp. \textit{lc}, \textit{$\epsilon$-lc}).

\subsection{Canonical bundle formula}

\begin{definition}\label{def:klt trivial fib}
	An \textit{lc-trivial fibration} (resp. \textit{klt-trivial fibration})  $f : (X, B) \rightarrow Z$ consists of a projective surjective morphism $f : X \rightarrow Z$ with connected fibers  between normal quasi-projective varieties and a pair $(X, B)$ satisfying the following properties:
	\begin{enumerate}
		\item (X, B) is sub-lc (resp. sub-klt) over the generic point of $Z$,
		\item $\operatorname{rank}f_*\cO _X ( \lceil\textbf{A}^*(X,B)\rceil ) = 1$, and
		\item there exists an $\bR$-Cartier $\bR$-divisor $L_Z$ on $Z$ such that
		\[K_X + B \sim_{\bR} f^*L_Z.\]
	\end{enumerate}
\end{definition}

In \cite{ambroShokurovBoundaryProperty2004,ambroModuliBdivisorLctrivial2005}, klt-trivial fibrations as in Definition \ref{def:klt trivial fib} are called lc-trivial fibrations.

Let $f : (X, B) \rightarrow Z$ be an lc-trivial fibration such that $\dim Z>0$.
Fix a prime divisor $D$ on $Z$ and let $t_D$ be the lc threshold of $f^*D$ with respect to $(X,B)$ over the generic point of $D$. Now let  $B_Z:=\sum (1-t_D) D$, where the sum runs over all the prime divisors on $Z$.  Let $M_Z:=L_Z-(K_Z+B_Z)$. Then we have \[K_X+B\sim_{\bR} f^*(K_Z+B_Z+M_Z).\] 
We call $B_Z$ the \textit{discriminant $\bR$-divisor} and $M_Z$ the \textit{moduli $\bR$-divisor} of adjunction. Note that $B_Z$ is uniquely determined but $M_Z$ is determined only up to $\bR$-linear equivalence. 

Now let $\phi \colon X' \to X$ and $\psi \colon Z' \to Z$ be birational morphisms from normal quasi-projective varieties, and assume that the induced map $f' \colon X' \dashrightarrow Z'$ is a morphism.
Let $K_{X'}+B'$ be the pullback of $K_X+B$ on $X'$, and let $L_{Z'}=\psi^*L_Z$.
Then $f' \colon (X',B') \to Z'$ is an lc-trivial fibration, and we can define a discriminant $\bR$-divisor $B_{Z'}$ and a moduli $\bR$-divisor $M_{Z'}$ on $Z'$ such that
\[
K_{X'}+B' \sim_{\bR} f'^*(K_{Z'}+B_{Z'}+M_{Z'}), 
\]
with $B_Z=\psi_*B_{Z'}$ and $M_Z=\psi_*M_{Z'}$.
In particular, the lc-trivial fibration $f \colon (X,B)\to Z$ induces \textbf{b}-$\bR$-divisors $\textbf{B}$ and $\M$ on $Z$, called respectively the \textit{discriminant} and the \textit{moduli} \textbf{b}-$\bR$-divisors.

\begin{thm}[\cite{ambroShokurovBoundaryProperty2004,fujinoModuliBdivisorsLctrivial2014a,huLogAbundanceModuli2020}]
With the above notation and assumptions, suppose that $(X,B)$ is lc over the generic point of $Z$.
Then there exists a proper birational morphism $Z' \to Z$ from a normal quasi-projective variety such that $\M_{Z'}$ is nef, and $\M_{Z''}=\nu^*\M_{Z'}$ for any proper birational morphism $\nu \colon Z'' \to Z'$.
In particular, we can view $(Z,B_Z,\M)$ as a generalized pair.
\end{thm}

\begin{prop}  
[{\cite[Proposition 3.1]{ambroModuliBdivisorLctrivial2005}}]\label{amb05prop3.1}
Let $f : (X, B) \to Z$ be a klt-trivial fibration over a proper normal variety $Z$. Let $\tau: Z' \to Z $ be a surjective morphism from a proper normal variety $Z'$, let $X'$ be the normalization of the main component of $X\times_Z Z'$, and let $B'$ be the divisor on $X'$ such that $K_{X'}+B'=\tau_X^*(K_X+B)$. Then we say that $f':(X',B')\to Z'$ is the   \textit{klt-trivial fibration induced by base change}.
    Let $\M$ and $\M'$ be the corresponding moduli \textbf{b}-divisors of $f$ and $f'$ respectively. Then we have
    \[\tau^*\M=\M'.\]
\end{prop}

\begin{thm}[\cite{ambroModuliBdivisorLctrivial2005}]
		\label{moduli part is nef and good}
		Let $f:(X,B)\rightarrow Z$ be a klt-trivial fibration over a normal projective variety $Z$ such that $B$ is a $\bQ$-divisor. 
Suppose that the geometric generic fiber $X_ {\overline{\eta}}= X \times _Z \operatorname{Spec}\overline{k(Z)}$ is projective and 
		$B _{\overline{\eta}}$ is effective. Then there exist  non-singular projective varieties $\overline{Z}$, $T$ and $V$,  and a commutative diagram
			\[\xymatrix{
		(X,B)  \ar[d] _{f}& &   (X_T,B_T) \ar[d]_{f_T} &\\
		Z	\ar@/_1pc/@{-->}[rrr] _\gamma &  \overline{Z}\ar[l]_{\tau} \ar[r]^{\rho}& 	 T \ar@/_1pc/@{-->}[ll]_i  \ar[r]^{\pi} & V
	} \]

		such that
		\begin{enumerate}
			\item $f_T:(X_T,B_T)\rightarrow T$ is a klt-trivial fibration,
			\item $\tau$ is generically finite and surjective, and $\rho$ is surjective, 
			\item there exists a nonempty open subset $U\subset \overline{Z}$ and an isomorphism
			\[\xymatrix{
				(X,B)\times_Z U  \ar[rr] _{\cong}\ar[dr]& &   (X_T,B_T)\times_{T} U \ar[dl]\\
				&  U,& 	 
			} \]
			
			\item let $\M$, $\N$ be the corresponding moduli $\mathbf{b}$-divisors of $f$ and $f_T$, then $\N$ is $\mathbf{b}$-nef and big, and  $\tau^*\M=\rho^*\N$,
			\item   $\pi$ is generically finite and surjective, $\Phi:Z\dashrightarrow V$ is bimeromorphic to the period map defined in \cite[Proposition 2.1]{ambroModuliBdivisorLctrivial2005}, and 
			\item $i:T\dashrightarrow Z$ is a rational map such that  $f_T:(X_T,B_T)\to T$ is equal to the pullback of $f:(X,B)\to Z$ via $i$ over some open subset of $T$.
			
		\end{enumerate}
		
	\end{thm}
	\begin{proof}
The assertions (1)--(4) are stated in \cite[Theorem 3.3]{ambroModuliBdivisorLctrivial2005}, while  (5) and (6) can be derived from the proof of \cite[Theorem 2.2]{ambroModuliBdivisorLctrivial2005}. Indeed, by the algebraization theorem in \cite[Theorem 11]{kawamataKodairaDimensionCertain1983}, the period map defined in \cite[Proposition 2.1]{ambroModuliBdivisorLctrivial2005} is bimeromorphic to   a morphism $\gamma^o:Z^o\rightarrow V^o$ from a non-empty open subset of $Z$ to a non-singular quasi-projective variety $V^o$.
Let $T^o\rightarrow V^o$ be a generically finite surjective morphism from a non-singular quasi-projective variety $T^o$ such that if $\overline{Z}^o$  is the main part of $Z^o\times_{V^o} T^o$, then the induced morphism $\rho^o:\overline{Z}^o\rightarrow T^o$ has a  section $\alpha$.  By base change via the  section $i^o: T^o\stackrel{\alpha}{\longrightarrow}\overline{Z}^o\stackrel{\tau^o}{\longrightarrow}Z^o$, we induce a family $f_{T^o}:(X_{T^o},B_{T^o})\rightarrow T^o$. After replacing $\overline{Z}^o$ and $T^o$ by generically finite covers from non-singular quasi-projective varieties,  we have an isomorphism of pairs over $\overline{Z}^o$
\[(X,B)\times_{Z}\overline{Z}^o\stackrel{\sim}{\longrightarrow} (X_{T^o},B_{T^o})\times_{T^o} \overline{Z}^o.\]
Let $\overline{Z}$, $T$ and $V$ be non-singular projective compactifications of $\overline{Z}^o$, $T^o$ and $V^o$ respectively, and let $(X_T,B_T)$ be a normal projective compactification of $(X_{T^o},B_{T^o})$ so that $f_{T^o}$ induces a  klt-trivial fibration $f_T:(X_T,B_T)\rightarrow T$.
Then  (5) and (6) are satisfied.
	\end{proof} 

\subsection{Finite covers}

\begin{prop}\label{:prop:diagonal}
Let $f \colon (X,B) \to Z$ be a projective isotrivial fibration between normal quasi-projective varieties
with general fiber $(F,B_F)$.
Assume that there exists a Galois cover $Z' \to Z$ with Galois group $G$ such that
\[
(X,B) \cong \bigl((F,B_F) \times Z'\bigr)/G,
\]
where $G$ acts on $(F,B_F) \times Z'$ and the action is $G$-equivariant with respect to the projection
$(F,B_F) \times Z' \to Z'$.
Assume further that $(F,B_F)$ is klt and that $\kappa(F,K_F + B_F)\geq 0$.

Then there exists another Galois cover $\widetilde{Z} \to Z$ with the Galois group $H$ such that
\[
(\widetilde{X},\widetilde{B})
:= (X,B) \times_Z \widetilde{Z}
\cong (F,B_F) \times \widetilde{Z}.
\]
Moreover,
\[
(X,B)\simeq \bigl((F,B_F)\times \widetilde{Z}\bigr)/H,
\]
where $H$ acts diagonally on $(F,B_F)\times \widetilde{Z}$, that is, there exists a homomorphism
\[
\rho_F \colon H \to \Aut(F,B_F):=\{\sigma\in \Aut(F)\mid \sigma_*B_F=B_F\}
\]
such that for every $h\in H$,
\[
h\cdot ((f,b),\widetilde{z})
= \bigl(\rho_F(h)\cdot (f,b), h\cdot \widetilde{z}\bigr),
\]
for all $(f,b)\in (F,B_F)$ and $\widetilde{z}\in \widetilde{Z}$.
\end{prop}
\begin{remark}
One may replace $H$ by the quotient $H/\ker(\rho_F)$ and replace
$\widetilde{Z}$ by the corresponding Galois cover, so that the induced action on
$(F,B_F)$ becomes faithful. However, we do not require this, since in applications
we frequently replace $\widetilde{Z}$ by further Galois covers.
\end{remark}
\begin{proof}
   We follow the argument of \cite[Theorem 43]{kollarDeformationsEllipticCalabiYau2015}.

 Since $f \colon (X,B)\to Z$ becomes a product family after the Galois cover
$Z'\to Z$, we have
\[
\operatorname{Im}\bigl(G \to \Aut(F,B_F)\bigr)
=
\operatorname{Im}\bigl(\pi_1(Z)\to \Aut(F,B_F)\bigr).
\]
Consider the discrete part of the monodromy representation
\[
\rho^d \colon \pi_1(Z)\to \Aut(F,B_F)/\Aut^0(F,B_F),
\]
where $\Aut^0(F,B_F)$ is the connected component  of 
$\Aut(F,B_F)$ containing the identity.
Then
\[
H_d := \operatorname{Im}\bigl(\rho^d\bigr)
= \operatorname{Im}\bigl(G\to \Aut(F,B_F)/\Aut^0(F,B_F)\bigr)
\]
is a finite group.
Let $Z^d\to Z$ be the corresponding Galois cover with Galois group $H_d$.

The trivialization of the translation part $\Aut^0(F,B_F)$ is more subtle and it depends on additional choices.

Since $(F,B_F)$ is klt and $\kappa(F,K_F+B_F)\ge 0$, the group
$\Aut^0(F,B_F)$ is an abelian variety by \cite[Proposition~4.6]{ambroModuliBdivisorLctrivial2005}.
A general $\Aut^0(F,B_F)$-orbit $A_F \subset F$ defines an isotrivial abelian family
\[X^d \supset A^d \to Z^d.\]
By assumption, there exists an $f$-ample line bundle $L$ on $X$, whose pullback
defines a relatively ample line bundle $L_{A^d}$ on $A^d$.
We may assume that the degree $m$ of $L_{A^d}$ on the general fiber is at least $3$. 
By \cite[Remark 44]{kollarDeformationsEllipticCalabiYau2015}, there exists a closed subscheme
$T^d \subset A^d$ such that $T^d \to Z^d$ is \'etale and every fiber
$T^d_F \subset A_F$ is a translate of the subgroup of $m$-torsion points.
Since $L_{A^d}$ is $H_d$-invariant, the morphism $T^d\to Z^d$ is  $H_d$-equivariant, hence it defines a  monodromy representation
\[\pi_1(Z^d) \to \Aut^0(F,B_F).\]
Let $\Gamma \subset \pi_1(Z^d)$ be a finite-index subgroup that is normal in
$\pi_1(Z)$, and let $\widetilde{Z} \to Z$ be the corresponding Galois cover
with Galois group $H := \pi_1(Z)/\Gamma$.

By pullback, we obtain an isotrivial abelian fibration
$\widetilde{A} \to \widetilde{Z}$
 with a trivialization of the $m$-torsion points.
For $m \ge 3$, by \cite[p.~513, Lemma~2.9]{arbarelloGeometryAlgebraicCurves2011},
we have $\widetilde{A} \simeq A_F  \times \widetilde{Z}$.
Consequently, the same pullback also trivializes the fibration
\[(\widetilde{X},\widetilde{B})\cong (F,B_F)\times  \widetilde{Z}\to \widetilde{Z}.\]
The $H$-action on $(\widetilde{X},\widetilde{B})\cong (F,B_F)\times \widetilde{Z}$ is given by
\[
h\colon ((f,b),\widetilde{z}) \longmapsto
\bigl(\rho_{F,\widetilde{z}}(h)\cdot (f,b),\; h\cdot \widetilde{z}\bigr),
\]
where $\rho_{F,\widetilde{z}}\colon H \to \Aut(F,B_F)$.
Note that $\rho_{F,\widetilde{z}}$ preserves the $m$-torsion points, and that
the automorphisms of an abelian torsor (abelian variety without a specified origin) preserving a finite nonempty set
form a discrete group. It follows that $\rho_{F,\widetilde{z}}$ is in fact
independent of $\widetilde{z}$.
Hence the $H$-action on $(\widetilde{X},\widetilde{B})$
can be written as
\[
h\colon ((f,b),\widetilde{z}) \longmapsto
\bigl(\rho_F(h)\cdot (f,b),\; h\cdot \widetilde{z}\bigr).
\]
\end{proof}

The following lemma allows us to modify a generically finite cover into a finite cover.
 \begin{lem}\label{lem: finite modification}
     Let $\pi:T\to V$ be a generically finite cover between normal projective varieties. Then there exists a generically finite cover $S^!\to T$ from a smooth projective variety $S^!$ and a birational map  $S^*\dashrightarrow V$ from a projective variety $S^*$ such that  $S^! \to S^*$ is a finite cover. 
 \end{lem} 
 \begin{proof}
Let $S'\to T$ be a finite cover such that $S'\to V$ is Galois over an open subset of $V$ with Galois group $G$. Let $S''$ be the closure of $V$ in $K(S')$. Then $S'' \dashrightarrow S'$ is birational and $V=S''/G$. 

Let $S^!\to S''$ be a $G$-equivariant resolution such that $S^!\to S'$ is a morphism, and let $S^*$ be the quotient of $S^!$ by $G$. Then, the map $S^*\dashrightarrow V$ is birational, $S^!\to T$ is a generically finite surjective morphism, and $S^!\to S^*$ is a finite cover.
        \[\xymatrix{
    & S'  \ar[d] & S'' \ar[d]\ar@{-->}[l] & S^! \ar[d]\ar[l]\ar@/_1pc/[ll]\\
    & T \ar[r]^\pi &V  & S^*  \ar@{-->}[l]   
}\]
 \end{proof}

The following lemma shows that relative $\bQ$-linear triviality can descend under finite covers.
\begin{lem}\label{linear trivial preserved under finite base change}
Assume that \begin{enumerate}
    \item $f:X\to Z$ is a contraction between two normal projective varieties,
    \item $D$ is a $\bQ$-Cartier $\bQ$-divisor on $X$,
    \item $\mu:Z'\to Z$ is a finite cover,
    \item  $X'$ is the normalization of $X\times_Z Z'$, and
    \item  denote the induced finite cover $X'\to X$ by $\pi$ and the induced contraction $X'\to Z'$ by $f'$.
\end{enumerate}
\[\xymatrix{X'\ar[r]^{\pi}\ar[d]_{f'}&X\ar[d]^f\\ Z'\ar[r]^{\mu}& Z}\]
If $\pi^*D\sim_{\bQ,Z'} 0$, then $D\sim_{\bQ,Z} 0$.
\begin{proof}
    Replacing $Z'$ with a finite cover and replacing $X'$ accordingly, we can assume that $\mu:Z'\to Z$ is a Galois cover with Galois group $G$. Then $G$ acts on $X'$ by base change, hence $\pi:X'\to X$ is also a Galois cover. Since $\pi^*D\sim_{\bQ,Z'} 0$, there is a $\bQ$-Cartier $\bQ$-divisor $L'$ on $Z'$ such that $\pi^*D\sim_\bQ f'^*L'$. 
    Since $\pi^*D$ is $G$-invariant, replacing $L'$ with $\frac{1}{|G|}\sum_{g\in G}g^*L'$, we can assume that $L'$ is $G$-invariant. Therefore, there exists a $\bQ$-Cartier $\bQ$-divisor $L$ on $Z$ such that $L'=\mu^*L$. Then $\pi^*D\sim_\bQ f'^*\mu^*L$, hence $D\sim_\bQ f^*L$ and we finish the proof.
\end{proof}
\end{lem}

\subsection{Locally stable family}

\begin{definition}[Relative Mumford divisor]\label{def:relative mumford}
		Let $f:X\to Z$ be a flat finite type morphism with $S_2$ fibers of pure dimension $d$. 
		A subscheme $D\subset X$ is a \emph{relative Mumford divisor}   if there is an open set $U\subset X$ such that

		\begin{enumerate}
			\item $\operatorname{codim}_{X_z}(X_z\setminus U_z) \geq 2$ for each $z\in Z$, 
			\item $D\vert_U$ is a relative Cartier divisor,
			\item $D$ is the closure of $D\vert_U$, and
			\item $X_z$ is smooth at the generic points of $D_z$ for every $z\in Z$.
		\end{enumerate}
		
		By $D|_U$ being relative Cartier we mean that $D|_U$ is a Cartier divisor on $U$ and that 
		its support does not contain any irreducible component of any fiber $U_z$.  
		
		If $D\subset X$  is a relative Mumford divisor for $f:X\to Z$ and $T\to Z$ is a morphism, then the \textit{divisorial pullback}  $D_T$ on $X_T := X\times_Z T$ is the relative Mumford divisor defined to be the closure of the pullback of $D|_U $ to $U_T$.  In particular, for each $z\in Z$, we define $D_z = D|_{X_z}$ to be the closure of $D|_{U_z}$  which is the divisorial pullback of $D$ to $X_z$.
		
	\end{definition}

	\begin{definition}[Locally stable family]\label{def:locally stable}
		A \emph{locally stable family of slc pairs} $(X,B) \to Z$ over a reduced Noetherian scheme $Z$
		is a flat finite type morphism $X\to Z$ with $S_2$ fibers  and a $\bQ$-divisor $B$ on $X$ satisfying
		\begin{enumerate}
			\item  each prime component of $B$ is a relative Mumford divisor,
			\item $K_{X/Z} +B$ is $\bQ$-Cartier, and
			\item  $(X_z,B_z)$ is an slc pair for any point $z\in Z$.
		\end{enumerate}

	\end{definition}
Slc pairs naturally appear in the degeneration of lc pairs. For background on slc singularities, see \cite[\S5]{kollarSingularitiesMinimalModel2013}.

    Given a morphism $T \to Z$ of reduced schemes, we get the \textit{induced locally stable family} $ (X_T , B_ T ) \to T$ where $X_T = X \times_Z T$ and $B_T$ is defined by divisorial pullback.
\begin{definition}[Hodge line bundle]
If $f:(X,B)\to Z$ is  a locally stable family of pairs such that
$N(K_{X/Z}+B)\sim_Z 0$, we set 
\[
\lambda_{{\rm Hodge},f,N}
:= 
f_*\cO_X(N(K_{X/Z}+B)).
\]
\end{definition}

\begin{prop}\label{moduli part descends on locally stable morphisms}    
Let $f:(X,B)\to Z$ be a locally stable family of pairs such that
$N(K_{X/Z}+B)\sim_Z 0$.  
Then we have the following statements:
\begin{enumerate}
\item[(1)]  $\lambda_{{\rm Hodge},f,N}$ is the unique line bundle (up to isomorphism)
satisfying 
\[
\cO_X( N(K_{X/Z}+B)) \cong f^* \lambda_{{\rm Hodge},f,N}.
\]
\item[(2)] If $\varphi:Z'\to Z$ is a morphism and $f':(X',B')\to Z'$ denotes the pullback of $(X,B)\to Z$ by $\varphi$, then there is a canonical isomorphism
\[
 \varphi^* \lambda_{{\rm Hodge},f,N} 
 \overset{\sim}{\longrightarrow} \lambda_{{\rm Hodge},f',N}.\]
\item[(3)]  If $Z$ is smooth and the generic fiber of $X\to Z$ is normal, 
then $f:(X,B)\to Z$ is an lc-trivial fibration with 
$\cO_Z(N\M_{Z}) \cong \lambda_{{\rm Hodge},f,N}$, and the moduli \textbf{b}-divisor $\M$ of $f$ descends on $Z$.
\end{enumerate}
\end{prop}
\begin{proof}
This is \cite[Proposition 14.7]{ascherModuliBoundaryPolarized2023}. While the proposition is stated only for families of boundary polarized Calabi--Yau pairs, their proof also applies to families of general Calabi--Yau pairs.
\end{proof}

We need the following lemma about numerically trivial divisors in a flat family.

\begin{lem}\label{lem:num trivial}
Let $f: X \to S$ be a projective flat morphism with integral fibers and of relative dimension $d$, and let $L$ be a flat family of divisors over $S$. If there exists a point $0 \in S$ such that $L_0 \equiv 0$, then $L_s \equiv 0$ for all $s \in S$.
\end{lem}

 \begin{proof}
     Let $H$ be a relatively very ample line bundle on $X$. Take a closed point $s$ on $S$. Choose $m \gg 0$ such that
\[
\chi(X_s, n(mH_s + L_s)) = h^0(X_s, n(mH_s + L_s)) \quad \text{for } n\geq 1.
\]
Since $L$ is flat over $S$, it follows that
\[
\chi(X_s, n(mH_s + L_s)) = \chi(X_0, n(mH_0 + L_0)).\]
Therefore, we have 
\[h^0(X_s, n(mH_s + L_s))=h^0(X_0, n(mH_0 + L_0)).\]
From the leading term in the polynomial expansion in $n$, we obtain
\begin{equation} 
(mH_s + L_s)^d = (mH_0 + L_0)^d.\notag
\end{equation}
Similarly, we have
\begin{equation} 
(mH_s)^d = (mH_0)^d.\notag
\end{equation}
Since $L_0 \equiv 0$, it follows that
\begin{equation} 
(mH_s + L_s)^d = (mH_s)^d.\notag
\end{equation}
Expanding the left-hand side and canceling the dominant terms, we obtain
\begin{equation} 
H_s^{d-1} \cdot L_s = H_s^{d-2} \cdot L_s^2 = 0.\notag
\end{equation}
Restricting to a surface by taking  general hyperplane sections and applying the Hodge index theorem, we conclude that $L_s \equiv 0$.
 \end{proof}

 \subsection{Bounded families of pairs and morphisms}\label{sec:bdd family}

We say that a collection of log pairs $\cP$ is \textit{log birationally bounded} (resp., \textit{log bounded}, or \textit{log bounded in codimension one}) if there is a quasi-projective scheme $\cX$, a reduced divisor $\cE$ on $\cX$, and a projective morphism  $h: \cX\to T$, 
where $T$ is of finite type and $\cE$ does not contain any fiber, such that for every $(X, B) \in\cP$, there is a closed point $t\in T$ and a birational map $f : \cX_t  \dashrightarrow X$ (resp. isomorphic, or isomorphic in codimension one) 
such that $\cE_t$ contains the support of $f^{-1}_*B$ and any $f$-exceptional divisor (resp. $\cE_t$  coincides with the support of $f^{-1}_*B$, or $\cE_t$  coincides with the support of $f^{-1}_*B$).

We say that a collection of morphisms $\cF$ is  \textit{bounded} if there exist quasi-projective schemes $\cX,\cZ$, and projective morphisms $\cX\xrightarrow{\phi} \cZ\to T$, where $T$ is of finite type, such that for every morphism $X\to Z$ in $\cF$, there is a closed point $t\in T$ such that $\cX_t\to \cZ_t$ is isomorphic to $X\to Z$. 

\section{Polarized log Calabi--Yau fibrations: finite coefficients}\label{sec:Polarized log Calabi--Yau fibrations: finite coefficients}

In this section, we consider the boundedness of weak polarized log Calabi--Yau fibrations $f: ((X,B),A) \to (Z,H)$ such that the coefficients of $B$ belong to a finite set $\Phi$.
We will prove the following more general form of Theorem \ref{thm: finite coeff bdd in codim 1}.
 
\begin{thm}\label{thm:bdd of weak pcy fibration}
 Let $d\in \bN$, $v,r,\epsilon\in\mathbb{Q}^{>0}$, and $\Phi\subset[0,1]\cap \bQ$ be a finite set. Then there exists a positive integer $l$ and a bounded set of couples $\cP$ depending only on $d,\Phi,v,r,\epsilon$ satisfying the following.
 
 Assume that $f:((X,B),A)\to (Z,H)$ is a weak $(d,\Phi,v,r,\epsilon)$-polarized log Calabi--Yau fibration, and $H_Z\geq 0$ is a general element of $|6dH|$. Then there exists a couple $(V,\Theta)$ and an effective integral divisor $J$ on $V$ such that
 \begin{enumerate}
     \item there is a contraction $h:V\to Z$ and $V$ is $\bQ$-factorial,
     \item $V\dashrightarrow X/Z$ is an isomorphism in codimension one,
     \item $(V,\Theta+\Supp(J))$ belongs to $\cP$,
     \item $\Theta$ contains $h^*H_Z$ and the strict transform of $B$, and
     \item $J \equiv lA_V$ over the generic point of $Z$, where $A_V$ is the strict transform of $A$ on $V$.
 \end{enumerate}

\end{thm}

\begin{lem}\label{lem: A>0}

     Assume that Theorem \ref{thm:bdd of weak pcy fibration} holds when $A$ is an effective integral divisor and $\vol(A|_F) = v$ for some fixed $v \in \bQ^{>0}$, where $F$ is the general fiber of $f:X\to Z$. Then the theorem holds in general.
     \begin{proof}
If $(F, B_F)$ is the general fiber of $f: (X, B) \to Z$ and $A_F := A|_F$, then by \cite[Theorem 1.3]{birkarGeometryPolarisedVarieties2023}, there exists a positive integer $m$ depending only on $\dim F$ and $\epsilon$ such that $H^0(F, \mathcal{O}_X(mA|_F)) \neq 0$.
By upper-semicontinuity of cohomology and \cite[Lemma 3.2.1]{birkarExistenceMinimalModels2010}, $mA \sim_Z G$ for some effective integral divisor $G$ on $X$. Replacing $A$ and $v$ with $G$ and $m^{\dim F} v$ respectively, we may assume that $A \geq 0$.
Moreover, by \cite[Corollary 1.6]{birkarGeometryPolarisedVarieties2023}, the pair $(F, \Supp (B_F + A_F))$ belongs to a log bounded family. Hence, we can assume that $\vol(A_F)$ is fixed.
     \end{proof}
 \end{lem}

From now until the end of this section, we will assume that $A$ is an effective integral divisor and that $\vol(A|_F) = v$ for some fixed $v \in \bQ^{>0}$.
\subsection{Family of polarized log Calabi--Yau pairs}

	\begin{definition}{(\cite{birkarModuliAlgebraicVarieties2022,birkarGeometryPolarisedVarieties2023})}
	Let $d\in \bN$, $v\in\mathbb{Q}^{>0}$, and $\Phi\subset[0,1]\cap \bQ$ be a finite set.  A \textit{$(d, \Phi, v)$-polarized  log Calabi--Yau pair}  $((X, B),A)$ is defined by the data:
	\begin{enumerate}
		\item $(X, B)$ is projective slc pair of dimension  $d$ with $K_X+B\sim_{\bQ} 0$,
		\item the coefficients of $B$ are in $\Phi$,
		\item $A \geq 0$ is an ample integral divisor with volume $\vol(A) = v$,
		\item $(X,B+tA)$ is slc  for some $t \in\bQ^{ >0}$.
	\end{enumerate}
	If $(X,B)$ is klt, then $((X, B),A)$ is called a \textit{klt $(d, \Phi, v)$-polarized log Calabi--Yau pair.}
\end{definition}

Given a weak $(d,\Phi,v,r,\epsilon)$-polarized log Calabi--Yau fibration $f:((X,B),A)\to (Z,H)$, it follows that the general fiber $((F,B_F),A_F)$ of $f$ is a klt $(\dim F,\Phi,v)$-polarized  log Calabi--Yau pair, hence it is bounded by \cite[Corollary 1.6]{birkarGeometryPolarisedVarieties2023}.

In the following theorem, we use the moduli theory for polarized log Calabi--Yau pairs \cite{birkarModuliAlgebraicVarieties2022} to  construct a locally stable family of polarized log Calabi--Yau pairs $f_{\cS}:((\cX,\cB),\cA)\to \cS$ such that, over an open subset of $Z$, the fibration $f:((X,B),A)\to Z$ is obtained as the pullback of $f_{\cS}$.
We then apply \cite{ambroModuliBdivisorLctrivial2005} to $f_{\cS}$ to produce a new family $f_{\cS^!}:(\cX^!,\cB^!)\to \cS^!$ of maximal variation. As a consequence, the moduli $\mathbf{b}$-divisor $\bm{\cM}^!$ of $f_{\cS^!}$ descends to a nef and big divisor $\bm{\cM}_{\cS^!}$ on $\cS^!$, which plays a crucial role in establishing the boundedness of the moduli map in Theorem \ref{moduli map is bounded}.
 A key step in the proof of Theorem \ref{diagram of the universal Calabi--Yau fibration} is the construction of a new polarization $\cL$ on $\cX$ induced from $\cX^!$, such that $\cL_s\equiv m\cA_s$ for some bounded integer $m\in\bN$ and all $s\in \cS$. This construction allows us to prove the boundedness of the log canonical volume of a certain log general type pair in Theorem \ref{volume is bounded}, and it is also essential for the proof of Theorem \ref{mainthm2}. Finally, we establish several auxiliary results that will be used in later subsections.

\begin{thm}\label{diagram of the universal Calabi--Yau fibration}
Let $d\in \bN$,  $v\in\mathbb{Q}^{>0}$, and $\Phi\subset[0,1]\cap \bQ$ be a finite set. Let $f:(X,B)\rightarrow Z$ be a klt-trivial fibration, and let $A$ be an effective integral divisor on $X$. Assume that the general fiber $((F,B_F),A_F)$ of $f$ is a klt $(d,\Phi,v)$-polarized log  Calabi--Yau pair. 
    Then there exists a commutative diagram 
    \[\xymatrix{((X,B),A) \ar[d]^f&((X_U,B_U),A_U)\ar[d] ^{f_U} \ar@{_{(}->}[l] \ar[r]&
    ((\cX,\cB),\cA,\cL )\ar[d]^{f_{\cS}} &\overline{\cX}\ar[r]^(.3){\rho_{\cX}} \ar[l]_(.3){\tau_{\cX}} \ar[d] & ((\cX^!,\cB^!),\cL^! )\ar[d] ^{f_{\cS^!}} & \\
  Z&U\ar[r]^{\phi} \ar@{_{(}->}[l]&  \cS\ar@/_1pc/[rrd]_{\gamma}  & \overline{\cS}\ar[r]^{\rho} \ar[l]_{\tau} \ar[l] &( \cS^!,\cM^!) \ar[d]^\pi
  \\&&&& (\cS^*, \cH)
    }\]
    satisfying the following:
     
\begin{enumerate}
    \item $\cS$, $\overline{\cS}$, and $\cS^!$ are smooth schemes.

    \item $\cS^!$ and $\cS^*$ are projective schemes.

    \item  $\tau\colon \overline{\cS}\to \cS$ is finite, 
$\pi\colon \cS^!\to \cS^*$ is generically finite, 
$\rho\colon \overline{\cS}\to \cS^!$ is dominant, 
and $\gamma\colon \cS\to \cS^*$ is a morphism.

    \item The generic fiber of the base change of
    $(\cX,\cB)\to \cS$ to $\overline{\cS}$
    is isomorphic to the generic fiber of the base change of
    $(\cX^!,\cB^!)\to \cS^!$ to $\overline{\cS}$.

    \item $\overline{\cX}$ is a common resolution of the main components of
    $\cX\times_{\cS}\overline{\cS}$ and
    $\cX^!\times_{\cS^!}\overline{\cS}$.

    \item There exist $\bQ$-Cartier integral divisors $\cA$ and $\cL$ on $\cX$,
    and a $\bQ$-Cartier integral divisor $\cL^!$ on $\cX^!$,
    such that for some $m\in\bN$, depending only on $(d,\Phi,v)$, the following hold:
    \begin{itemize}
        \item $\cL_s \equiv m\cA_s$ for all $s\in \cS$, and
        \item $\tau_{\cX}^*\cL = \rho_{\cX}^*\cL^!$.
    \end{itemize}

    \item The morphisms $(\cX,\cB+\alpha\cL)\to \cS$ and $(\cX^!,\cB^!+\alpha\cL^!)\to \cS^!$ are locally stable for some $\alpha\in\bQ^{>0}$ depending only on $(d,\Phi,v)$.

    \item There exists a very ample divisor $\cH\geq 0$ on $\cS^*$ such that
    \begin{itemize}
        \item the morphism $\pi$ is étale and Galois over
        $\cS^*\setminus \Supp(\cH)$, and
        \item every fiber of
        $((\cX^!,\cB^!),\cL^!)\to \cS^!$
        over $\cS^!\setminus \Supp(\pi^*\cH)$
        is a klt $(d,\Phi,m^d v)$-polarized log Calabi--Yau pair.
    \end{itemize}

    \item The moduli $\mathbf{b}$-divisor $\bm{\cM}^!$ of
    $(\cX^!,\cB^!)\to \cS^!$ descends to $\cS^!$.
    Moreover, there exists an effective divisor
    $\cM^!\sim_{\bQ}\bm{\cM}^!_{\cS^!}$
    such that $l\cM^!$ is Cartier and
    $l\cM^!\geq \pi^*\cH$
    for some $l\in\bN$ depending only on $(d,\Phi,v)$.

    \item There exists an open subset $U\subset Z$ and a morphism
    $\phi\colon U\to \cS$ such that
    $((X_U,B_U),A_U)\to U$
    is isomorphic to the base change of
    $((\cX,\cB),\cA)\to \cS$ via $\phi$.

    \item If the composition $\gamma\circ\phi$ extends to a morphism
    $\psi\colon Z\to \cS^*$,
    then
    \[
        \psi(Z)\not\subset \pi(\Supp(\cM^!)).
    \]
\end{enumerate}
    \begin{proof}
    \begin{enumerate} [label=\textsl{Step} \arabic{enumi}., wide=13pt, itemsep=13pt]
    \item In this step, we construct a universal family $((\cX_{(1)},\cB_{(1)}),\cA_{(1)})\to \cS_{(1)}$ parametrizing the general fibers of
$f\colon ((X,B),A)\to Z$.

By \cite[Lemma~10.2]{birkarModuliAlgebraicVarieties2022}, there exists
$\alpha\in\bQ^{>0}$ and $r\in\bN$, depending only on $(d,\Phi,v)$, such that the following hold:
\begin{itemize}
    \item $(F,B_F+\alpha A_F)$ is klt for the general fiber $((F,B_F),A_F)$ of $f$, and
    \item  $r(K_F+B_F+\alpha A_F)$ is very ample without higher cohomology.
\end{itemize}
Set
$n := h^0\bigl(r(K_F+B_F+\alpha A_F)\bigr)-1 $.
Then, $r(K_F+B_F+\alpha A_F)$ defines an embedding $F\hookrightarrow \bP^n$.
Since $r(K_F+B_F+\alpha A_F)$ is very ample without higher cohomology, there exists a non-empty open subset
$U\subset Z$ such that $r(K_{X_U}+B_U+\alpha A_U)$ defines an embedding
$X_U\hookrightarrow \bP^n_U$.

By \cite[Proposition~9.5]{birkarModuliAlgebraicVarieties2022}, there exists a finite type scheme
$\cS_{(1)}$ representing the functor of strongly embedded
$(d,\Phi_{1/c},v,\alpha,r,\bP^n)$-polarized log Calabi--Yau families
(see \cite[Definition~9.3]{birkarModuliAlgebraicVarieties2022})
over reduced schemes, where $c\in\bN^{>0}$ satisfies $c\Phi\subset\bN$.
After replacing $\cS_{(1)}$ by a suitable locally closed subscheme, we may assume that
$\cS_{(1)}$ parametrizes klt $(d,\Phi,v)$-polarized log Calabi--Yau pairs.
Let
\[
((\cX_{(1)}\subset \bP^n_{\cS_{(1)}},\cB_{(1)}),\cA_{(1)})\to \cS_{(1)}
\]
be the corresponding universal family.
Then 
$(\cX_{(1)},\cB_{(1)}+\alpha\cA_{(1)})\to \cS_{(1)}$
is locally stable and 
$K_{\cX_{(1)}}+\cB_{(1)} \sim_{\bQ,\cS_{(1)}} 0 $.
Moreover, there exists a moduli map $\phi\colon U\to \cS_{(1)}$ such that
the restriction $((X_U,B_U),A_U)\to U$ is isomorphic to the pullback of the universal family
via $\phi$.

        	\item   In this step, we apply Theorem \ref{moduli part is nef and good} to the universal family  obtained in Step 1.
        	
             By applying Theorem \ref{moduli part is nef and good} to a projective compactification of $(\cX_{(1)},\cB_{(1)}) \rightarrow \cS_{(1)}$, we have a non-singular quasi-projective variety  $\overline{\cS}_{(1)}$, non-singular projective varieties $\cT$ and $\cV$, and a commutative diagram
        \[\xymatrix{
			(\cX_{(1)},\cB_{(1)})  \ar[d]_{f_{\cS_{(1)}}} & &   (\cX_{\cT},\cB_{\cT}) \ar[d]_{f_{\cT}} &\\
			\cS_{(1)} 	\ar@/_1pc/@{-->}[rrr] _\gamma &  \overline{\cS}_{(1)} \ar[l]_{\tau} \ar[r]^\rho& 	 \cT \ar@/_1pc/@{-->}[ll]_i  \ar[r]^\pi & \cV,
		}\]
		such that
		
        \begin{itemize}
			\item $(\cX_{\cT},\cB_{\cT})\rightarrow \cT$ is a klt-trivial fibration,
			\item $\tau:\overline{\cS}_{(1)} \rightarrow \cS_{(1)}$ and $\pi:\cT\rightarrow \cV$ are generically finite, surjective morphisms, $\rho:\overline{\cS}_{(1)} \rightarrow \cT$ is a dominant morphism,
			\item there exists a nonempty open subset $\cU\subset \overline{\cS}_{(1)}$ and an isomorphism
			\[\xymatrix{
				(\cX_{(1)},\cB_{(1)})\times_{\cS_{(1)}} \cU  \ar[rr] ^{\cong}\ar[dr]& &   (\cX_{\cT},\cB_{\cT})\times_{\cT} \cU \ar[dl]\\
				&  \cU& 	 
			}\]
			
			\item the moduli $\mathbf{b}$-divisor of $f_{\cT}$ is $\mathbf{b}$-nef and big,
			\item $\gamma:\cS_{(1)}\dashrightarrow \cV$  is bimeromorphic to the period map defined in \cite[Proposition 2.1]{ambroModuliBdivisorLctrivial2005}, and 
			\item $i:\cT\dashrightarrow \cS_{(1)}$ is a generically finite rational map such that $f_{\cT}:(\cX_{\cT},\cB_{\cT})\to\cT$ is equal to the pullback of  $f_{\cS_{(1)}}$ via $i$   over some open subset of $\cT$.
		\end{itemize}
  
  	\item In this step, we shrink $\cS_{(1)}$ and construct a smooth projective variety $\cS^!$ over which $(\cX^!,\cB^!)\to \cS^!$ is a locally stable family 
    of maximal variation. Then, we verify  (1)--(4).
  	
        Let $\cS_{(2)}$ be an open subset of $\cS_{(1)}$ and $\overline{\cS}_{(2)}$ be an  open subset of $\cU$ such that
        \begin{itemize}
        \item  $\cS_{(2)}$ is smooth,
        	\item $\gamma$ is a morphism on $\cS_{(2)}$, 
        	\item $\overline{\cS}_{(2)}\rightarrow \cS_{(2)}$ is a finite cover, and
        	\item $i|_{\cT^o}:\cT^o\to \cS_{(2)}$ is a finite morphism for some open subset $\cT^o$ of $\cT$.
        \end{itemize}
         Let $((\cX_{(2)},\cB_{(2)}),\cA_{(2)})\rightarrow \cS_{(2)}$ be the corresponding base change. Then,  the pullback of $(\cX_{(2)},\cB_{(2)}+\alpha\cA_{(2)})\rightarrow \cS_{(2)}$ via $i$ defines a locally stable morphism  $(\cX_{\cT^o},\cB_{\cT^o}+\alpha\cA_{\cT^o})\rightarrow \cT^o$.

        By \cite[Lemma 4]{kollarModuliPolarizedCalabi2020}, there exists a generically finite cover $\overline{\cT}^o\rightarrow \cT^o$ and a  compactification $\overline{\cT}^o\hookrightarrow \cS^!$ such that the pullback of $(\cX_{\cT^o},\cB_{\cT^o}+\alpha\cA_{\cT^o})\rightarrow \cT^o$ on $\overline{\cT}^o$ extends to a locally stable morphism 
        $(\cX^!,\cB^!+\alpha\cA^!)\rightarrow \cS^!$.
        
        By Lemma \ref{lem: finite modification}, after replacing $\cS^!$  with a generically finite cover from a smooth projective variety and $(\cX^!,\cB^!+\alpha\cA^!)\rightarrow \cS^!$  with the corresponding base change,
        we may assume that there exists a birational map $\cS^*\dashrightarrow \cV$ such that $\pi:\cS^!\to \cS^*$  is a finite cover.
           Replacing $\cS_{(2)}$ by an open subset and shrinking $\overline{\cS}_{(2)}$ accordingly, we may assume that $\gamma:\cS_{(2)}\to \cS^*$ is a morphism.
        
After replacing $\overline{\cS}_{(2)}$ by a finite cover, we may assume that $\overline{\cS}_{(2)}\rightarrow \cS^!$ is a dominant morphism. In this case, we have an isomorphism
\[
(\cX_{(2)}, \cB_{(2)}) \times_{\cS_{(2)}} \overline{\cS}_{(2)} \cong (\cX^!, \cB^!) \times_{\cS^!} \overline{\cS}_{(2)}.
\]

Next, after replacing $\overline{\cS}_{(2)}$ by another finite cover, we may assume that $\overline{\cS}_{(2)} \rightarrow \cS_{(2)}$ is a Galois cover with Galois group $G$. Replacing $\cS_{(2)}$ by an open subset and shrinking $\overline{\cS}_{(2)}$ accordingly, we may assume that $\overline{\cS}_{(2)} \rightarrow \cS_{(2)}$ is an étale Galois cover. Therefore, $\overline{\cS}_{(2)}$ is smooth.

  	\item In this step, we construct new polarizations $\cL_{(2)}$ and $\cL^!$ on $\cX_{(2)}$  and $\cX^!$ respectively that satisfy (6).

        Consider the following diagram:
\[
\xymatrix{
\cX_{(2)} \ar[d]
  & \cX_{(2)}\times_{\cS_{(2)}} \overline{\cS}_{(2)} \overset{\Psi}{\cong} \cX^!\times_{\cS^!}\overline{\cS}_{(2)}
    \ar[l]_-{\tau_{\cX}}
    \ar[r]^-{\rho_{\cX}}
    \ar[d]
  & \cX^! \ar[d] \\
\cS_{(2)}\ar@/_1pc/[rrr]_{\gamma}
  & \overline{\cS}_{(2)}
    \ar[l]_{\tau}
    \ar[r]^{\rho}
  & \cS^!\ar[r]^\pi & \cS^*
}
\]
Recall that $\overline{\cS}_{(2)} \rightarrow \cS_{(2)}$ is an \'etale Galois cover with Galois group $G$.  
For each $g \in G$ acting on $\overline{\cS}_{(2)}$, let 
\[
\sigma'_g \colon \cX_{(2)} \times_{\cS_{(2)}} \overline{\cS}_{(2)}
\longrightarrow 
\cX_{(2)} \times_{\cS_{(2)}} \overline{\cS}_{(2)}
\]
be the morphism induced by the base change $\text{id}_{\cX_{(2)}} \times_{\cS_{(2)}} g$, and let
\[
\sigma_g \colon \cX^! \times_{\cS^!} \overline{\cS}_{(2)}
\longrightarrow 
\cX^! \times_{\cS^!} \overline{\cS}_{(2)}
\]
be the morphism $\Psi \circ \sigma'_g \circ \Psi^{-1}$.  
Then the action of $G$ on  $\cX^! \times_{\cS^!} \overline{\cS}_{(2)}$ is $G$-equivariant with respect to the projection
$\cX^! \times_{\cS^!} \overline{\cS}_{(2)} \rightarrow \overline{\cS}_{(2)}$, and
$\cX^! \times_{\cS^!} \overline{\cS}_{(2)} \rightarrow \cX_{(2)}$
is also an \'etale Galois cover with Galois group $G$.
Let   \[\overline{\cL}_{(2)}:=\sum_{g\in G}g^* \rho_{\cX}^* \cA^!,\]
        then $\overline{\cL}_{(2)}$ is $G$-invariant, and hence there exists an effective $\mathbb{Q}$-Cartier integral divisor $\cL_{(2)}$ on $\cX_{(2)}$ such that $\overline{\cL}_{(2)}=\tau_{\cX}^* \cL_{(2)}$.

Denote by $\cS^!_{(2)}$ and $\cS^*_{(2)}$ the images of $\overline{\cS}_{(2)}$ in $\cS^!$ and $\cS^*$, respectively, and let $\cX_{\cS^!_{(2)}}$ be the base change of $\cX^!$ over $\cS^!_{(2)}$.
Let $v \in \cS^*_{(2)}$ be a closed point. Then $\pi^{-1}(v)$ is a finite set of points since $\pi:\cS^!\to \cS^*$ is finite. Let $s \in \pi^{-1}(v)$ be a closed point.
Set $S := \gamma^{-1}(v)$ and $\overline{S} := \tau^{-1}(S)$.
Let $(\cX_s,\cB_s)\to s$, $(\cX_S,\cB_S)\to S$, and $(\cX_{\overline{S}},\cB_{\overline{S}})\to \overline{S}$ be the families obtained by base change, where \[(\cX_{\overline{S}},\cB_{\overline{S}}):=(\cX_S,\cB_S) \times_S \overline{S} \overset{\Psi}{\cong} (\cX_s,\cB_s) \times \overline{S}.\]
Now the group $G$ acts on $(\cX_s,\cB_s) \times \overline{S}$, and this action is $G$-equivariant with respect to the projection $(\cX_s,\cB_s)\times \overline{S} \to \overline{S}$.

By Proposition~\ref{:prop:diagonal}, there exists a Galois cover
$\widetilde{\tau}\colon \widetilde{S}\to S$ with Galois group $H$,
which also acts on $\cX_s$, such that
\[
(\cX_{\widetilde{S}},\cB_{\widetilde{S}})
:= (\cX_S,\cB_S) \times_S \widetilde{S}
\cong (\cX_s,\cB_s) \times \widetilde{S}.
\]
Moreover,
\[
(\cX_S,\cB_S)\simeq \bigl((\cX_s,\cB_s)\times \widetilde{S}\bigr)/H,
\]
where $H$ acts diagonally on $(\cX_s,\cB_s)\times \widetilde{S}$.
After replacing $\widetilde{S}$ by a higher Galois cover, we may assume
that $\widetilde{\tau}$ factors through $\overline{S}$.
We have the following diagram:
\[
\xymatrix{
&&(\cX_{\widetilde{S}},\cB_{\widetilde{S}})\ar[dd]\ar[ld]^{\widetilde{\pi}_{\cX}}\ar[rd]^{\widetilde{\rho}_{\cX}}\ar[lld]_{\widetilde{\tau}_{\cX}}
\\
(\cX_S,\cB_S) \ar[dd]
  & (\cX_{\overline{S}},\cB_{\overline{S}})
    \ar[l]_-{\tau_{\cX}}
    \ar[rr]^-{\rho_{\cX}\hspace{2em}}
    \ar[dd]
  & &\bigcup_{s\in\pi^{-1}(v)}(\cX_s,\cB_s) =:(\cX_v,\cB_v) \ar[dd] \\
&&\widetilde{S}\ar[ld]^{\widetilde{\pi}}\ar[rd]^{\widetilde{\rho}}\ar[lld]_{\widetilde{\tau}\hspace{2em}}
  \\
S
  & \overline{S}
    \ar[l]_{\tau}
    \ar[rr]^{\rho}
  && \pi^{-1}(v)
}
\]
Denote $\cA^!_v := \cA^!|_{\cX_v}$. 
Then, for each $g \in G$, there exists $h \in H$, which is a lift of $g$, such that
\[
\widetilde{\pi}_{\cX}^* g^* \rho_{\cX}^* \cA^!_v
= h^* \widetilde{\pi}_{\cX}^* \rho_{\cX}^* \cA^!_v
= h^* \widetilde{\rho}_{\cX}^* \cA^!_v
= \widetilde{\rho}_{\cX}^* h^* \cA^!_v,
\]
which is vertical over $\cX_v$ via $\widetilde{\rho}_{\cX}$.
Here the last equality follows from the fact that $H$ acts diagonally on $(\cX_s,\cB_s) \times \widetilde{S}$. 
Therefore, $(\widetilde{\pi}_{\cX})_* \widetilde{\pi}_{\cX}^* g^* \rho_{\cX}^* \cA^!_v$ is also vertical over $\cX_v$ via $\rho_{\cX}$. 
Hence,
\[
\overline{\cL}_{(2),v} := \overline{\cL}_{(2)}|_{\cX_{\overline{S}}} 
= \sum_{g \in G} g^* \rho_{\cX}^* \cA^!_v
\]
is vertical over $\cX_v$ via $\rho_{\cX}$.
Since $v$ can be any closed point of $\cS^*_{(2)}$, it follows that $\overline{\cL}_{(2)}$ is vertical over $\cX_{\cS^!_{(2)}}$.
Let
\[
\overline{\cS}_{(2)} \longrightarrow \overline{\cS}^!_{(2)} \longrightarrow \cS^!_{(2)}
\]
be the Stein factorization of $\rho\colon \overline{\cS}_{(2)} \to \cS^!_{(2)}$.
Replacing $\cS^!_{(2)}$ by $\overline{\cS}^!_{(2)}$, we may assume that $\rho$ has connected fibers. After shrinking $\cS^!_{(2)}$, and then replacing $\cX_{\cS^!_{(2)}}$ accordingly, 
we may assume that
\[
\rho_{\cX}\colon \cX^! \times_{\cS^!} \overline{\cS}_{(2)} \to \cX_{\cS^!_{(2)}}
\]
is equidimensional with integral fibers.
Therefore, there exists a divisor $\cL^!_{(2)}$ on $\cX_{\cS^!_{(2)}}$ such that
$\overline{\cL}_{(2)} = \rho_{\cX}^* \cL^!_{(2)}$.
Let $\cL^!$ be the closure of $\cL^!_{(2)}$ on $\cX^!$,
then we have
\[
\tau_{\cX}^* \cL_{(2)} = \rho_{\cX}^* \cL^!.
\]
        
Note that $i$ is a morphism on $\cS^!_{(2)}$, and let $\overline{\cT}$ be the preimage of
$i(\cS^!_{(2)})$ in $\overline{\cS}$.
Since $\cA^! = i^*\cA_{(2)}$, we have
\[
\rho_{\cX}^*\cA^!|_{\overline{\cT}}
= \rho_{\cX}^* i^*\cA_{(2)}|_{\overline{\cT}}
= \tau_{\cX}^*\cA_{(2)}|_{\overline{\cT}}.
\]
By Lemma~\ref{lem:num trivial}, it follows that
\[
(\rho_{\cX}^*\cA^!)_s \equiv (\tau_{\cX}^*\cA_{(2)})_s
\quad \text{for all } s \in \overline{\cS}_{(2)}.
\]
Moreover, since $\tau_{\cX}$ is the quotient morphism by $G$, the divisor
$\tau_{\cX}^*\cA_{(2)}$ is $G$-invariant. Hence,
\[
\tau_{\cX}^*\cA_{(2)}
= \frac{1}{|G|} \sum_{g\in G} g^*\tau_{\cX}^*\cA_{(2)}.
\]
Therefore, for any $s \in \overline{\cS}_{(2)}$, we obtain
\[
(\tau_{\cX}^*\cL_{(2)})_s
= (\overline{\cL}_{(2)})_s
\equiv \biggl(\sum_{g\in G} g^* \rho_{\cX}^* \cA^!\biggr)_s
\equiv \biggl(\sum_{g\in G} g^*\tau_{\cX}^*\cA_{(2)}\biggr)_s
= |G|(\tau_{\cX}^*\cA_{(2)})_s.
\]
Since the morphism $\overline{\cS}_{(2)} \to \cS_{(2)}$ is surjective, we conclude that
\[
(\cL_{(2)})_s \equiv |G|(\cA_{(2)})_s
\quad \text{for all } s \in \cS_{(2)}.
\]

   \item In this step, we verify  (7)--(9).
   
   By the construction in the previous step, the general fiber of $((\cX_{(2)},\cB_{(2)}),\cL_{(2)}) \to \cS_{(2)}$ is a $(d,\Phi,|G|^d v)$-polarized log Calabi--Yau pair.
After replacing $\cS_{(2)}$ with an open subset and decreasing $\alpha$, we may assume that $(\cX_{(2)},\cB_{(2)}+\alpha\cL_{(2)}) \to \cS_{(2)}$ is locally stable.
Applying \cite[Lemma 4]{kollarModuliPolarizedCalabi2020} to an open subset $(\cS^!)^o \subset \cS^!$ over which $(\cX^!,\cB^!+\alpha\cL^!) \to \cS^!$ is locally stable, and then replacing $\overline{\cS}_{(2)}$ with a finite cover, we may assume that $(\cX^!,\cB^!+\alpha\cL^!) \to \cS^!$ is locally stable and that $\overline{\cS}_{(2)} \to \cS^!$ is dominant.
In this process, we may have lost the local stability of 
$(\cX^!,\cB^!+\alpha\cA^!) \to \cS^!$ and the finiteness of 
$\pi:\cS^!\to \cS^*$; however, these properties will not be used later.  
Therefore, (7) holds.

      For (8),  let $\cH\geq 0$ be a very ample divisor on $\cS^*$. Because $\cS^!\rightarrow \cS^*$ is  generically finite, $\pi^*\cH$ is a big divisor on $\cS^!$. Then we can choose $\cH$ general such that
        \begin{itemize}
            \item $\pi$ is \'etale and Galois over $\cS^*\setminus \Supp(\cH)$, and
            \item every fiber of $((\cX^!,\cB^!),\cL^!)\rightarrow \cS^!$ over $\cS^!\setminus \Supp(\pi^*\cH)$ is a klt $(d,\Phi,v)$-polarized log Calabi--Yau pair.
        \end{itemize}
        
       Now, we address (9). Since $\cS^!$ is smooth and  $(\cX^!,\cB^!)\rightarrow \cS^!$ is locally stable of maximal variation, by Proposition \ref{moduli part descends on locally stable morphisms}, the moduli \textbf{b}-divisor $\bm{\cM}^!$ descends to a  nef and big divisor $\bm{\cM}^!_{\cS^!}$  on $\cS^!$.  We can choose a general member $0\leq \cM^!\in |\bm{\cM}^!_{\cS^!}|_{\mathbb{Q}}$ such that $l\cM^!$ is Cartier and
           $\pi^*\cH \leq l\cM^!$ for some $l\in \bN$ depending only on $(d,\Phi,v)$.

      Let  $\cS_{(3)}$ be the open subset of $\cS_{(2)}$ such that $\gamma^{-1}(\pi(\Supp(\cM^!)))\cap \cS_{(3)}=\emptyset$  and $\overline{\cS}_{(3)}$ be the preimage of $\cS_{(3)}$. 
         Let $((\cX_{(3)},\cB_{(3)}),\cA_{(3)}, \cL_{(3)})\rightarrow \cS_{(3)}$  and $\overline{\cX}_{(3)}\rightarrow \overline{\cS}_{(3)}$  be the corresponding base change.

  	\item In this step,  we construct $\cS$ and verify (10) and (11).
         
Note that  $\cS_{(3)}$ is an open subset of $\cS_{(1)}$, and the moduli map 
$\phi \colon U \to \cS_{(1)}$ obtained in Step~1 may map into 
$\cS_{(1\setminus3)} := \cS_{(1)} \setminus \cS_{(3)}$. Since the constructions in Steps~2--5 depend only on the given family, 
they apply equally to the induced family 
$((\cX_{(1\setminus3)},\cB_{(1\setminus3)}),\cA_{(1\setminus3)}) \to \cS_{(1\setminus3)}$ over the complement locus.
By Noetherian induction, after repeating this process finitely many times, we obtain a locally closed decomposition 
$\cS=\bigcup \cS_j$ of $\cS_{(1)}$. 
Let $\overline{\cS}_j$, $\cS^!_j$, and $\cS^*_j$ be the schemes constructed in Steps~2--5 corresponding to $\cS_j$.
Let $\overline{\cS}=\bigcup \overline{\cS}_j$, and replace $\cS^!$ and $\cS^*$ by 
$\bigcup \cS^!_j$ and $\bigcup \cS^*_j$, respectively.
Then we obtain the following diagram: 
  \[\xymatrix{
	((\cX,\cB),\cA,\cL )\ar[d] &\overline{\cX}\ar[r]^(.3){\rho_{\cX}} \ar[l]_(.3){\tau_{\cX}} \ar[d] & ((\cX^!,\cB^!),\cL^!) \ar[d]  & \\
	\cS\ar@/_1pc/[rrr]_{\gamma}  & \overline{\cS}\ar[r]^{\rho} \ar[l]_{\tau} \ar[l] & \cS^! \ar[r]^\pi & \cS^*,
}\]
        that satisfies the requirements (1)--(9), where $\overline{\cX}$ is the  common resolution of the main components of $\cX \times_\cS \overline{\cS}$ and $\cX^! \times_{\cS^!} \overline{\cS}$.
       
Recall  that  in Step 1, $((X_U,B_U),A_U)\to U$ is isomorphic to the pullback of $((\cX_{(1)},\cB_{(1)}),\cA_{(1)})\to \cS_{(1)}$ via the moduli map $\phi \colon U \to \cS_{(1)}$. After replacing $U$ by an open subset, we may assume that $\phi$ induces a morphism $\phi \colon U \to \cS$. Then $((X_U,B_U),A_U)\to U$ is isomorphic to the pullback of $((\cX,\cB),\cA)\to \cS$ via $U \to \cS$. Therefore, (10) follows.

Finally, we deal with  (11). Suppose that $\gamma\circ \phi$ extends to a morphism $\psi:Z\rightarrow \cS^*$. By the construction of $\cS_{(3)}$, we have $\gamma^{-1}(\pi(\Supp(\cM^!)))=\emptyset$. Since $\psi|_U$ factors through $\cS$, we have $\psi(Z)\not \subset \pi(\Supp(\cM^!))$.
    \end{enumerate}
    \end{proof}
\end{thm}

\subsection{Boundedness of moduli map}
In this subsection, we construct a birational model $(W,D)$ of $Z$ such that $(W,D)$ is log bounded and the map $W \dashrightarrow \cS^*$ induced by the moduli map $Z \dashrightarrow \cS^*$ is a bounded morphism.

\begin{thm}\label{moduli map is bounded}
    Let $d\in \bN$, $v,r,\epsilon\in \mathbb{Q}^{>0}$, and $\Phi\subset[0,1]\cap \bQ$ be a finite set. 
       Let  $f:((X, B),A)\rightarrow (Z, H)$ be a weak $(d,\Phi,v,r,\epsilon)$-polarized log Calabi--Yau fibration.  
    Let \[\xymatrix{((X,B),A)  \ar@{-->}[r] \ar[d]&
     	((\cX,\cB),\cA,\cL )\ar[d] &\overline{\cX}\ar[r]^(.3){\rho_{\cX}} \ar[l]_(.3){\tau_{\cX}} \ar[d] & ((\cX^!,\cB^!),\cL^! )\ar[d]  & \\
     	(Z,H)\ar@{-->}[r]_{\phi}&  \cS\ar@/_1pc/[rrr]_{\gamma}  & \overline{\cS}\ar[r]^{\rho} \ar[l]_{\tau} \ar[l] & (\cS^!,\cM^!) \ar[r]^\pi & (\cS^*, \cH)
     }\]
   be  the commutative diagram obtained in Theorem \ref{diagram of the universal Calabi--Yau fibration}.  
   Then there exists a birational morphism $h:W\to Z$ from a normal projective variety $W$ and a reduced divisor $D$ on $W$ such that
   \begin{enumerate}
   \item the induced rational map $\psi_W:W\dashrightarrow \cS^*$ is a morphism,

      \item $D\supset \Supp(h_*^{-1}B_Z+E+\psi_W^*\cH)$, where $B_Z$ is the discriminant $\bQ$-divisor with respect to $f:(X,B)\to Z$, and $E$ is the sum of reduced exceptional divisors of $h$, and 
      \item $K_W+D-h^*H$ is big.
   \end{enumerate}
   Moreover, the set of such pairs $(W,D)$ is log bounded, and the morphism $\psi_W:W\to \cS^*$ is bounded.

    \begin{proof}
    \begin{enumerate}
     [label=\textsl{Step} \arabic{enumi}., wide=13pt, itemsep=13pt]
     
\item  In this step,  we construct a birational model $W$ of $Z$ such that $W\dashrightarrow Z$ and $W\dashrightarrow \cS^*$ are morphisms.
     	
Since the coefficients of $B$ belong to the finite set $\Phi$, by \cite[Lemma 6.7]{birkarVariationsGeneralisedPairs2022}, there exists $q\in\bN$ and $\delta\in\bQ^{>0}$ depending only on $d, \Phi, v, \epsilon$, such that we can write the canonical bundle formula
\[
q(K_X + B) \sim qf^*(K_Z + B_Z + \M_Z)
\]
such that $(Z, B_Z,\M)$ is a $\delta$-lc generalized pair, $qB_Z$ is integral, and $q\M_{Z'}$ is Cartier, where $\M_{Z'}$ is the moduli divisor on any sufficiently high resolution $Z' \to Z$. In particular, the coefficients of $B_Z$ belong to a fixed finite set $\cI$. Replacing $\cI$ by $\cI \cup \{1 - \frac{\delta}{2} \}$, we may assume that $1 - \frac{\delta}{2} \in \cI$.

      Let $g: Z' \to Z$ be a log resolution of $(Z, B_Z)$ such that the moduli \textbf{b}-divisor $\M$ of $f$ descends to $Z'$, and the rational map $\gamma \circ \phi: Z \dashrightarrow \cS^*$ extends to a morphism $\psi': Z' \to \cS^*$. In particular, $\M_{Z'}$ is nef.
Define
\[
B_{Z'} := g_*^{-1} B_Z + (1 - \frac{\delta}{2} ) E_{Z'},
\]
where $E_{Z'}$ is the sum of all reduced $g$-exceptional divisors. Since $(Z, B_Z,\M)$ is $\delta$-lc, it follows that
\[
K_{Z'} + B_{Z'} + \M_{Z'} - g^*(K_Z + B_Z + \M_Z)
\]
is effective and has the same support as $E_{Z'}$. Moreover, the coefficients of $B_{Z'}$ belong to the finite set $\cI$.

  By the boundedness of the length of extremal rays, $K_Z + B_Z + \M_Z + 3dH$ is ample. Since 
\[
K_{Z'} + B_{Z'} + \M_{Z'} - g^*(K_Z + B_Z + \M_Z)
\]
is effective, it follows that
\[
K_{Z'} + B_{Z'} + \M_{Z'} + 3d g^*H + 3d \psi'^*\cH
\]
is big. Consider $(Z', B_{Z'},\M+3d \overline{g^*H} + 3d \overline{\psi'^*\cH})$ as a  $\frac{\delta}{2}$-lc generalized pair with nef part $\M+ 3d \overline{g^*H} + 3d\overline{\psi'^*\cH}$. By \cite[Lemma 4.4]{birkarEffectivityIitakaFibrations2016}, the divisor 
\[
K_{Z'} + B_{Z'} + \M_{Z'} + 3d g^*H + 3d \psi'^*\cH
\]
admits a generalized log canonical model $Z' \dashrightarrow W$. In particular, the birational map $Z' \dashrightarrow W$ does not extract any divisor.
Since $d \geq \dim Z'$, the boundedness of the length of extremal rays ensures that the birational map $Z' \dashrightarrow W$ is automatically over both $Z$ and $\cS^*$, inducing morphisms $h: W \to Z$ and $\psi_W: W \to \cS^*$. 
Let $B_W$ be the pushforward of $B_{Z'}$. Then 
\[
\Supp(B_W) \supset \Supp(h_*^{-1} B_Z + E),
\]
where $E$ is the sum of reduced exceptional divisors of $h$.
 \[\xymatrix{\overline{Z}\ar[d] \ar[r]_{\pi_{Z'}} \ar[rd]_{\overline{\psi}}&Z' \ar[r]_{g} \ar@{-->}@/^1pc/[rr]
	 \ar[rd]_{\psi'} &Z \ar@{-->}[d]^{\psi}&W \ar[l]^h\ar[ld]^{\psi_W}& \\
	\overline{\cS}\ar[r]&\cS^!\ar[r]_{\pi}&   \cS^*}\]
 
\item In this step, we show that $l\M_{Z'}-\psi'^*\cH$ is pseudo-effective for some  $l\in \bN$  depending only on $(d,\Phi,v)$.

    Let $\pi_{Z'}:\overline{Z}\rightarrow Z'$ be a generically finite cover from a smooth variety $\overline{Z}$ such that
    \begin{itemize}
        \item  $\psi':Z'\rightarrow \cS^*$ lifts to a morphism $\overline{\psi}:\overline{Z}\rightarrow \cS^!$ which factors through
$\overline{\cS}$, and
        \item the generic fiber of $(X,B)\times_Z \overline{Z}\rightarrow \overline{Z}$ is isomorphic to the generic fiber of $(\cX^!,\cB^!)\times_{\cS^!}\overline{Z}\rightarrow \overline{Z}$.
    \end{itemize}
Since $(\cX^!,\cB^!) \to \cS^!$ is locally stable over the smooth base $\cS^!$, the morphism 
\[
(\cX^!,\cB^!) \times_{\cS^!} \overline{Z} \to \overline{Z}
\]
is also locally stable over the smooth base $\overline{Z}$. By parts (2) and (3) of Proposition \ref{moduli part descends on locally stable morphisms}, the moduli \textbf{b}-divisor $\overline{\M}$ of $(\cX^!,\cB^!) \times_{\cS^!} \overline{Z} \to \overline{Z}$ descends to $\overline{Z}$ and satisfies
\[
\overline{\psi}^*  \cM^! \sim_{\mathbb{Q}} \overline{\M}_{\overline{Z}}.
\]
 Let $\widetilde{\M}$ be the moduli $\mathbf{b}$-divisor of
$(X,B)\times_Z \overline{Z} \to \overline{Z}$.
Since the generic fiber of
$(X,B)\times_Z \overline{Z} \to \overline{Z}$
is isomorphic to the generic fiber of
$(\cX^!,\cB^!)\times_{\cS^!}\overline{Z} \to \overline{Z}$,
it follows from \cite[Lemma~3.5]{birkarAntipluricanonicalSystemsFano2019}
that $\widetilde{\M}_{\overline{Z}} \simeq \overline{\M}_{\overline{Z}}$. We may still denote the moduli \textbf{b}-divisor of $(X,B) \times_Z \overline{Z} \to \overline{Z}$ by $\overline{\M}$ without confusion, and it descends to $\overline{Z}$.

Since $\pi_{Z'}: \overline{Z} \to Z'$ is a generically finite cover and $\M$ descends to $Z'$, it follows from Proposition \ref{amb05prop3.1} that
\[
\overline{\M}_{\overline{Z}} \sim_{\mathbb{Q}} \pi_{Z'}^* \M_{Z'}.
\]
By parts (9) and (11) of Theorem \ref{diagram of the universal Calabi--Yau fibration}, there exists $l \in \bN$ depending only on $(d,\Phi,v)$ such that 
\[
\pi^* \cH \leq l \cM^!,
\]
and $\psi'(Z') \not\subset \pi(\Supp(\cM^!))$. Then, we obtain
\[
\pi_{Z'}^* l \M_{Z'} \sim_{\mathbb{Q}} l \overline{\M}_{\overline{Z}} \sim_{\mathbb{Q}} \overline{\psi}^* l \cM^! \geq \overline{\psi}^* \pi^* \cH \sim_{\mathbb{Q}} \pi_{Z'}^* \psi'^* \cH.
\]
Therefore, $l \M_{Z'} - \psi'^* \cH$ is pseudo-effective.

\item In this step, we show that $\vol(K_W + B_W + \M_W + 3d h^*H + 3d \psi_W^*\cH)$ is bounded from above.

Since $Z'\dashrightarrow W$ is the generalized log canonical model of 
\[
K_{Z'}+B_{Z'}+\M_{Z'}+ 3d g^*H +3d \psi'^*\cH,
\]
and $l\M_{Z'}-\psi'^*\cH$ is pseudo-effective by Step 2, we have
\begin{equation}
	\begin{aligned}
		&\vol(K_W+B_W+\M_W+3d h^*H+3d \psi_W^*\cH) \\
		\leq & \vol(K_{Z'}+B_{Z'}+\M_{Z'}+ 3d g^*H+3d \psi'^*\cH) \\
		\leq & \vol(K_{Z'}+B_{Z'}+(3dl+1)\M_{Z'}+ 3d g^*H).
	\end{aligned}
\end{equation}

By Step 1, $q\M_{Z'}$ is Cartier. Hence, replacing $l$ with $ql$, we may assume that 
\[
l(\M_{Z'}+ 3d g^*H)
\]
is Cartier. 

Since the coefficients of $B_{Z'}$ belong to the finite set $\cI$, by \cite[Theorem 8.1]{birkarEffectivityIitakaFibrations2016}, there exists $e\in (0,1)$ depending only on $d, \cI, l$ such that 
\[
K_{Z'}+B_{Z'}+e\M_{Z'}+3d g^*H
\]
is big. Choose $\lambda\in (0,1)$ such that
\[
\lambda e +(1-\lambda)(3dl+1) = 1.
\]
Then, we have
\begin{equation}	\nonumber
	\begin{aligned}
		&\lambda(K_{Z'}+B_{Z'}+e\M_{Z'}+ 3d g^*H)\\
		+&(1-\lambda)(K_{Z'}+B_{Z'}+(3dl+1)\M_{Z'}+ 3d g^*H)\\
		=&K_{Z'}+B_{Z'}+\M_{Z'}+ 3d g^*H.
	\end{aligned}
\end{equation}
Thus, we obtain
\begin{equation}
	\begin{aligned}
		&\vol(K_{Z'}+B_{Z'}+(3dl+1)\M_{Z'}+ 3d g^*H) \\
		\leq & \frac{1}{(1-\lambda)^d} \vol(K_{Z'}+B_{Z'}+\M_{Z'}+ 3d g^*H).
	\end{aligned}
\end{equation}

By the definition of the weak polarized log Calabi--Yau fibration, $H-(K_Z+B_Z+\M_Z)$ is pseudo-effective. Since 
\[
K_{Z'}+B_{Z'}+\M_{Z'}-g^*(K_Z+B_Z+\M_Z)
\]
is effective and exceptional over $Z$, and $\dim Z\leq d$, we have 
\begin{equation}
	\begin{aligned}
		&\vol(K_{Z'}+B_{Z'}+\M_{Z'}+3d g^*H) \\
		= &\vol(K_Z+B_Z+\M_Z+3d H) \\
		\leq & (3d+1)^d H^{\dim Z} \\
		\leq & (3d+1)^d r.
	\end{aligned}
\end{equation}

Combining equations (3.1)--(3.3), we conclude that
\[
\vol(K_W+B_W+\M_W+3d h^*H+3d \psi_W^*\cH) \leq \frac{(3d+1)^d}{(1-\lambda)^d} r.
\]

\item In this step, we show that  $W$ belongs to a bounded family. Moreover,   the morphism $\psi_W: W \to \cS^*$ is also bounded.

Since $Z'$ is smooth and $\M$ descends to $Z'$, after replacing $Z'$ with a higher model so that $Z' \to W$ is a morphism,  
       $(W,B_W,\M_W+3d \overline{h^*H}+3d \overline{\psi_W^*\cH})$ is a $\frac{\delta}{2}$-lc generalized pair with nef part $\M+ 3d \overline{g^*H} +3d \overline{\psi'^*\cH}$, satisfying the following conditions:
\begin{itemize}
    \item The coefficients of $B_W$ belong to the finite set $\cI$,
    \item $l(\M_{Z'}+ 3d g^*H +3d \psi'^*\cH)$ is Cartier, and
    \item $K_W+B_W+\M_W+3d h^*H+3d \psi_W^*\cH$ is ample with bounded volume,
\end{itemize}
it follows from \cite[Lemma 6.6]{birkarVariationsGeneralisedPairs2022} that  $(W,B_W,\M+3d \overline{h^*H}+3d \overline{\psi_W^*\cH})$ is log bounded.
In particular, there exists $m \in \mathbb{N}$ depending only on $d, \Phi, v, \epsilon, r$ such that 
\[
H_W := m(K_W + B_W + \M_W + 3d h^*H + 3d \psi_W^*\cH)
\]
is very ample and 
         $ \vol(H_W)$ is bounded from above.
          
Let $\Gamma_{\psi_{W}}\subset W\times \cS^*$ be the graph of the morphism $\psi_{W}: W\rightarrow \cS^*$. 
Since $H_W$ and $\cH$ are very ample, the product $W\times \cS^*$ can be embedded into a projective space via the Segre embedding $\bP^{N_1}\times \bP^{N_2}\subset \bP^{N}$. 
Moreover, the restriction of $\cO_{\bP^N}(1)$ to $\Gamma_{\psi_{W}}\cong W$ is given by $H_{W}+\psi_{W}^*\cH$.
Note that
\[
\begin{aligned}
\vol(H_W+\psi_W^*\cH)
&= \vol\Bigl(m(K_W+B_W+\M_W+3dh^*H) + (3dm+1)\psi_W^*\cH\Bigr) \\
&\leq \vol\Bigl((m+1)(K_W+B_W+\M_W+3dh^*H) + (3dm+1)\psi_W^*\cH\Bigr) \\
&\leq \vol\Bigl((m+1)\bigl(K_W+B_W+\M_W+3dh^*H+3d\psi_W^*\cH\bigr)\Bigr),
\end{aligned}
\]
which is bounded from above by Step~3. The first inequality follows from the fact that 
$K_W+B_W+\M_W+3dh^*H$ is big (see Step~1). Consequently, $\Gamma_{\psi_W}$ is bounded.
Since every morphism $\psi_{W}: W\rightarrow \cS^*$ is determined by its graph $\Gamma_{\psi_{W}}$, it follows that the morphism $\psi_{W}: W\rightarrow \cS^*$ is bounded.

\item  In this step,  we define a reduced divisor $D$ on $W$ and conclude the proof.

Since $h^*H$ and $\psi_W^*\cH$ are base point free, it follows that $3d(H_W+h^*H+\psi_W^*\cH)$ is very ample. 
We can find  a general reduced divisor
\[
0\leq D\in |3d(H_W+h^*H+\psi_W^*\cH)|
\]
such that $D$ contains the support of $h_*^{-1}B_Z+E+\psi_{W}^*\cH$, where $E$ is the sum of the reduced exceptional divisors of $h$. 
Moreover, by the boundedness of the length of extremal rays, the divisor $K_W+D-h^*H$ is big. 

Let $d'=\dim W$, by \cite[Lemma~3.2]{haconBirationalAutomorphismsVarieties2013}, we have
\[
\begin{aligned}
((4d'+2)H_W)^{d'-1}\cdot D
&\leq 2^{d'}\vol(K_W+D+(4d'+2)H_W) \\
&\leq 2^{d'}\vol(K_W+B_W+\M_W+D+(4d'+2)H_W) \\
&= \alpha^{d'}
\vol(K_W+B_W+\M_W+3d h^*H+3d \psi_W^*\cH),
\end{aligned}
\]
which is bounded from above by Step~3, where
$\alpha = 2+2m(4d'+3d+2)$.
Hence, $(W,D)$ is log bounded, completing the proof.

        \end{enumerate}
    \end{proof}
\end{thm}
 Let $d\in \bN$, $v,r,\epsilon\in \mathbb{Q}^{>0}$, and $\Phi\subset[0,1]\cap \bQ$ be a finite set. Let $\alpha$ be the positive rational number defined in Theorem \ref{diagram of the universal Calabi--Yau fibration}.
    Let $f:((X,B),A)\rightarrow (Z,H)$  be a weak $(d,\Phi,v,r,\epsilon)$-polarized log Calabi--Yau fibration.
 By Theorem \ref{moduli map is bounded}, there exists a family of pairs $(\cW, \cD) \to T$ over a finite type scheme $T$, and a projective morphism $\Theta: \cW \to \cS^*$  such that $(W, D) \cong (\cW_t, \cD_{t})$,  and $\psi_W:W\to \cS^*$ is equivalent to $\Theta_t:\cW_t\to \cS^*$ for some closed point $t \in T$. 
 
Let $\overline{\cW}$ be the normalization of the main component of $\cW \times_{\cS^*} \cS^!$, and let $\overline{\cD}_{\overline{\cW}}$ denote the preimage of $\cD$ via the map $\overline{\cW} \to \cW$. 
After replacing $(\overline{\cW}, \overline{\cD}_{\overline{\cW}})$ with its log resolution and passing to a stratification of $T$, we may assume that the pair $(\overline{\cW}, \overline{\cD}_{\overline{\cW}})$ is log smooth over $T$.
Let $\overline{\Theta}: (\overline{\cW},\overline{\cD}_{\overline{\cW}})\rightarrow \cS^!$ be the induced morphism, and let $\overline{\cF}:((\overline{\cX}_{\overline{\cW}},\overline{\cB}_{\overline{\cW}}),\overline{\cL}_{\overline{\cW}})\rightarrow \overline{\cW}$ be the pullback of $((\cX^!,\cB^!),\cL^!)\rightarrow \cS^!$ via $\overline{\Theta}$.  
We have the following commutative diagram:
   \[\xymatrix{
    & (\overline{\cX}_{\overline{\cW}},\overline{\cB}_{\overline{\cW}}),\overline{\cL}_{\overline{\cW}} \ar[d] \ar[r]^-{\overline{\cF}} & (\overline{\cW},\overline{\cD}_{\overline{\cW}}) \ar[d]^{\overline{\Theta}} \ar[r] & (\cW,\cD) \ar[d] ^\Theta\ar[r]&T\\
    & (\cX^!,\cB^!),\cL^! \ar[r] & \cS^!  \ar[r] ^\pi& \cS^*   
}\]

\begin{lem}\label{volume is bounded up to a finite cover}
  There exists $w\in \bN$ depending only on $d,\Phi,v,r,\epsilon$ such that 
    \[\vol(K_{\overline{\cX}_{\overline{\cW}_t}}+\overline{\cB}_{\overline{\cW}_t}+\alpha\overline{\cL}_{\overline{\cW}_t}+\overline{\cF}_t^*\overline{\cD}_{\overline{\cW}_t})\leq w\]
    for every closed point $t\in T$.
    \begin{proof}
Since $(\cX^!,\cB^!+\alpha\cL^!)\rightarrow \cS^!$ is locally stable, it follows that 
\[
(\overline{\cX}_{\overline{\cW}},\overline{\cB}_{\overline{\cW}}+\alpha\overline{\cL}_{\overline{\cW}})\rightarrow \overline{\cW}
\]
is also locally stable. 
Since $(\overline{\cW},\overline{\cD}_{\overline{\cW}})$ is log smooth, it follows from \cite[Corollary 4.55]{kollarFamiliesVarietiesGeneral2023} that 
\[
(\overline{\cX}_{\overline{\cW}},\overline{\cB}_{\overline{\cW}}+\alpha\overline{\cL}_{\overline{\cW}}+\overline{\cF}^*\overline{\cD}_{\overline{\cW}})
\]
is lc. After passing to a stratification of $T$, we may assume that 
\[
(\overline{\cX}_{\overline{\cW}},\overline{\cB}_{\overline{\cW}}+\alpha\overline{\cL}_{\overline{\cW}}+\overline{\cF}^*\overline{\cD}_{\overline{\cW}})\to T
\]
admits a fiberwise log resolution 
$(\overline{\cY}_{\overline{\cW}},\overline{\cR}_{\overline{\cW}})\to T$.
Then, by \cite[Theorem 1.8 (3)]{haconBirationalAutomorphismsVarieties2013},
\[
\vol(K_{\overline{\cX}_{\overline{\cW}_t}}+\overline{\cB}_{\overline{\cW}_t}+\alpha\overline{\cL}_{\overline{\cW}_t}+\overline{\cF}_t^*\overline{\cD}_{\overline{\cW}_t})
= \vol(K_{\overline{\cY}_{\overline{\cW}_t}}+\overline{\cR}_{\overline{\cW}_t,>0})
\]
is independent of $t\in T$.
    \end{proof}
\end{lem}

\subsection{Log birational boundedness}
In this subsection, we do some preparation for the proof of log birational boundedness of weak $(d,\Phi,v,r,\epsilon)$-polarized log Calabi--Yau fibrations $f:((X,B),A)\to (Z,H)$.

In the following theorem, we construct a special birational model $((X', \Delta'), A')\to (Z', D')$ of $((X, B), A)\to Z$, where $(Z', D') \to Z$ factors through the log bounded birational model $(W, D)$ of $Z$ constructed in Theorem \ref{moduli map is bounded}.

\begin{thm}\label{special model of a polarized log Calabi--Yau fibration}
Let $d\in \bN$, $v\in \mathbb{Q}^{>0}$, and $\Phi\subset [0,1]\cap \bQ$ be a finite set. Let $\alpha$ be the rational number defined in Theorem \ref{diagram of the universal Calabi--Yau fibration}. Assume that
\begin{itemize}
    \item $f:(X,B)\to Z$ is a log Calabi--Yau fibration such that $(X,B)$ is klt, and $A$ is an effective integral divisor on $X$,
    \item the general fiber $((X_g,B_g),A_g)$ is a $(d,\Phi,v)$-polarized log Calabi--Yau pair,
    \item there is a canonical bundle formula $K_{X}+B\sim_{\mathbb{Q}}f^*(K_Z+B_Z+\M_Z)$,
    \item $W\to Z$ is a birational morphism, and
    \item $D$ is a reduced divisor on $W$ containing the strict transform of $\Supp(B_Z)$ together with the exceptional divisors over $Z$.
\end{itemize}

Then we can construct a commutative diagram
\[
\xymatrix{
((\overline{X},\overline{B}),\overline{A}) \ar[r] \ar[d]^{\overline{f}} 
& ((X',B'),A') \ar@{-->}[rr] \ar[d]^{f'} 
& & ((X,B),A )\ar[d]^f \\
(\overline{Z},\overline{D}) \ar[r]^{\pi'} 
& (Z',D') \ar[r] 
& (W,D) \ar[r] 
& Z
}
\]
satisfying the following properties:
\begin{enumerate}[label=(\roman*)] 
    \item $Z'\to W$ is a birational morphism,
    \item $f':X'\to Z'$ is a contraction, and $B', A'$ are horizontal $\bQ$-divisors on $X'$,
    \item the generic fiber of $(X',B'+\alpha A')\to Z'$ is isomorphic to the generic fiber of $(X,B+\alpha A)\to Z$,
    \item $\pi':\overline{Z}\to Z'$ is a finite cover,
    \item $(Z',D')$ and $(\overline{Z},\overline{D})$ are log smooth, where $D'$ is the sum of the strict transform of $D$ and all exceptional divisors over $W$, and $\overline{D}$ is the preimage of $D'$ under $\pi'$,
    \item $\overline{X}$ is the normalization of $X'\times_{Z'}\overline{Z}$, and $\overline{B},\overline{A}$ are horizontal $\bQ$-divisors equal to the pullback of $B',A'$ on $\overline{X}$ over the generic point of $Z'$, and
    \item $\overline{f}:((\overline{X},\overline{B}),\overline{A})\to \overline{Z}$ is a family of $(d,\Phi,v)$-polarized log Calabi--Yau pairs.
\end{enumerate}

Furthermore, if $\overline{f}:(\overline{X},\overline{B})\to \overline{Z}$ has klt fibers over codimension one points of $\overline{Z}\setminus \overline{D}$, then, setting $\Delta':=B'+\red(f'^*D')$, we have
\begin{enumerate}
    \item $f'$ has integral fibers over codimension one points of $Z'\setminus D'$,
    \item $(X',\Delta'+\alpha A')$ is lc,
    \item $K_{X'}+\Delta'\sim_{\bQ} f'^*(K_{Z'}+D'+\M_{Z'})$, and
    \item $\Supp(\Delta')$ contains the strict transform of $\Supp(B)$ together with all exceptional divisors over $X$.
\end{enumerate}
\end{thm}

    \begin{proof}
    \begin{enumerate} [label=\textsl{Step} \arabic{enumi}., wide=13pt, itemsep=13pt]
    \item In this step we construct a birational morphism $Z'\to W$ and a finite cover $\overline{Z}\to Z'$.
    
        Let $Y$ be a log resolution of  $(X,B+\alpha A)$. Let $B_Y$ be the strict transform of $B$ plus the reduced horizontal exceptional divisors over $Z$, and let $A_Y$ be the strict transform of $A$. Let $Z^o\subset Z$ be an open subset such that \begin{itemize}
            \item $W\to Z$ is an isomorphism over $Z^o$,
            \item $f:((X,B),A)\rightarrow Z$ is a family of $(d,\Phi,v)$-polarized log Calabi--Yau pairs over $Z^o$, and
            \item $f_Y:(Y,B_Y+A_Y)\rightarrow Z$ is log smooth over $Z^o$.
        \end{itemize}
         Then $B_Y$ and $A_Y$ are effective $\bQ$-divisors which are horizontal over $Z^o$.  By \cite[Theorem 2.1 and Proposition 4.4]{abramovichWeakSemistableReduction2000}, there is an extension $Z^o \hookrightarrow Z'$ such that\begin{itemize}
             \item $Z'$ is a log resolution of $(W, D)$, 
             \item there is an equidimensional toroidal morphism $f'_Y: Y' \to Z'$,
             \item if $B'_Y, A'_Y$ are the closures of $B_Y|_{Z^o}$ and $A_Y|_{Z^o}$, respectively, then they are contained in the toroidal boundary of $Y'$, and 
             \item $((Y', B'_Y), A'_Y )\to Z'$ is an extension of $((Y, B_Y), A_Y )\times_Z Z^o \to Z^o$.
         \end{itemize} 
         
Let $D'$ be the sum of the strict transform of $D$ on $Z'$ and all reduced exceptional divisors over $W$.
By \cite[Proposition~5.1]{abramovichWeakSemistableReduction2000}, there exists a finite cover
$\pi'\colon \overline{Z}\to Z'$ such that
$\overline{f}_Y\colon \overline{Y}\to \overline{Z}$ is an equidimensional toroidal morphism
with reduced fibers, where $\overline{Y}$ is the normalization of
$Y'\times_{Z'}\overline{Z}$.

Note that the finite cover $\overline{Z}\to Z'$ is a Kawamata covering.
To ensure the smoothness of $\overline{Z}$ in the construction, we add extra branch loci
artificially (see \cite[Theorem~1.8.2]{Kawamatabook}).
Let $R'$ be the divisor on $Z'$ whose support contains the union of
$\Supp(D')$ and the branch divisors of $\pi'$.
Define $\overline{R}:=\red(\pi'^*R')$.
Then $(\overline{Z},\overline{R})$ is log smooth by
\cite[Lemma~5.9]{abramovichWeakSemistableReduction2000}.

Let $\overline{D}:=\red(\pi'^*D')$ be the reduced divisor on $\overline{Z}$.
By construction, $\overline{D}\subset \overline{R}$.
Let $\overline{B}_Y$ and $\overline{A}_Y$ be the pullbacks of $B'_Y$ and $A'_Y$ to $\overline{Y}$.
Then they are contained in the toroidal boundary of $\overline{Y}$.

By \cite[Proposition~2.16]{ambroPositivityModuliPart2022},
$(\overline{Y},\overline{B}_Y+\mu\overline{A}_Y+\overline{f}_Y^*\Sigma)$ is lc
for any reduced simple normal crossing divisor $\Sigma$ on $\overline{Z}$,
where $\mu\in(0,1)$ is sufficiently small.
It follows that
\[
\overline{f}_Y\colon (\overline{Y},\overline{B}_Y+\mu\overline{A}_Y)\to \overline{Z}
\]
is a locally stable morphism by
\cite[Corollary~4.55]{kollarFamiliesVarietiesGeneral2023}.

\item In this step we construct a family of $(d,\Phi,v)$-polarized log Calabi--Yau pairs $((\overline{X},\overline{B}),\overline{A})\to\overline{Z}$.

Since $(\overline{Y},\overline{B}_Y+\mu\overline{A}_Y)\to\overline{Z}$ is locally stable and $\overline{Z}$ is smooth, every lc center of $(\overline{Y},\overline{B}_Y+\mu\overline{A}_Y)$ dominates $\overline{Z}$ by \cite[Corollary 4.56]{kollarFamiliesVarietiesGeneral2023}. As a general fiber $(Y'_g,B'_g+\mu A'_g)$ is klt, we conclude that $(\overline{Y},\overline{B}_Y+\mu\overline{A}_Y)$ is klt. The general fiber $(\overline{Y}_g,\overline{B}_{Y_g})$ has a semi-ample model $(X_g,B_g)$, and hence it admits a good minimal model by \cite[Lemma 2.9.1]{haconBoundednessModuliVarieties2018}. Thus, by \cite[Theorem 1.1]{haconExistenceLogCanonical2013}, running an MMP on $K_{\overline{Y}}+\overline{B}_Y$ over $\overline{Z}$ yields a good minimal model $(\overline{X}',\overline{B}')$ over $\overline{Z}$. Let $\overline{A}'$ be the pushforward of $\overline{A}_Y$.

By \cite[Corollary 4.57.1]{kollarFamiliesVarietiesGeneral2023}, $(\overline{X}',\overline{B}')\to\overline{Z}$ is also locally stable. Since $K_{\overline{X}'}+\overline{B}'$ is semi-ample over $\overline{Z}$ and has Kodaira dimension $0$ on the generic fiber, and since $\overline{X}'\to\overline{Z}$ is equidimensional, upper semi-continuity of fiber dimensions gives
\[K_{\overline{X}'}+\overline{B}'\sim_{\bQ,\overline{Z}}0.\]
Define $(\overline{X}',\overline{B}'+\mu\overline{A}')\dashrightarrow(\overline{X},\overline{B}+\mu\overline{A})$ to be the log canonical model of $K_{\overline{X}'}+\overline{B}'+\mu\overline{A}'$ over $\overline{Z}$. As $\overline{X}'\dashrightarrow\overline{X}$ does not extract any divisor, \begin{equation}\label{eq:key}
K_{\overline{X}}+\overline{B}\sim_{\bQ,\overline{Z}}0.
\end{equation}
Finally, by \cite[Corollary 4.57.2]{kollarFamiliesVarietiesGeneral2023}, $((\overline{X},\overline{B}),\overline{A})\to\overline{Z}$ is a stable family of polarized log Calabi--Yau pairs. Since the general fiber is $(d,\Phi,v)$-polarized, by the definition of $\alpha$ the family $(\overline{X},\overline{B}+\alpha\overline{A})\to\overline{Z}$ is locally stable.

\item In this step we construct a contraction $f':X'\to Z'$ together with horizontal $\bQ$-divisors $B',A'$ on $X'$, and show that the generic fiber of $(X',B'+\alpha A')\to Z'$ is isomorphic to that of $(X,B+\alpha A)\to Z$.

By the Hurwitz formula \cite[\S2.41.4]{kollarSingularitiesMinimalModel2013} we have
\[K_{\overline{Z}}+\overline{R}=\pi^*(K_{Z'}+R'),\]
where both $(\overline{Z},\overline{R})$ and $(Z',R')$ are log smooth by construction. By \cite[Corollary 4.55]{kollarFamiliesVarietiesGeneral2023}, $(\overline{Y},\overline{B}_Y+\mu\overline{A}_Y+\overline{f}_Y^*\overline{R})$ is lc. Let $\pi_Y:\overline{Y}\to Y'$ be the natural finite cover. Since étale morphisms are stable under base change, the ramification divisor of $\pi_Y$ is contained in $\Supp(\overline{f}_Y^*\overline{R})$. Thus, by \cite[\S2.41.4]{kollarSingularitiesMinimalModel2013},
\[K_{\overline{Y}}+\overline{B}_Y+\mu\overline{A}_Y+\overline{f}_Y^*\overline{R}
=\pi_Y^*\big(K_{Y'}+B'_Y+\mu A'_Y+\red(f'^*_YR')\big),\]
and $(Y',B'_Y+\mu A'_Y+\red(f'^*_YR'))$ is lc. 

Since the general fiber $(Y'_g,B'_{Y_g})$ admits a semi-ample model $(X_g,B_g)$,
by \cite[Theorem~1.1]{haconExistenceLogCanonical2013} we can run an MMP on
\[
K_{Y'}+B'_Y+\red(f'^*_Y R')
\]
over $Z'$ (equivalently, on
$K_{Y'}+B'_Y+\red(f'^*_Y R')-a f'^*_Y R'$ for $a\ll 1$).
This yields a good minimal model
\[
(X'',B''+\red(f''^*R'))
\]
over $Z'$, with morphism $f''\colon X''\to Z'$.
Let $A''$ be the pushforward of $A'_Y$.

By Lemma~\ref{minimal models and canonical models commutes with base change}(1),
$\overline{X}'$ is isomorphic in codimension one to the normalization of
$X''\times_{Z'}\overline{Z}$.
Now let
\[
(X',B'+\red(f'^*R')+\mu A')
\]
be the log canonical model of
$K_{X''}+B''+\red(f''^*R')+\mu A''$ over $Z'$, where $f'\colon X'\to Z'$.
Then the generic fiber of
\[
(X',B'+\red(f'^*R')+\alpha A')\to Z'
\]
coincides with that of
\[
(X,B+\alpha A)\to Z.
\]

As both $B'_Y$ and $A'_Y$ are horizontal over $Z'$, so are $B'$ and $A'$. By Lemma \ref{minimal models and canonical models commutes with base change}(2), $\overline{X}$ is isomorphic to the normalization of $X'\times_{Z'}\overline{Z}$ and 
\[K_{\overline{X}}+\overline{B}+\alpha\overline{A}+\overline{f}^*\overline{R}
=\pi_X^*\big(K_{X'}+B'+\alpha A'+\red(f'^*R')\big),\]
where $\pi_X:\overline{X}\to X'$. Using $\overline{A}=\pi_X^*A'$, this simplifies to
\[K_{\overline{X}}+\overline{B}+\overline{f}^*\overline{R}
=\pi_X^*(K_{X'}+B'+\red(f'^*R')).\]
Combining (\ref{eq:key}) with Lemma \ref{linear trivial preserved under finite base change}, we conclude that
\[K_{X'}+B'+\red(f'^*R')\sim_{\bQ, Z'}0.\]
Finally, since $(\overline{X},\overline{B}+\alpha\overline{A}+\overline{f}^*\overline{R})$ is lc, \cite[Corollary 2.43]{kollarSingularitiesMinimalModel2013} implies that $(X',B'+\red(f'^*R')+\alpha A')$ is also lc.

\item In this step we prove the furthermore part. From now on we assume that $(\overline{X},\overline{B})\rightarrow \overline{Z}$ has klt fibers over codimension one points in $\overline{Z}\setminus \overline{D}$, and denote $\Delta':=B'+\red(f'^*D')$.

Let $P$ be a prime divisor on $Z'$ not contained in $\Supp(D')$, and let $\widetilde{B_Z}$ be the strict transform of $B_Z$ on $W$. Since $\Supp(\widetilde{B_Z})\subseteq \Supp(D)$ and $\Supp(D')$ contains both the strict transform of $D$ and all exceptional divisors over $Z$, by the definition of the discriminant part in the canonical bundle formula, we conclude that $f':X' \to Z'$ has a reduced fiber over the generic point of $P$.

Let $\overline{P}$ be an irreducible component of the preimage of $P$ on $\overline{Z}$. By assumption, $(\overline{X},\overline{B})$ has a klt fiber over the generic point of $\overline{P}$. By inversion of adjunction, $(\overline{X},\overline{B}+\overline{f}^*\overline{P})$ is plt near the fiber over the generic point of $\overline{P}$. By \cite[\S2.41.4]{kollarSingularitiesMinimalModel2013}, over the generic point of $\overline{P}$ the divisor $K_{\overline{X}}+\overline{B}+\overline{f}^*\overline{P}$ is equivalent to the pullback of $K_{X'}+B'+f'^*P$. Hence, near the fiber of the generic point of $P$, the pair $(X',B'+f'^*P)$ is plt by \cite[Corollary 2.43]{kollarSingularitiesMinimalModel2013}. Therefore, $f'^*P$ is irreducible over the generic point of $P$. This proves (1).

Because $f'$ is equidimensional and has reduced fibers over codimension one points of $Z'\setminus D'$, we obtain
\[
\red(f'^*R')= \red(f'^*D)+f'^*(R'-D').
\]
Since 
\[
K_{X'}+B'+\red(f'^*R')\sim_{\mathbb{Q},Z'}0
\]
and $(X',B'+\red(f'^*R')+\alpha A')$ is lc, it follows that
\[
K_{X'}+\Delta'=K_{X'}+B'+\red(f'^*D)\sim_{\mathbb{Q},Z'}0
\]
and $(X',\Delta'+\alpha A')$ is also lc. This proves (2).

Next, observe that if $P$ is a prime divisor on $Z'$ not contained in $\Supp(D')$, then $(X',\Delta'+f'^*P)$ is plt over the generic point of $P$, which implies that the discriminant divisor of $f':(X',\Delta')\to Z'$ is contained in $\Supp(D')$. If $P$ is a prime divisor contained in $\Supp(D')$, then $\lct(X',\Delta';P)=0$. Thus,
\[
K_{X'}+\Delta'\sim_{\mathbb{Q}}f'^*(K_{Z'}+D'+\M_{Z'}),
\]
where $\M$ is the moduli $\mathbf{b}$-divisor corresponding to $f:(X,B)\to Z$. This proves (3).

Finally, we prove (4). First, we show that $\Supp(f'^*D')$ contains all exceptional divisors over $X$. Suppose $E'$ is a prime divisor on $X'$ exceptional over $X$ but not contained in $\Supp(f'^*D')$. Since $(X,B)\to Z$ and $(X',B')\to Z'$ have the same generic fiber, $E'$ is vertical over $Z'$. As $f':X'\to Z'$ is equidimensional, $P':=f'(E')$ is a prime divisor on $Z'$ not contained in $\Supp(D')$. Because $\Supp(D')$ contains all exceptional divisors over $Z$, the image of $P'$ on $Z$ is also a prime divisor $P$. Let $F$ be a component of $f^{-1}P$ dominating $P$. Then $F$ is a non-klt center of $(X',B'+f'^*P')$ over the generic point of $P'$, distinct from $E'$, since $E'$ is exceptional over $X$. This contradicts the fact that $(X',B'+f'^*P')$ is plt near the fiber over the generic point of $P'$.

Let $Q'$ be an irreducible component of the strict transform of $\Supp(B^\vv)$ in $X'$, and let $Q$ be the image of $Q'$ in $X$. Then $Q\subseteq \Supp(B^\vv)$. Since $f':X'\to Z'$ is equidimensional, $f'(Q')$ is a prime divisor on $Z'$. By \cite[Lemma 2.6.(b)]{jiaoBoundednessPolarizedCalabiYau2022}, every $f$-vertical log center of $(X,B)$ dominates a generalized log center of $(Z,B_Z,\M)$. It follows that $f(Q)$ is a generalized log center of $(Z,B_Z,\M)$, and hence $f'(Q')$ is a generalized log place of $(Z,B_Z,\M)$. By construction, the divisor $D'$ contains the strict transform of $\Supp(B_Z)$ together with all the exceptional divisors over $Z$. Therefore, $f'(Q')\subseteq \Supp(D')$, and consequently $Q'\subseteq \Supp(f'^*D')$. Hence, $\Supp(f'^*D')$ contains the strict transform of $\Supp(B^\vv)$. On the other hand, since the fibrations $(X,B)\to Z$ and $(X',B')\to Z'$ have the same generic fiber, $\Supp(B')$ contains the strict transform of $\Supp(B^\h)$. Combining the above, we conclude that $\Supp(\Delta')$ contains the strict transform of $\Supp(B)$ and all exceptional divisors over $X$. This proves (4). 
        \end{enumerate}
    \end{proof}

In the following theorem, we aim to bound the log canonical volume of the special birational model
constructed in Theorem \ref{special model of a polarized log Calabi--Yau fibration}.

\begin{thm}\label{volume is bounded}
    Let $d\in \bN$,  $v,r,\epsilon\in \mathbb{Q}^{>0}$, and $\Phi\subset[0,1]\cap \bQ$ be a finite set. Then there exists a rational number $\alpha\in (0,1)$ and positive numbers $m,w$ depending only on $d,\Phi,v,r,\epsilon$ satisfying the following:
    
    If $f:((X, B),A)\rightarrow (Z, H)$ is a weak $(d,\Phi,v,r,\epsilon)$-polarized log Calabi--Yau fibration, then there exists a polarized log Calabi--Yau fibration $f':((X',\Delta'),L')\rightarrow Z'$ such that
    \begin{enumerate}
    \item $X'\dasharrow X$ is a birational map, and $Z'\to Z$ is a birational morphism,
    \item the generic fiber of $f:(X,B)\rightarrow Z$ is isomorphic to the generic fiber of $f':(X',\Delta')\rightarrow Z'$,
    \item $L'_g:=L'|_{X'_g}$ is numerically equivalent to $mA'_g:=mA'|_{X'_g}$ on $X'_g$, where $A'$ is the strict transform of $A$ on $X'$, and
    \item The coefficients of $\Delta'$ are in $\Phi\cup \{1\}$.
\end{enumerate}
Moreover, we have
\begin{enumerate}  \setcounter{enumi}{4} 
    \item $\Delta'$ contains the strict transform of $\Supp(B)$ on $X'$ and all exceptional divisors over $X$,
    \item $(X',\Delta'+\alpha L')$ is lc,
    \item $K_{X'}+\Delta'+\alpha L'- h'^*H$ is big, where $h':X'\to Z$, and 
    \item $\vol(K_{X'}+\Delta'+\alpha L')\leq w$.
\end{enumerate}

    \begin{proof}
    \begin{enumerate} [label=\textsl{Step} \arabic{enumi}., wide=13pt, itemsep=13pt]
\item In this step we construct a polarized log Calabi--Yau fibration 
\[
f':((X',\Delta'),L')\to Z'
\]
by Theorem \ref{special model of a polarized log Calabi--Yau fibration}.

By Theorem \ref{moduli map is bounded}, there exists a birational morphism $h:W\to Z$ and a reduced divisor $D$ on $W$ such that 
\begin{itemize}
    \item $(W,D)$ is log bounded,
    \item the induced rational map $\psi_W:W\dashrightarrow \cS^*$ is a bounded morphism,
    \item $D\supseteq \Supp(h_*^{-1}B_Z+E+\psi_W^*\cH)$, where $E$ is the sum of the reduced exceptional divisors of $h$, and $\cH$ is a very ample divisor on $\cS^*$,
    \item $K_W+D-h^*H$ is big.
\end{itemize}

Let $\overline{W}$ be the normalization of the main component of 
$W \times_{\cS^*} \cS^!$, and let $D_{\overline{W}}$ denote the preimage of $D$ via
$\overline{W} \to W$.
After replacing $(\overline{W},D_{\overline{W}})$ with its log resolution,
we may assume that $(\overline{W},D_{\overline{W}})$ is log smooth.
Then $\overline{W} \to W$ is generically finite.
Let \[f_{\overline{W}}:((X_{\overline{W}},B_{\overline{W}}),L_{\overline{W}})\to \overline{W}\]
be the pullback of $((\cX^!,\cB^!),\cL^!)\to \cS^!$ via $\overline{W}\to \cS^!$.

Let $L$ on $X$ be the closure of the pullback of $\cL$ via the moduli map
$U \to \cS$ for some open subset $U\subset Z$.
By Theorem~\ref{diagram of the universal Calabi--Yau fibration},
the general fiber $((X_g,B_g),L_g)$ is a
$(\dim X_g,\Phi,v')$-polarized log Calabi--Yau pair,
where $v'$ depends only on $d,\Phi,v$.

Applying Theorem~\ref{special model of a polarized log Calabi--Yau fibration},
we obtain a family of
$(\dim X_g,\Phi,v')$-polarized log Calabi--Yau pairs
\[
\overline{f}\colon ((\overline{X},\overline{B}),\overline{L})\to \overline{Z},
\]
and a polarized log Calabi--Yau fibration
\[
f'\colon ((X',\Delta'),L')\to Z'
\]
satisfying (1)--(4).
We may assume that $Z'$ is the log resolution of $(W,D)$
extracting all exceptional divisors of $\overline{W}\to W$.

         \item In this step we prove (5)--(7).
         
     By Theorem \ref{special model of a polarized log Calabi--Yau fibration} (2)(4), 
to show that $\Delta'$ contains the strict transform of $\Supp(B)$ on $X'$ and all exceptional divisors over $X$, and that $(X',\Delta'+\alpha L')$ is lc, 
it suffices to prove that $(\overline{X},\overline{B})\to \overline{Z}$ has klt fibers in $\overline{Z}\setminus \overline{D}$.

         Let $\widetilde{Z}\to \overline{Z}$ be a generically finite morphism such that \begin{itemize}
             \item $\widetilde{Z}\to Z\dashrightarrow \cS$ is a morphism and factors through $\overline{\cS}\to \cS$, and
             \item $\widetilde{Z}\to Z'\to W$ factors through $\overline{W}\to W$.
         \end{itemize}
         We have the following commutative diagram:
         \[
\xymatrix{
     & & (Z',D') \ar[d]^g & (\overline{Z},\overline{D}) \ar[l] & (\widetilde{Z} ,\widetilde{D})\ar[l]\ar[ld] \ar[dd] \\
  U\ar[rd]\ar@{^{(}->}[r]&  Z & (W,D) \ar[l]^h \ar[d] & (\overline{W},D_{\overline{W}}) \ar[l] \ar[d] \\
   & \cS \ar[r] & ({\cS^*},{\cH}) & {\cS^!} \ar[l] & \overline{\cS} \ar[l] \ar@/^1pc/[lll]
}
\]
 Let $((\widetilde{X},\widetilde{B}),\widetilde{L})$ be the normalization of the main component of the base change of $((\overline{X},\overline{B}),\overline{L})$ by $\widetilde{Z}\to \overline{Z}$. Let $((\widetilde{X}_W,\widetilde{B}_W),\widetilde{L}_W)$ be the normalization of the main component of the base change of $((X_{\overline{W}},B_{\overline{W}}),L_{\overline{W}})$ by $\widetilde{Z}\to \overline{W}$.

By Theorem \ref{diagram of the universal Calabi--Yau fibration} (4) and (6), the generic fiber of $((\widetilde{X},\widetilde{B}),\widetilde{L})\to \widetilde{Z}$ is isomorphic to the generic fiber of $((\widetilde{X}_W,\widetilde{B}_W),\widetilde{L}_W)\to \widetilde{Z}$. Moreover, since both $((\widetilde{X},\widetilde{B}),\widetilde{L})\to \widetilde{Z}$ and $((\widetilde{X}_W,\widetilde{B}_W),\widetilde{L}_W)\to \widetilde{Z}$ are families of polarized log Calabi--Yau pairs, by the separatedness of the moduli of polarized log Calabi--Yau pairs, we obtain
\[
((\widetilde{X},\widetilde{B}),\widetilde{L})\cong ((\widetilde{X}_W,\widetilde{B}_W),\widetilde{L}_W).
\]

By Theorem \ref{diagram of the universal Calabi--Yau fibration} (8), Theorem \ref{moduli map is bounded} (2), and the fact that $D_{\overline{W}}$ is the preimage of $D$, we conclude that $(X_{\overline{W}},B_{\overline{W}})\to \overline{W}$ has klt fibers over $\overline{W}\setminus D_{\overline{W}}$. Therefore, $(\widetilde{X},\widetilde{B})\to \widetilde{Z}$ has klt fibers over $\widetilde{Z}\setminus \widetilde{D}'$, where $\widetilde{D}'$ is the preimage of $D_{\overline{W}}$. Since $\overline{D}$ contains the preimage of $D$ on $\overline{Z}$, it follows that $\Supp(\widetilde{D}')\subseteq \Supp(\widetilde{D})$, where $\widetilde{D}$ is the preimage of $\overline{D}$. Hence $(\overline{X},\overline{B})$ has klt fibers over $\overline{Z}\setminus \overline{D}$.

\medskip
We now show that \[K_{X'}+\Delta'+\alpha L'-h'^*H\] is big. By Theorem \ref{moduli map is bounded} (3), $K_W+D-h^*H$ is big. Since $D'$ contains the strict transform of $D$ plus the reduced exceptional divisors over $W$, it follows that $K_{Z'}+D'-(K_W+D)$ is effective. Hence $K_{Z'}+D'-g^*h^*H$ is big. Let $0<a\ll1$. By Theorem \ref{special model of a polarized log Calabi--Yau fibration} (3), we have
\[
K_{X'}+\Delta'+\alpha L'-h'^*H
= f'^*(K_{Z'}+D'+\M_{Z'}-g^*h^*H) + aL' + (\alpha-a)L'.
\]
Since $L'$ is big over $Z'$ and $L'\geq 0$, it follows that $K_{X'}+\Delta'+\alpha L'-h'^*H$ is the sum of a big $\bQ$-divisor and an effective $\bQ$-divisor, and hence big.

        \item In this step we prove that $\vol(K_{X'}+\Delta'+\alpha L')$ is bounded from above.

        Consider the following commutative diagram:
        \[\xymatrix{
        (X',B'+\alpha L')\ar[d]_{f'} & (\widetilde{X},\widetilde{B}+\alpha\widetilde{L})\ar[d]_{\widetilde{f}} \ar[l]_-{\mu} \ar[r]^-{\eta} & \widetilde{X}' \ar[r]^-{\nu} & (X_{\overline{W}},B_{\overline{W}}+\alpha L_{\overline {W}})\ar[d]^{f_{\overline{W}}} \\
        (Z',D') & (\widetilde{Z},\widetilde{D})\ar[l]_-\pi \ar[rr]^-\tau&  & ({\overline {W}},D_{\overline{W}})
        }
        \] 
        Here $\widetilde{X}\xrightarrow{\eta} \widetilde{X}'\xrightarrow{\nu} X_{\overline{W}}$ is the Stein factorization of $\widetilde{X}\rightarrow X_{\overline{W}}$, hence $\nu$ is a finite morphism and $\eta$ is a birational morphism. Now we claim that \[\eta_*\mu^*(K_{X'}+\Delta'+\alpha L')=\nu^*(K_{X_{\overline{W}}}+B_{\overline{W}}+\alpha L_{\overline{W}}+f_{\overline{W}}^*D_{\overline{W}}).\]

Since the generic fiber of 
\[
(X',\Delta'+\alpha L')\times_{Z'}\widetilde{Z}\rightarrow \widetilde{Z}
\] 
is equal to the generic fiber of 
\[
(X_{\overline {W}},B_{\overline {W}}+\alpha L_{\overline{W}}+f_{\overline{W}}^*D_W)\times_{\overline{W}} \widetilde{Z}\rightarrow \widetilde{Z},
\] 
it suffices to compare vertical divisors. Let $\widetilde{P}$ be a prime divisor on $\widetilde{X}$ which is vertical over $\widetilde{Z}$, and assume that its image $\widetilde{P}' := \eta(\widetilde{P})$ is also a prime divisor on $\widetilde{X}'$. Let $P_{\overline{W}}$ denote the image of $\widetilde{P}$ on $X_{\overline{W}}$. We claim that the image of $\widetilde{P}$ on $X'$ is a prime divisor as well.

Indeed, since $P_{\overline{W}}$ is a prime divisor and $f_{\overline{W}}$ is equidimensional, its image $f_{\overline{W}}(P_{\overline{W}})$ is a prime divisor on $\overline{W}$. By Step~1, the morphism $Z'\to W$ is a log resolution of $(W,D)$ extracting all exceptional divisors of $\overline{W}\to W$. It follows that 
$\pi \circ \widetilde{f}(\widetilde{P})$
is a prime divisor on $Z'$. Since both $f'$ and $\widetilde{f}$ are equidimensional with fibers of the same dimension, the image $\mu(\widetilde{P})$ is a prime divisor on $X'$, which we denote by $P'$.

We now distinguish two cases:

\textbf{Case (1):} $\coeff_{P'}\Delta'=1$. In this case, $f'(P')$ is a prime divisor contained in $\Supp(D')$. By construction, $\pi^{-1}(D')$ is the union of $\tau^{-1}(D_{\overline{W}})$ and some $\tau$-exceptional divisors. Therefore, $f_{\overline{W}}(P_{\overline{W}})$ is a prime divisor contained in $\Supp(D_{\overline{W}})$. Hence, by \cite[\S2.41.4]{kollarSingularitiesMinimalModel2013}, over the generic point of $\widetilde{P}'$, we have
\[
\nu^*(K_{X_{\overline{W}}}+B_{\overline{W}}+\alpha L_{\overline{W}}+f_{\overline{W}}^*D_{\overline{W}})=K_{\widetilde{X}'}+\widetilde{P}'=\eta_*\mu^*(K_{X'}+\Delta'+\alpha L').
\]

\textbf{Case (2):} $\coeff_{P'}\Delta'=0$. In this case, $f'(P')$ is not contained in $\Supp(D')$ and $f_{\overline{W}}(P_{\overline{W}})$ is not contained in $\Supp(D_{\overline{W}})$. Hence, by Theorem \ref{special model of a polarized log Calabi--Yau fibration} (1), $f'$ has reduced fibers over the generic point of $f'(P')$. Since the ramification locus of ${\overline{W}}\rightarrow W$ is contained in $\Supp(D_{\overline{W}})$ by Theorem \ref{diagram of the universal Calabi--Yau fibration} (8), the map ${\overline{W}}\dashrightarrow Z'$ is étale over the generic point of $f'(P')$. Therefore, the ramification index of $\mu$ along $\widetilde{P}$ equals that of $\nu$ along $\widetilde{P}$. By the Hurwitz formula, we have
\[
\nu^*(K_{X_{\overline{W}}}+B_{\overline{W}}+\alpha L_{\overline{W}}+f_{\overline{W}}^*D_{\overline{W}})= \eta_*\mu^*(K_{X'}+\Delta'+\alpha L')
\]
over the generic point of $\widetilde{P}$. This proves the claim.

By the claim, we obtain
\[
\vol(\mu^*(K_{X'}+\Delta'+\alpha L'))\leq \vol(\nu^*(K_{X_{\overline{W}}}+B_{\overline{W}}+\alpha L_{\overline{W}}+f_{\overline{W}}^*D_{\overline{W}})).
\] 
By \cite[Lemma 4.3]{holschbachChebotarevtypedensitytheorem}, it follows that
\begin{align*}
\vol(\mu^*(K_{X'}+\Delta'+\alpha L')) &= \deg(\mu)\vol(K_{X'}+\Delta'+\alpha L'),\\
\vol(\nu^*(K_{X_{\overline{W}}}+B_{\overline{W}}+\alpha L_{\overline{W}}+f_{\overline{W}}^*D_{\overline{W}})) &= \deg(\nu)\vol(K_{X_{\overline{W}}}+B_{\overline{W}}+\alpha L_{\overline{W}}+f_{\overline{W}}^*D_{\overline{W}}).
\end{align*}
Since $\deg(\nu)\cdot \deg({\overline{W}}/Z) =\deg(\mu)$, we conclude
\[
\vol(K_{X'}+\Delta'+\alpha L')\leq \frac{1}{\deg({\overline{W}}/Z)} \vol(K_{X_{\overline{W}}}+B_{\overline{W}}+\alpha L_{\overline{W}}+f_{\overline{W}}^*D_{\overline{W}})\leq w
\]
by Lemma \ref{volume is bounded up to a finite cover}, where $w$ is a positive integer depending only on $d,\Phi,v,r,\epsilon$.
    \end{enumerate}
    \end{proof}
\end{thm}

\subsection{Log boundedness in  codimension one}
We now proceed to establish the main theorem of this section.
\begin{proof}[Proof of Theorem \ref{thm:bdd of weak pcy fibration}]
 \begin{enumerate} [label=\textsl{Step} \arabic{enumi}., wide=13pt, itemsep=13pt]

\item Let \[h':((X',\Delta'),L')\to Z'\to Z\] be the fibration constructed in Theorem \ref{volume is bounded}. By \cite[Theorem 1.3]{haconACCLogCanonical2014}, there exists a fixed positive integer $n$ such that the linear system 
\[
|n(K_{X'}+\Delta'+\alpha L')|
\] 
defines a birational map. Let $\pi:Y'\to X'$ be a log resolution of $(X',\Delta'+L')$ such that  
\[
|n\pi^*(K_{X'}+\Delta'+\alpha L')|=|M|+F,
\]  
where $|M|$ is the free part and $F$ is the fixed part. Set \[G:=M+\pi^*h'^*H.\] Then $|G|$ is base point free and defines a birational morphism $\mu:Y'\to Y$ such that $\mu_*G$ is very ample on $Y$. Moreover, every curve contracted by $\mu$ intersects $\pi^*h'^*H$ trivially, hence the induced map $g:Y\dashrightarrow Z$ is in fact a morphism.  

By construction, we have  
\[
G+F\sim_{\bQ}n\pi^*(K_{X'}+\Delta'+\alpha L')+\pi^*h'^*H.
\]  
Let $\eta_Z$ denote the generic point of $Z$. Since $K_{X'}+\Delta'\sim_{\bQ,\eta_Z}0$, we obtain  
\[
G+F\sim_{\bQ,\eta_Z}n\alpha\pi^*L'.
\]  

\item Define 
\[
\Sigma':=\red(\pi_*^{-1}\Delta')+G+F+\pi^*h'^*H_Z+E',
\] 
where $E'$ denotes the reduced exceptional divisor of $\pi:Y'\to X'$, and set $\Sigma=\mu_*\Sigma'$. In this step, we prove that $(Y,\Sigma)$ belongs to a log bounded family and that the morphism $g:Y\to Z$ is bounded.  

Since $K_{X'}+\Delta'+\alpha L'-h'^*H$ is big by Theorem \ref{volume is bounded} (7), it follows that  
\[
\vol(G)\leq \vol\big((n+1)(K_{X'}+\Delta'+\alpha L')\big)\leq (n+1)^d w.
\]  
By \cite[Lemma 7.3]{haconACCLogCanonical2014}, there exists a fixed positive number $\beta<1$ such that $K_{X'}+\beta(\Delta'+\alpha L')$ is big.  

Define  
\[
c:=\frac{1}{\min\{c_i\in \Phi\cup\{1\}\mid c_i\neq 0\}},
\]  
and choose a fixed positive number $t$ satisfying  
\[
\frac{c+t\beta}{1+t}\leq 1,\quad \text{equivalently,}\quad 
t\geq \frac{c-1}{1-\beta}.
\]  
 Then we conclude that 
\begin{align*}
    &\vol(K_{Y'}+\Sigma'+(4d+2)G)\\
    \leq &\vol(K_{X'}+\pi_*\Sigma'+(4d+2)\pi_*G)\\
    \leq &\vol(K_{X'}+c\Delta'+(10d+3)(n+1)(K_{X'}+\Delta'+\alpha L'))\\
    \leq &\vol(K_{X'}+c\Delta'+t(K_{X'}+\beta(\Delta'+\alpha L'))+(10d+3)(n+1)(K_{X'}+\Delta'+\alpha L'))\\
     \leq &\vol((1+t+(10d+3)(n+1))(K_{X'}+\Delta'+\alpha L'))\\
    \leq &(1+t+(10d+3)(n+1))^dw,
\end{align*}
where the second inequality holds since $K_{X'}+\Delta'+\alpha L'-h'^*H$ is big.
Therefore, by \cite[Lemma 3.2]{haconBirationalAutomorphismsVarieties2013},
\begin{align*}
    \Sigma\cdot((4d+2)\mu_*G)^{d-1}&=\Sigma'\cdot((4d+2)G)^{d-1}\\
    &\leq 2^d\vol(Y',K_{Y'}+\Sigma'+(4d+2)G)\\
    &\leq 2^d(1+t+(10d+3)(n+1))^dw.
\end{align*}
Thus by \cite[Lemma 2.4.2 (4)]{haconBirationalAutomorphismsVarieties2013}, $(Y,\Sigma)$ forms a log bounded family. By \cite[Lemma 2.8]{hanBirationalBoundednessRationally2022}, $g:Y\to Z$ is a bounded morphism.

\item There exists a family of contractions \[\cY\to \cZ\to T\] and three effective divisors $\Omega$, $\cG$, and $\cF$ on $\cY$ such that for some closed point $t\in T$, the fiber $\cY_t\to \cZ_t$ is isomorphic to $g:Y\to Z$, with \[\Omega_t\simeq \Sigma,\quad \cG_t\simeq \mu_*G,\quad  \cF_t\simeq \mu_*F.\] Since $\mu_*G$ is a very ample divisor on $Y$ and ampleness is an open condition, after passing to a stratification we may assume that $\cG$ is ample over the generic point of $\cZ$. If we write \[\cJ_{\cY}=\cG+\cF,\] then $\cJ_{\cY}$ is big over $\cZ$. Setting \[J_Y=\mu_*G+\mu_*F,\] we have 
\[
J_Y\sim_{\bQ,\eta_Z}n\alpha \mu_*\pi^*L'.
\]

After taking a log resolution of $(\cY,\Omega)$ and passing to a stratification of $T$, we may assume that $T$ is smooth and $(\cY,\Omega)$ is log smooth over $T$. Replacing $\cJ_{\cY}$ by its pullback, it remains big over $\cZ$. After passing to a finite \'etale cover of a stratification of $T$ (see \cite[Claim 4.38.1]{kollarSingularitiesMinimalModel2013}), we may assume that every prime component of $\Omega$ restricts to a prime divisor fiberwise. Moreover, after replacing $(\cY,\Omega)$ by a sequence of blowups of strata, we extract all divisors whose log discrepancies with respect to $(\cY,(1-\epsilon)\Omega)$ are at most one. Up to a further stratification of $T$, this process can be assumed to be fiberwise. Therefore, the induced birational map \[\cY_t\dashrightarrow X'\dashrightarrow X\] does not extract any divisor.

Since $H$ is very ample and $H_Z\in |6dH|$ is general, we may assume that $(X,B+\tfrac{1}{2}f^*H_Z)$ is $\epsilon$-lc. By the canonical bundle formula, 
\[
K_X+B\sim_{\bQ} f^*(K_Z+B_Z+\M_Z).
\]
By the boundedness of the length of extremal rays, $K_Z+B_Z+\M_Z+3dH$ is ample, and hence
\[
K_X+B+\tfrac{1}{2}f^*H_Z \sim_{\bQ} f^*(K_Z+B_Z+\M_Z+\tfrac{1}{2}H_Z)
\]
is semi-ample.

Let $\Gamma_{\cY_t}$ be the strict transform of $B+\tfrac{1}{2}f^*H_Z$ on $\cY_t$, together with $(1-\tfrac{1}{2}\epsilon)E$, where $E$ is the reduced exceptional divisor of $\cY_t\dashrightarrow X$. Define $\Gamma_{\cY}$ to be the divisor supported on $\Omega$ whose restriction to $\cY_t$ is $\Gamma_{\cY_t}$. Since the coefficients of $B+\tfrac{1}{2}f^*H_Z$ lie in a finite set, the possible coefficients appearing in $\Gamma_{\cY_t}$ also belong to a finite set $\Phi\cup\{\tfrac{1}{2},1-\tfrac{1}{2}\epsilon\}$. Therefore, without loss of generality, we may assume that $\Gamma_{\cY}$ is fixed on $\cY$.

By construction, $(X,B+\tfrac{1}{2}f^*H_Z)$ is a good minimal model of $(\cY_t,\Gamma_{\cY_t})$. By \cite[Theorem 1.2]{haconBoundednessModuliVarieties2018}, the pair $(\cY,\Gamma_{\cY})$ admits a relative good minimal model $(\cV,\Gamma)$ over $T$, which, up to a stratification of $T$, induces good minimal models fiberwise. By the boundedness of the length of extremal rays, the induced map $\cV\dashrightarrow \cZ$ is a morphism. If we denote the pushforward of $\cJ_{\cY}$ by $\cJ$, then $\cJ$ is big over $\cZ$. By \cite[Lemma 2.4]{haconExistenceLogCanonical2013}, the pair $(\cV_t,\Gamma_t)$ is isomorphic in codimension one to $(X,B+\tfrac{1}{2}f^*H_Z)$. 

Since $L'$ is numerically equivalent to the strict transform of $mA$ on the generic fiber of $X'\to Z$, and since $J_Y\sim_{\bQ,\eta_Z}n\alpha \mu_*\pi^*L'$, we conclude that $\cJ_t$ is numerically equivalent to the strict transform of $mn\alpha A$ on the generic fiber of $\cV_t\to Z$. Replacing $n$ by a bounded multiple, we may assume that 
$l:=mn\alpha$
is an integer.  
Thus, the pair $(\cV_t,\Gamma_t)$ and the integral divisor $\cJ_t$ are what we need.

\end{enumerate}
\end{proof}

\begin{remark}\label{big divisor in family}
    We remark that the relative bigness of $\cJ$ will be used in the proof of Theorem \ref{mainthm2}.
\end{remark}

\section{Polarized log Calabi--Yau fibrations: arbitrary coefficients}\label{sec:Polarized log Calabi--Yau fibrations: arbitrary coefficients}

In this section, we consider the boundedness of polarized log Calabi--Yau fibrations $f :((X, B),A) \to (Z, H)$ where the coefficients of $B$ are arbitrary. 

We first recall the definition and boundedness result for Fano type fibrations.

\begin{definition}[{\cite[Definition 1.1]{birkarBoundednessFanoType2024}}]\label{def:FT fib}
Let $d\in \bN$ and $r,\epsilon\in\mathbb{R}^{>0}$. 
 A ($d,r,\epsilon$)-\textit{Fano type fibration}  $f:(X,B)\to (Z,H)$ consists of 
\begin{enumerate}
   
    \item a projective $\epsilon$-lc pair $(X,B)$ of dimension $d$,
    \item a fibration $f:X\to Z$ such that $K_X+B\sim_{\bR}f^*N$ for some $\bR$-divisor $N$ on $Z$, 
    \item $-K_X$ is big over $Z$, i.e., $X$ is of Fano type over $Z$,
    \item $H\geq 0$ is a very ample divisor  on $Z$ with $H^{\dim Z}\leq r$, and
    \item $H-N$ is ample.
    \end{enumerate}
    
\end{definition}

\begin{thm}[{\cite[Theorem 1.3]{birkarBoundednessFanoType2024}}]
\label{FTF}
Let $d\in \bN$ and $r,\epsilon,\delta \in \bR^{>0}$. Consider the set of all $(d,r,\epsilon)$-Fano type fibrations $(X,B)\to (Z,H)$ and $\bR$-divisors $0\leq \Delta\leq B$ where the non-zero coefficients of $\Delta$ are larger than $\delta$. 

Then the set of such $(X,\Delta+f^*H)$ is log bounded.
\end{thm}

\begin{proof}
    By \cite[Theorem 1.4]{birkarBoundednessFanoType2024}, there exists a positive number $t<1$ depending only on $d,r,\epsilon$ such that $(X,B+tf^*H)$ is $\frac{\epsilon}{2}$-lc. Then $(X,B+tf^*H)\to Z$ is a $(d,2^dr,\frac{\epsilon}{2})$-Fano type fibration, hence by \cite[Theorem 1.3]{birkarBoundednessFanoType2024}, $(X,\Delta+f^*H)$ is log bounded.
\end{proof}

\begin{lem}\label{induction method}
Let $d\in \bN$ and $r,\epsilon,\delta\in \bR^{>0}$. Assume that 
\begin{itemize}
\item $(X,B)$ is an $\epsilon$-lc pair of dimension $d$,
\item $f:X\to Z$ is a contraction to a normal projective variety,
\item $K_X+B\sim_{\bR}f^*N$ for some $\bR$-divisor $N$ on $Z$,
\item $H$ is a very ample divisor on $Z$ such that $H^{\dim Z}\leq r$ and $H-N$ is ample,
\item $0\leq \Delta\leq B$ is an $\bR$-divisor on $X$ such that the non-zero coefficients of $\Delta$ are larger than $\delta$,
\item $f:X\to Z$ factors through a contraction $h:X\to Y$, and denote the morphism $Y\to Z$ by $g$,
\item $-K_X$ is big over $Y$,
\item $\mu: Y\dashrightarrow Y'/Z$ is a birational map which does not extract any divisor, and denote the morphism $Y'\to Z$ by $g'$, and
\item $(Y',g'^*H)$ is log bounded.
\end{itemize}
Then there exists a $\bQ$-factorial projective variety $X'$ and a contraction $f':X'\to Z$ such that 
\begin{enumerate}
    \item $\nu:X\dashrightarrow X'/Z$ is an isomorphism in codimension one,
    \item $(X',B')$ is $\epsilon$-lc, where $B'=\nu_*B$,
    \item $f':X'\to Z$ factors through $h':X'\to Y'$, where $-K_{X'}$ is big over $Y'$, and
    \item $(X',\Delta'+f'^*H)$ is log bounded, where $\Delta'=\nu_*\Delta$.
\end{enumerate}
\[\xymatrix{
X\ar@{-->}[rr]^\nu \ar[d]_h & & X'\ar[d]^{h'}\\
Y\ar@{-->}[rr]^\mu \ar[dr]_g & & Y'\ar[dl]^{g'}\\
& Z &}\]
\end{lem}

\begin{proof}
Since $K_X+B\sim_{\bR,Z}0$, it follows that $K_X+B\sim_{\bR,Y}0$. By
\cite[Proposition 3.6]{birkarBoundednessEllipticCalabiYau2024}, after replacing $Y$ by a $\bQ$-factorial model and replacing $X$ accordingly, we may assume that $Y$ is $\bQ$-factorial.

Let $A'$ be an ample divisor on $Y'$. As $\mu: Y\dashrightarrow Y'$ does not extract any divisor, we can define $A$ to be the pullback of $A'$ on $Y$. Note that $h^*A$ is abundant over $Z$ because the Kodaira and numerical dimension are invariant by pullback under contractions, see \cite[Chapter 5, Proposition 2.7(4)]{nakayamaZariski-decompositionandAbundance}. By \cite[Lemma 2.13]{hashizumeMMPforAbundantLcPairs2020}, for $0<\epsilon\ll1$, $(X,B+\epsilon h^*A)$ admits a good minimal model $X''$ over $Z$. Since $Y'$ is the ample model of $(X,B+\epsilon h^*A)$ over $Z$, there is a contraction $h'':X''\to Y'$.

Let $B''$ denote the strict transform of $B$ on $X''$. Then $(X'',B'')$ remains an $\epsilon$-lc pair. Let $\{E_1,\cdots,E_k\}$ be the divisors contracted by the birational map $X\dashrightarrow X''$. The log discrepancy of any $E_i$ with respect to $(X'',B'')$ is at most 1. Therefore, by \cite[Corollary 1.4.3]{birkarExistenceMinimalModels2010}, there exists a birational model $X'\to X$ on which the only extracted divisors are $\{E_1,\cdots,E_k\}$. Consequently, we obtain a birational map
\[
\nu:X \dashrightarrow X'/Z
\] 
which is an isomorphism in codimension one, together with a contraction \[h':X'\to Y'.\] 
Let $f':X'\to Z$ denote the induced morphism $X'\to Y'\to Z$. 
Define \[K_{X'}+B' := \nu_*(K_X+B).\] Since $K_X+B\sim_\bR f^*N$, we have 
\[
K_{X'}+B'\sim_\bR f'^*N,
\] 
and $(X',B')$ is also $\epsilon$-lc. 

As $(Y',g'^*H)$ is log bounded, there exists a constant $r'\in \bR^{>0}$ and a very ample divisor $H_{Y'}$ on $Y'$ such that 
\[
H_{Y'}^{\dim Y'}\leq r', \quad \text{and} \quad H_{Y'}-g'^*H \text{ is ample.}
\] 
In particular, $H_{Y'}-g'^*N$ is ample. 
Note that $-K_{X'}$ is big over $Y'$ because $-K_X$ is big over $Y$, and $\nu:X \dashrightarrow X'$ is isomorphic in codimension one. Therefore, 
\[
h':(X',B')\to (Y',H_{Y'})
\] 
is a $(d,r',\epsilon)$-Fano type fibration. By Theorem \ref{FTF}, $(X',\Delta'+f'^*H)$ is log bounded.
\end{proof}

\begin{remark}\label{birational map which does not extract any divisor}
    Applying Lemma \ref{induction method} in the special case where $X=Y$, we obtain the following. Suppose that
\[
X \dashrightarrow X''/Z
\] 
is a birational map which does not extract any divisor and $(X'',f''^*H)$ is log bounded, where $f'': X''\to Z$. Then there exists a $\bQ$-factorial variety $X'$ that is isomorphic to $X$ in codimension one over $Z$, and 
\[
(X',\Delta'+f'^*H)
\] 
is log bounded, where $f':X'\to Z$.
\end{remark}

For the polarized log Calabi--Yau fibration $((X,B),A)\to (Z,H)$, if the horizontal part $B^\h\neq 0$, we can decompose it into a Fano type fibration and a lower-dimensional polarized log Calabi--Yau fibration.

\begin{prop}\label{nonzero horizontal part}
Assume that Theorem \ref{mainthm1} holds in dimension $\leq d-1$. Moreover, assume it also holds when $X$ is of dimension $d$ and $B$ is vertical over $Z$.
Then Theorem \ref{mainthm1} holds in dimension $d$.
\end{prop}

\begin{proof}
The argument follows closely that of \cite[Theorem 10.1] {birkarSingularitiesFanoFibrations2023}.  
\begin{enumerate} [label=\textsl{Step} \arabic{enumi}., wide=13pt, itemsep=13pt]
\item By assumption it suffices to consider the case where the horizontal part 
$B^\h$ of $B$ is non-zero. Then $K_X$ is not pseudo-effective over $Z$ because $K_X+B\sim_{\bR,Z} 0$. Let $t$ be the smallest number such that $K_X+tA$ is pseudo-effective over $Z$. By \cite[Lemma 4.11]{birkarGeometryPolarisedVarieties2023}, $t$ is bounded from above. Moreover, if we consider $(X,0,\overline{tA})$ as a generalized pair with nef part $\overline{tA}$, it follows from \cite[Lemma 4.4]{birkarEffectivityIitakaFibrations2016} that $K_X+tA$ admits a good minimal model over $Z$, which is not of general type. Considering intersection numbers shows that $t$ belongs to a fixed set of rational
number which is discrete away from zero, see \cite[Lemma 4.11]{birkarGeometryPolarisedVarieties2023}.

\item In this step we reduce to the case when $X$ is $\bQ$-factorial and $t \geq 1$.

Let $l$ be the largest integer such that $\widetilde{A} = lK_X + A$ is big over $Z$. Then
\[
    \vol(\widetilde{A}_F) = \vol(-lB_F + A_F) \leq \vol(A_F) \leq v,
\]
where $F$ is the general fiber of $f \colon X \to Z$. 
Let $f_1 \colon X_1 \to Z$ be the ample model of $\widetilde{A}$ over $Z$, and let $B_1, A_1$ be the pushdowns of $B, \widetilde{A}$ on $X_1$. 
If the horizontal part $B_1^\h = 0$, then $(X_1, f_1^*H)$ is log bounded in codimension one by assumption. 
By Remark \ref{birational map which does not extract any divisor}, this implies that $(X, \Delta+f^*H)$ is also log bounded in codimension one. 
Therefore, we may assume $B_1^\h \neq 0$. 

Repeating this procedure, we obtain a sequence of birational maps over $Z$:
\[
    ((X,B),A )\dasharrow ((X_1,B_1),A_1 )\dasharrow \cdots \dasharrow ((X_k,B_k),A_k )\dasharrow \cdots
\]
satisfying $B_i^\h \neq 0$ for all $i$. 
Since the Picard number $\rho(X)$ is finite, there exists $k \in \bN$ such that $X_i \dasharrow X_{i+1}$ is isomorphic in codimension one for all $i \geq k$. 
By the definition of $\widetilde{A}_k$, we know $K_{X_k} + \widetilde{A}_k$ is not big over $Z$, hence $K_{X_{k+1}} + A_{k+1}$ is not big over $Z$. 
Denote by $t_{k+1}$ the smallest number such that $K_{X_{k+1}} + t_{k+1}A_{k+1}$ is pseudo-effective over $Z$. 
Then $t_{k+1} \geq 1$. 
By Remark \ref{birational map which does not extract any divisor}, to prove that $(X,\Delta+f^*H)$ is log bounded in codimension one, it suffices to show that $(X_{k+1}, f_{k+1}^*H)$ is log bounded in codimension one, where $f_{k+1} \colon X_{k+1} \to Z$. 
Thus we may replace $((X,B),A)$ with $((X_{k+1},B_{k+1}),A_{k+1})$ and assume that $X$ is $\bQ$-factorial and $t \geq 1$.

\item Since $t$ belongs to a fixed set of rational numbers which is discrete away from zero, and since $t \geq 1$ and $t$ is bounded above, there are only finitely many possibilities for $t$. 
In the following we assume that $t$ is fixed. 

View $(X,\N)$ as a generalized pair over $Z$ with nef part $\N:=\overline{tA}$. 
Then $(X,\N)$ is an $\epsilon$-lc generalized pair. 
By \cite[Lemma 4.4]{birkarEffectivityIitakaFibrations2016}, there exists a good minimal model $f' \colon X' \to Z$ of $K_X+tA$ over $Z$. 
Let $B', \Delta', A'$ be the pushdowns of $B, \Delta, A$ on $X'$. 
By Remark \ref{birational map which does not extract any divisor}, it suffices to prove that $(X', f'^*H)$ is log bounded in codimension one. 

Let $h \colon X' \to Y/Z$ be the non-birational contraction induced by $K_{X'}+tA'$, and denote the morphism $Y \to Z$ by $g$. 
By \cite{filipazziGeneralizedCanonicalBundle2020}, there exists a generalized canonical bundle formula
\[
    K_{X'} + tA' \sim_\bQ h^*(K_Y + C_Y + \R_Y).
\]
Since $A'$ is big over $Z$, it follows that $-K_{X'}$ is big over $Y$. 
Hence, by \cite[Theorem 8.3]{birkarSingularitiesFanoFibrations2023}, $(Y,C_Y,\R)$ is a $\tau$-lc generalized pair for some fixed $\tau \in \bR^{>0}$. 

As $t$ is fixed, there exists $p \in \bN$ such that $p(K_{X'}+tA')$ is integral. 
Let $G$ be the general fiber of $h \colon X' \to Y$. 
Then $G$ is $\epsilon$-lc and belongs to a bounded family by \cite{birkarSingularitiesLinearSystems2021}. 
After replacing $p$ by a bounded multiple, we may assume that $p(K_G+tA_G)$ is Cartier by \cite[Theorem 1.10]{hanACCLocalVolumes2023}. 
Since $G$ is of Fano type, we have $\Pic^0(G)=0$, and thus $p(K_G+tA_G) \sim 0$. 
Therefore, there exists a rational function $\alpha$ on $X'$ such that $p(K_{X'}+tA')+\Div(\alpha)$ is vertical over $Y$. 
Since
\[
    p(K_{X'}+tA') + \Div(\alpha) \sim_{\bQ,Y} 0,
\]
it follows from \cite[Lemma 2.5]{chenEffectiveLogIitaka2024} that $p(K_{X'}+tA')+\Div(\alpha)$ is the pullback of a $\bQ$-Cartier $\bQ$-divisor on $Y$. 
Hence we obtain the canonical bundle formula
\[
    p(K_{X'}+tA') \sim ph^*(K_Y+C_Y+\R_Y).
\]
Since $p(K_{X'}+tA')$ is integral and the multiplicities of the fibers of $h$ over codimension one points are bounded, 
after replacing $p$ by a bounded multiple we may assume that
\[
    J := p(K_Y+C_Y+\R_Y)
\]
is an integral divisor.

\item In this step we prove that the volume of the restriction of $J$ on the general fiber of $g:Y\to Z$ is bounded from above.

Let $\phi:W\to X$ and $\psi:W\to X'$ be common resolutions. Pick a general point of $Z$ and let $F_W,F_X,F_{X'},F_Y$ be the corresponding fiber over this point. By \cite[Theorem 1.1]{birkarGeometryPolarisedVarieties2023}, there exists a fixed positive integer $m$ such that $|mA|_{F_X}|$ defines a birational map. Let $c=\dim F_W$ and $e=\dim F_Y$. Then
\begin{align*}
    \vol(J|_{F_Y})&\leq  (\phi^*(mA)|_{F_W})^{c-e}\cdot(\psi^*(p(K_{X'}+tA'))|_{F_W})^e\\
    &\leq m^{c-e}p^e\vol(\phi^*A|_{F_W}+\psi^*(K_{X'}+tA')|_{F_W})\\
    &\leq m^{c-e}p^e\vol(\phi^*(K_X+(1+t)A)|_{F_W})\\
    &\leq m^{c-e}p^e\vol((1+t)A|_{F_X})\leq (1+t)^cm^{c-e}p^ev.
\end{align*}

\item Applying \cite[Theorem 1.1]{birkarGeometryPolarisedVarieties2023} to a $\bQ$-factorialization of $F_Y$ and the divisor $J$, we conclude that there exists a fixed positive integer $n$ such that $|nJ|_{F_Y}|$ defines a birational map. Therefore, there exists an effective integral divisor $J'$ such that $J'\sim nJ/Z$, and then 
\[
\vol(J'|_{F_Y})\leq v',
\]
where $v'=(1+t)^cm^{c-e}n^ep^ev$.

By \cite[Lemma 2.11]{zhuBoundednessStableMinimal2025} (see also \cite[Theorem 1.9]{birkarSingularitiesFanoFibrations2023}), there exists a fixed $\tau\in \bR^{>0}$ such that we can write a canonical bundle formula
\[
    K_{X'}+B' \sim_{\bR} h^*(K_Y+D_Y+\mathbf{S}_Y),
\]
where $(Y,D_Y,\mathbf{S})$ is a $\tau$-lc generalized pair. By \cite[Theorem 4.1]{ambroModuliBdivisorLctrivial2005}, we may choose a boundary $\widetilde{D}_Y$ such that 
\[
K_Y+\widetilde{D}_Y \sim_{\bR} K_Y+D_Y+\mathbf{S}_Y
\]
and $(Y,\widetilde{D}_Y)$ is $\tfrac{\tau}{2}$-lc. Consequently, 
\[
g:((Y,\widetilde{D}_Y),J') \to (Z,H)
\] 
is a $(\dim Y,v',r,\tfrac{\tau}{2})$-polarized log Calabi--Yau fibration. By the induction hypothesis, $(Y,g^*H)$ is log bounded in codimension one. Hence, by Lemma \ref{induction method}, $(X',f'^*H)$ is log bounded in codimension one. Finally, by Remark \ref{birational map which does not extract any divisor}, the same holds for $(X,\Delta+f^*H)$.

\end{enumerate}
\end{proof}

If the horizontal part $B^\h$ vanishes, we can run an MMP for very exceptional divisors to reduce the problem to Theorem \ref{thm:bdd of weak pcy fibration} with $\Phi=\{0\}$.
\begin{proof}[Proof of Theorem \ref{mainthm1}]
By Proposition \ref{nonzero horizontal part}, it suffices to treat the case where $B$ is vertical over $Z$.

By \cite[Lemma 2.11]{zhuBoundednessStableMinimal2025} (see also \cite[Theorem 1.9]{birkarSingularitiesFanoFibrations2023}), there exists a fixed $\delta\in \bR^{>0}$ such that we can write the canonical bundle formula \[K_X+B\sim_\bR f^*(K_Z+B_Z+\M_Z),\] where $(Z,B_Z,\M)$ is a $\delta$-lc generalized pair. 
By \cite[Theorem 2.3]{birkarBoundednessFanoType2024}, there exists a $\bQ$-factorialization 
\[\mu \colon Z' \to Z\] such that $Z'$ belongs to a bounded family. 
Let $H'$ be a very ample divisor on $Z'$ satisfying $H'^{\dim Z'} \leq r'$ for some fixed $r' \in \bR^{>0}$, 
and such that $H' - \mu^*H$ is ample. 

By \cite[Proposition 3.6]{birkarBoundednessEllipticCalabiYau2024}, there exists a $\bQ$-factorial $\epsilon$-lc pair 
$(X',B')$, isomorphic to $(X,B)$ in codimension one, together with a contraction \[f' \colon X' \to Z'\] satisfying $K_{X'}+B'\sim_{\bR,Z'}0$. 
Let $A'$ be the strict transform of $A$ in $X'$. Then 
\[
\vol(A'|_{F'}) = \vol(A|_F) \leq v,
\] 
where $F, F'$ denote the general fibers of $f \colon X \to Z$ and $f' \colon X' \to Z'$, respectively.

Suppose that $f' \colon X' \to Z'$ admits a very exceptional divisor $E$. We 
run an MMP for $(X', B' + \lambda E)$ over $Z'$, where $\lambda$ is a sufficiently small positive number. 
By \cite[Theorem 1.8]{birkarExistenceLogCanonical2012}, this MMP terminates with a model $X''$ on which $E$ is contracted. 
If $X'' \to Z'$ still has a very exceptional divisor, we repeat this process. 
Since the relative Picard number $\rho(X'/Z')$ strictly decreases each time, after finitely many steps we reach a contraction \[g \colon Y \to Z'\] with no very exceptional divisor. 
Let $B_Y$ and $A_Y$ be the pushdowns of $B'$ and $A'$ to $Y$. 
Then $K_Y + B_Y \sim_{\bR,Z'}0$ and $(Y,B_Y)$ is $\epsilon$-lc. 
Moreover, since $X' \dashrightarrow Y$ is an isomorphism over an open subset of $Z'$, we have 
\[
\vol(A_Y|_{F_Y}) = \vol(A'|_{F'}) \leq v,
\]
where $F_Y$ is the general fiber of $g \colon Y \to Z'$. 

Let $Y'$ be the ample model of $A_Y$ over $Z'$, and denote by $B'_Y$ and $A'_Y$ the pushdowns of $B_Y$ and $A_Y$ to $Y'$. 
Then $K_{Y'} + B'_Y \sim_{\bR,Z'} 0$ and $(Y',B'_Y)$ is $\epsilon$-lc. 
Furthermore,
\[
\vol(A'_Y|_{F'_Y}) = \vol(A_Y|_{F_Y}) \leq v,
\] 
where $F'_Y$ is the general fiber of $g' \colon Y' \to Z'$. 
Therefore, \[g': ((Y',B'_Y), A'_Y)\to Z'\] is a $(d,v,r',\epsilon)$-polarized log Calabi--Yau fibration.

Note that $g' \colon Y' \to Z'$ has no very exceptional divisor. 
Since $B'_Y$ is vertical over $Z'$ and $Z'$ is $\bQ$-factorial, it follows that $B'_Y$ is of fiber type over $Z'$. 
Hence, there exists an effective $\bR$-Cartier $\bR$-divisor $C'$ on $Z'$ such that \[B'_Y = g'^*C'.\]
Consequently, 
\[
K_{Y'} \sim_{\bR} g'^*(\mu^*N - C').
\] 
Note that $H' - (\mu^*N - C')$ is in general pseudo-effective, rather than ample. 
Hence, \[g' \colon ((Y',0), A'_Y) \to Z'\] is merely a weak $(d,0,v,r',\epsilon)$-polarized log Calabi--Yau fibration. 

By Theorem \ref{thm: finite coeff bdd in codim 1}, $(Y', g'^*H')$ is log bounded in codimension one. 
Since $H' - \mu^*H$ is ample, it follows that $(Y', \widetilde{g}'^*H)$ is also log bounded in codimension one, 
where $\widetilde{g}' \colon Y' \to Z$. 
Finally, as $X \dashrightarrow Y'$ is a birational map which does not extract any divisor, 
Remark \ref{birational map which does not extract any divisor} implies that $(X, \Delta + f^*H)$ is log bounded in codimension one.
\end{proof}

\section{Fibrations whose general fibers have vanishing irregularity}\label{sec:Fibrations whose general fibers have vanishing irregularity}
In this section, we consider the boundedness of polarized log Calabi--Yau fibrations $f : ((X, B), A) \to (Z, H)$ such that $\Supp R^1f_*\cO_X\subsetneq Z$.

The following lemma addresses the issue when the base of a polarized log Calabi--Yau fibration is not $\mathbb{Q}$-factorial.

\begin{lem}\label{bounded Q-factorialization} 
    Let $d,r\in \bN$ and $\epsilon\in \bR^{>0}$. Assume that
    \begin{itemize}
        \item $(X,B,\M)$ is an $\epsilon$-lc generalized projective pair of dimension $d$, 
        \item $H$ is a very ample divisor such that $H^d\leq r$, and 
        \item $H-(K_X+B+\M_X)$ is ample.
    \end{itemize}
    Then there exists a positive integer $r'$ depending only on $d,r,\epsilon$ such that 
    \begin{enumerate}
        \item there exists a couple $(X',\Sigma')$ such that $\pi:X'\to X$ is a $\bQ$-factorialization,
        \item the irreducible components of $\Sigma'$ generate $N^1(X'/X)$,
        \item $H'$ is a very ample divisor on $X'$ such that $H'^d\leq r'$, and
        \item $H'-\Sigma'$ and $H'-\pi^*H$ are ample.
    \end{enumerate}
\end{lem}

\begin{proof}

By \cite[Theorem 2.3]{birkarBoundednessFanoType2024}, there exists a $\bQ$-factorialization \[\pi: X'\to X\] such that $X'$ is in a bounded family. 
Therefore, there exists a bounded resolution $W$ of $X'$ such that the Picard number $\rho(X')\leq\rho(W)$ is bounded from above, and hence the relative Picard number $\rho(X'/X)$ is also bounded from above. In the following we apply induction on $\rho(X'/X)$ to construct $\Sigma'$ and $H'$ on $X'$ satisfying the desired properties. 

By the cone theorem \cite[Theorem 3.7]{kollarBirationalGeometryAlgebraic1998}, the birational morphism $\pi: X'\to X$ can be decomposed into a sequence of extremal contractions \[X'=X_1\to X_2\to \cdots \to X_{l-1}\to X_l=X.\] Again by \cite[Theorem 2.3]{birkarBoundednessFanoType2024}, $X_{l-1}$ also belongs to a bounded family. Hence, there exists a positive integer $r_{l-1}$ depending only on $d,r,\epsilon$ and a very ample divisor $H_{l-1}$ on $X_{l-1}$ such that \[H_{l-1}^d\leq r_{l-1} \quad\text{and}\quad H_{l-1}-\mu^*H\enspace \text{is ample,}\] where $\mu:X_{l-1}\to X$. Writing \[K_{X_{l-1}}+B_{l-1}+\M_{X_{l-1}}=\mu^*(K_X+B+\M_X),\] it follows that \[H_{l-1}-(K_{X_{l-1}}+B_{l-1}+\M_{X_{l-1}})\] is ample. Since $\rho(X'/X_{l-1})<\rho(X'/X)$, by the inductive hypothesis we obtain the following:
\begin{itemize}
    \item there exists a couple $(X',\Sigma')$ such that the irreducible components of $\Sigma'$ generate $N^1(X'/X_{l-1})$,
    \item there exists a fixed $r'\in \bN$ and a very ample divisor $H'$ on $X'$ such that $H'^d\leq r'$, and
    \item both $H'-\Sigma'$ and $H'-\nu^*H_{l-1}$ are ample, where $\nu:X'\to X_{l-1}$. 
\end{itemize}
In particular, $H'-\pi^*H$ is also ample.

Finally, since $H_{l-1}$ is ample over $X$ and $\mu:X_{l-1}\to X$ is an extremal contraction, by replacing $\Sigma'$ with $\Sigma'\cup \Supp(\nu^*H_{l-1})$, $H'$ with $2H'$, and $r'$ with $2^dr'$, we conclude that the irreducible components of $\Sigma'$ generate $N^1(X'/X)$.
\end{proof}

The following lemma bounds certain vertical divisors in a log bounded family.
\begin{lem}\label{approximate unbounded polarization via bounded polarization}
    Let $\epsilon,\delta\in \bR^{>0}$ and $\Phi\subset [0,1]\cap\bQ$ be a finite set. Assume that
    \begin{itemize}
        \item $(X,B)$ is a projective $\mathbb{Q}$-factorial $\epsilon$-lc pair belonging to a bounded family,
        \item the coefficients of $B$ are in $\Phi$,
        \item $K_X+B$ is semi-ample and defines a contraction $f:X\rightarrow Z$,
        \item there is a canonical bundle formula $K_X+B\sim_\bQ f^*(K_Z+B_Z+\M_Z)$ such that $(Z,B_Z,\M)$ is a generalized $\delta$-lc pair, and
        \item $N$ is an integral divisor on $X$ such that $N\sim_{\mathbb{Q},\eta_Z}0$, where $\eta_Z$ is the generic point of $Z$.
    \end{itemize}
    Then there exists an effective $\bQ$-divisor $D$ on $X$ such that
    \begin{enumerate}
        \item $D$ is vertical over $Z$,
        \item $N\sim_{\mathbb{Q},Z} D$, and
        \item $(X,\Supp(B)\cup\Supp(D))$ is log bounded.
    \end{enumerate}
    \begin{proof}
    \begin{enumerate} [label=\textsl{Step} \arabic{enumi}., wide=13pt, itemsep=13pt]
    \item Since $(X,B)$ belongs to a log bounded family of $\epsilon$-lc pairs with coefficients of $B$ contained in a finite set,
it follows from \cite[Lemma 2.17]{birkarBoundednessEllipticCalabiYau2024} that the set of morphisms $f:(X,B)\to Z$ is bounded. Moreover, there exists a bounded $m\in \bN$ such that $m(K_X+B)$ is base point free and satisfies \[m(K_X+B)\sim mf^*(K_Z+B_Z+\M_Z).\]
Set \[H:=m(K_Z+B_Z+\M_Z).\] Then $H$ is a very ample divisor on $Z$, and $H^{\dim Z}$ is bounded from above.

Applying Lemma \ref{bounded Q-factorialization}, we obtain that
    \begin{itemize}
        \item there exists a couple $(Z',\Sigma')$ such that $\pi:Z'\to Z$ is a $\bQ$-factorialization,
        \item the irreducible components of $\Sigma'$ generate $N^1(Z'/Z)$,
        \item $H'$ is a very ample divisor on $Z'$ such that $H'^{\dim Z'}$ is bounded from above, and
        \item both $H'-\Sigma'$ and $H'-\pi^*H$ are ample.
    \end{itemize}
Therefore, by replacing $m$ with a bounded multiple, we may assume that $H-\Sigma$ is pseudo-effective, where $\Sigma=\pi_*\Sigma'$.

By \cite[Proposition 3.6]{birkarBoundednessEllipticCalabiYau2024}, there exists a commutative diagram
    \[\xymatrix{
      X'\ar@{-->}[r]^\mu \ar[d]_{f'} & X\ar[d]^{f}\\
      Z'\ar[r]^\pi  & Z}\]
      such that $\mu:X'\dashrightarrow X$ is an isomorphism in codimension one.

\item Let $N' :=\mu^* N$. Since $N \sim_{\bQ,\eta_Z} 0$, we have
\[
N' \sim_{\bQ,\eta_{Z'}} 0,
\] 
where $\eta_{Z'}$ is the generic point of $Z'$. 
As $Z'$ is $\bQ$-factorial, there exists a very exceptional/$Z'$ $\bQ$-divisor $L'$ on $X'$ such that 
\[
N' \sim_{\bQ, Z'} L'.
\] 
Since the irreducible components of $\Sigma'$ generate $N^1(Z'/Z)$, there exists a $\bQ$-Cartier $\bQ$-divisor $C'$ on $Z'$ with 
\[
\Supp(C') \subseteq \Supp(\Sigma') \quad \text{and} \quad N' \sim_{\bQ,Z} L' + f'^* C'.
\] 
Let $L := \mu_* L'$. Then $L$ is very exceptional over $Z$, because $L'$ is very exceptional over $Z'$, and $\pi, \mu$ are isomorphisms in codimension one. Hence, we have
\[
N \sim_{\bQ,Z} L + \mu_* f'^* C'.
\] 
Since $f: X \to Z$ is a bounded morphism, by \cite[Lemma 2.20]{birkarBoundednessEllipticCalabiYau2024}, we conclude that $\Supp(L)$ is bounded.

Possibly enlarging $\Sigma$ and replacing $H',H$ by bounded multiples, we may assume that there exists an effective $\bQ$-Cartier $\bQ$-divisor $T$ on $Z$ such that 
\[
\Supp(T) \subset \Sigma \quad \text{and} \quad C' + \pi^* T, \; L + f^* T \text{ are effective}.
\] 
Replacing $C'$ and $L$ by $C' + \pi^* T$ and $L + f^* T$, we may assume that $C'$ and $L$ are effective $\bQ$-divisors. 

Since $H - \Sigma$ is pseudo-effective, it follows that $m(K_X + B) - \mu_* f'^* \Sigma'$ is pseudo-effective. Therefore, the couple
\[
(X, \Supp(B) \cup \Supp(L) \cup \Supp(\mu_* f'^* C'))
\] 
is log bounded.

Finally, set
\[
D := L+\mu_*f'^*C'.
\] 
Then $D$ is effective, vertical over $Z$, satisfies $N\sim_{\bQ,Z}D$, and $(X,\Supp(B)\cup\Supp(D))$ is log bounded. This completes the proof.
    \end{enumerate}
    \end{proof}
\end{lem}

We also need the following result on the finiteness of log canonical models when the boundary divisors vary in a polytope.
\begin{lem}\label{finiteness of canonical models}
    Assume that
    \begin{enumerate}
        \item $(X,B)$ is a projective $\bQ$-factorial klt pair,
        \item $X\to Z$ is a contraction to a normal projective variety,
        \item $K_X+B\sim_{\bQ,Z} 0$,
        \item $L$ is an effective $\bQ$-divisor on $X$ which is big over $Z$,
        \item $D_1,\cdots,D_k$ are effective $\bQ$-divisors on $X$ that are vertical over $Z$, and
        \item $V$ is the affine subspace generated by $L$ and all $D_i$ in the real vector space of divisors, and $P$ is the polytope in $V$ generated by $L$ and all $D_i$.
    \end{enumerate}
    Then there exist finitely many rational maps $\pi_j:X\dashrightarrow Y_j/Z$ for $1\leq j\leq l$ satisfying the following.
    
    For each point $C\in P$, there exists $1\leq j\leq l$ such that $\pi_j$ gives the ample model of $C$ over $Z$.

\begin{proof}

Let $\delta$ be a positive rational number such that $(X, B + \delta (L + \sum_{i=1}^k D_i))$ is klt. 
Let $V'$ be the affine subspace generated by $B + \delta L$ and $B + \delta D_i$ in the real vector space of divisors, 
and let $P'$ be the polytope in $V'$ generated by $B + \delta L$ and $B + \delta D_i$. 
Since $K_X + B \sim_{\bQ,Z} 0$, it suffices to prove the finiteness of ample models of $K_X + C'$ over $Z$ for all $C' \in P'$. 
By \cite[Theorem 5.1]{mengMMPforlocallystablefamilies} (see also \cite[Theorem 2.10.3]{Kawamatabook}), 
it is enough to show that $(X, C')$ has a good minimal model over $Z$ for every $C' \in P'$.

Let $C' \in P'$ and write 
\[
C' = B + a_0 L + \textstyle\sum a_i D_i, \quad \text{with} \ \textstyle\sum a_i = \delta.
\]
If $a_0 > 0$, then $K_X + C'$ is big over $Z$, and by \cite{birkarExistenceMinimalModels2010}, $(X, C')$ has a good minimal model over $Z$. 
If $a_0 = 0$, then since each $D_i$ is vertical over $Z$, we have $K_X + C' \sim_{\bQ,\eta_Z} 0 $, where $\eta_Z$ is the generic point of $Z$. 
By \cite[Theorem 1.4]{birkarExistenceLogCanonical2012} or \cite[Theorem 1.1]{haconExistenceLogCanonical2013}, $(X, C')$ has a good minimal model over $Z$. This completes the proof.
\end{proof}
\end{lem}

With the necessary preparations complete, we can now prove the main theorem of this section.

\begin{proof}[Proof of Theorem \ref{mainthm2}]
    By Theorem \ref{thm:bdd of weak pcy fibration}, there exists a couple $(V,\Theta)$, an effective integral divisor $J$ on $V$ and a positive integer $l$, depending only on $d,\Phi,v,r,\epsilon$, such that 
    \begin{itemize}
     \item there is a contraction $h:V\to Z$ and $V$ is $\bQ$-factorial,
     \item $V\dashrightarrow X/Z$ is an isomorphism in codimension one,
     \item $(V,\Theta+\Supp(J))$ is bounded,
     \item $\Theta$ contains both $B_V$ and $h^*H_Z$, where $B_V$ is the strict transform of $B$ on $V$, and $H_Z$ is a general element of $|6dH|$, and
     \item $J_X \equiv lA$ over the generic point of $Z$, where $J_X$ is the strict transform of $J$ on $X$.
 \end{itemize}

 Since $\Supp R^1f_*\cO_X\subsetneq Z$, Grauert's theorem implies $h^1(X_g,\cO_{X_g})=0$, where $X_g$ is the general fiber of $f:X\to Z$. Consequently, \[J_X\sim_{\bQ,\eta_Z}lA,\] where $\eta_Z$ is the generic point of $Z$. Since $V\dashrightarrow X/Z$ is an isomorphism in codimension one, we have \[J\sim_{\bQ,\eta_Z}lA_V,\] where $A_V$ is the strict transform of $A$ on $V$. By Lemma \ref{approximate unbounded polarization via bounded polarization}, there exists a log bounded pair $(V,B_V+J+\sum D_i)$ and rational numbers $a_i\geq 0$ for $1\leq i\leq k$ such that \[J+\textstyle\sum a_iD_i\sim_{\bQ,Z} lA_V.\]

 By log boundedness, we may assume there exists a family $(\cV,\cB_\cV)\rightarrow \cS$ together with divisors $\cJ$ and $\cD_i$ on $\cV$ such that there exists a point $s \in \cS$ with 
\[
(V,B_V) \simeq (\cV_s,\cB_{\cV_s}), \quad \cJ_s \simeq J, \quad \text{and} \quad \cD_{i,s} \simeq D_i.
\] 
By \cite[Proposition 2.4]{HaconlogCalabiYaupairsofFanotype}, after passing to a stratification of $\cS$, we may assume that $K_{\cV}+\cB_\cV$ is $\mathbb{Q}$-Cartier and klt. By \cite[Lemma 2.8]{hanBirationalBoundednessRationally2022}, there is a fibration $g:\cV \to \cZ$ over $\cS$ such that $\cV_s \to \cZ_s$ is isomorphic to $V \to Z$. Moreover, by Remark \ref{big divisor in family}, we may assume that $\cJ$ is big over $\cZ$. Since each $D_i$ is vertical over $Z$, $\cD_i$ is vertical over $\cZ$ as well. 

Let $\cH$ be a Cartier divisor on $\cZ$ which is ample over $\cS$, and let $\cG \in |6n \cH|$ be a general member, where $n = \dim \cX$. By the boundedness of the length of extremal rays, the log canonical model of 
\[
(\cV, \cB_\cV + \mu ( \cJ + \textstyle\sum a_i \cD_i))
\] 
over $\cZ$ coincides with the log canonical model of 
\[
(\cV, \cB_\cV + \mu (\cJ + \textstyle\sum a_i \cD_i) + \tfrac{1}{2} \cG)
\] 
over $\cS$, for $\mu>0$ sufficiently small. After passing to a stratification and applying \cite[Theorem 1.2]{haconBoundednessModuliVarieties2018} to a fiberwise log resolution of 
\[
(\cV, \cB_\cV + \mu (\cJ + \textstyle\sum a_i\cD_i) + \tfrac{1}{2} \cG)
\] 
over $\cS$, we may assume that it admits a relative log canonical model over $\cS$ which induces log canonical models fiberwise. 

By Lemma \ref{finiteness of canonical models}, there exist finitely many rational maps 
\[
\cV \dashrightarrow \cY_j / \cZ
\] 
such that for every $(a_1, a_2, \dots, a_k)$, there exists $j$ for which $\cY_j$ is the log canonical model of 
\[
(\cV, \cB_\cV + \mu (\cJ + \textstyle\sum a_i \cD_i))
\] 
over $\cZ$. Then $\cY_{j,s}$ is the log canonical model of 
\[
(V, B_V + \mu (J + \textstyle\sum a_i D_i))
\] 
over $Z$, hence also of $(V, B_V + \mu l A_V)$ over $Z$. Since $(X, B + \mu l A)$ is the log canonical model of $(V, B_V + \mu l A_V)$ over $Z$, it follows that $\cY_{j,s} \simeq X$. Therefore, we conclude that $(X, B + f^* H)$ is log bounded.
\end{proof}

\section{Stable minimal models and fibered Calabi--Yau varieties}\label{sec:Stable minimal models and fibered Calabi--Yau varieties}
In this section, we apply our boundedness results on polarized log Calabi--Yau fibrations to stable minimal models and fibered Calabi--Yau varieties.
 
\begin{definition}[{\cite[Definition 1.1]{birkarBoundednessVolumeGeneralised2021}}]\label{familyofgpairs}
	Let $d\in \bN$,  $u\in \bQ^{>0}$, and $\Phi\subset \bQ^{\geq 0}$ be a DCC set. Let $\cF_{gklt}(d,\Phi,u)$ be the set of projective generalized pairs $(X,B,\M)$ such that 
		\begin{enumerate}
			\item $(X,B,\M)$ is klt of dimension $d$,
			\item the coefficients of $B$ are in $\Phi$,
			\item there is a birational morphism  $X'\to X$ such that $\M$ descends to $X'$ and $\M_{X'}=\sum \mu_i M_i'$ where $M_i'$ is nef Cartier and $\mu_i\in \Phi$ for any $i$,
			\item $K_X+B+\M_X$ is ample, and
            \item $\vol(K_X+B+\M_X)=u$.
		\end{enumerate}
\end{definition}

\begin{proof}[Proof of Corollary \ref{boundedness of klt stable minimal models}]
By Lemma \ref{lem: A>0}, we may assume that $A$ is an effective integral divisor and $\vol(A|_F)=v$ is fixed.
By \cite[Lemma 8.2]{birkarBoundednessVolumeGeneralised2021}, all log discrepancies of $(X,B)$ that are smaller than $1$ are in a fixed finite set $\Phi'$, depending only on $d,u,v,\Phi$. In particular, the coefficients of $B$ belong to $\Phi'$, and there exists a positive number $\epsilon$, depending only on $d,u,v,\Phi$, such that $(X,B)$ is $\epsilon$-lc. 

By \cite[Lemma 7.4]{birkarBoundednessVolumeGeneralised2021},
there exists a positive integer $p$, depending only on $d,u,\Phi$, such that we can write a canonical bundle formula \[K_X+B\sim_{\bQ} f^*(K_Z+B_Z+\M_Z),\] where $p\M_{Z'}$ is Cartier on some high resolution $Z'\to Z$. 

By \cite[Theorem 1.1]{haconACCLogCanonical2014}, the coefficients of $B_Z$ are in a fixed DCC set $\Psi$, depending only on $d,\Phi$. Replacing $\Psi$ by $\Psi \cup\{\frac{1}{p}\}$, we obtain \[(Z,B_Z,\M)\in \cF_{gklt}(d',\Psi,u),\] where $d'=\dim Z$. By \cite[Theorem 1.4]{birkarBoundednessVolumeGeneralised2021}, $(Z,B_Z,\M)$ belongs to a bounded family. Furthermore, by the remark following \cite[Theorem 4.3]{birkarModuliAlgebraicVarieties2022}, there exists a fixed positive integer $l$ such that \[H:=l(K_Z+B_Z+\M_Z)\] is very ample. Consequently, \[f:((X,B),A)\to (Z,H)\] is a $(d,\Phi',v,l^{d'}u,\epsilon)$-polarized log Calabi--Yau fibration. Therefore, the corollary follows from Theorem \ref{mainthm1} and Theorem \ref{mainthm2}. 
\end{proof}

\begin{proof}[Proof of Corollary \ref{cor: bdd of fibered cy}]
By \cite[Lemma 2.11]{zhuBoundednessStableMinimal2025} (see also \cite[Theorem 1.9]{birkarSingularitiesFanoFibrations2023}),  there exists a canonical bundle formula \[K_X+B\sim_{\bR}f^*(K_Z+B_Z+\M_Z)\] such that $(Z,B_Z,\M)$ is a $\delta$-lc generalized pair for some $\delta\in \bR^{>0}$ depending only on $d,\epsilon,v$. Since $Z$ is rationally connected, by \cite[Theorem 1.7]{birkarSingularitiesFanoFibrations2023}, there exists a projective variety $Z'$ such that
\begin{itemize}
    \item $Z'\dashrightarrow Z$ is an isomorphism in codimension one, and
    \item there is a fixed positive integer $r$ and a very ample divisor $H'$ on $Z'$ such that $H'^{\dim Z'}\leq r$.
\end{itemize}

By \cite[Propositions~3.6 and~3.7]{birkarBoundednessEllipticCalabiYau2024}, there exists an $\epsilon$-lc pair $(X',B')$ which is isomorphic in codimension one to $(X,B)$, and a contraction  $f':X'\to Z'$ with $K_{X'}+B'\sim_{\bR,Z'}0$. Let $A'$ be the strict transform of $A$ on $X'$. Then \[\vol(A'|_{F'})=\vol(A|_F)\leq v,\] where $F'$ is the general fiber of $f':X'\to Z'$. Let $X''$ be the ample model of $A'$ over $Z'$, and $B''$ and $A''$ be the pushdown of $B'$ and $A'$ on $X''$. Then we conclude that \[f'':((X'',B''),A'')\to Z'\] is a $(d,v,r,\epsilon)$-polarized log Calabi--Yau fibration. Therefore, Theorem \ref{mainthm1} implies that $X''$ is bounded in codimension one. Since $X\dashrightarrow X''$ does not extract any divisor, by \cite[Corollary 2.13]{birkarBoundednessEllipticCalabiYau2024}, $X$ is bounded in codimension one.
\end{proof}

\bibliography{pcyf_abbreviation}
\bibliographystyle{alphaurl}

\end{document}